\documentclass[11pt,a4paper]{amsart}

\usepackage[margin=3cm,top=2.5cm,bottom=2.5cm]{geometry}

\usepackage{amsmath, amssymb, amsthm, mathrsfs}
\usepackage{graphicx,enumitem,microtype,indentfirst} 
\usepackage[dvipsnames]{xcolor}

\usepackage[colorlinks=true,linkcolor=red!70!black,urlcolor=MidnightBlue,citecolor=MidnightBlue]{hyperref}

\usepackage[utf8]{inputenc} 
\usepackage[T1]{fontenc}
\usepackage{lmodern}
\usepackage{cancel}

\newcommand{\R}{\mathbb R}
\newcommand{\T}{\mathbb T}

\newcommand{\Sp}{\mathbf S}
\newcommand{\Z}{\mathbb Z}

\renewcommand{\P}{\mathbf P}
\renewcommand{\S}{\mathbb S}

\newcommand{\FF}{\mathcal F}

\newcommand{\la}{\langle}
\newcommand{\ra}{\rangle}

\renewcommand{\d}{\mathrm{d}}

\newcommand{\eps}{\varepsilon}

\newcommand{\Nt}{|\hskip-0.04cm|\hskip-0.04cm|}

\DeclareMathOperator{\re}{Re}
\DeclareMathOperator{\Ker}{Ker}

\numberwithin{equation}{section}

\setlist[enumerate]{label=\textnormal{(\arabic*)},itemsep=5pt,topsep=4pt,leftmargin=*}

\theoremstyle{plain}
\newtheorem{theo}{Theorem}
\newtheorem{prop}{Proposition}[section]
\newtheorem{lem}[prop]{Lemma}
\newtheorem{cor}[prop]{Corollary}
\theoremstyle{remark}
\newtheorem{rem}{Remark}
\theoremstyle{definition}

\def\le{\leqslant}
\def\ge{\geqslant}
\def\leq{\leqslant}
\def\geq{\geqslant}

\title[The Navier-Stokes limit of kinetic equations for low regularity data]{The Navier-Stokes limit of kinetic equations \\
for low regularity data}

\author[K. Carrapatoso]{Kleber Carrapatoso}
\address[K.~Carrapatoso]{Centre de Math\'ematiques Laurent Schwartz, \'Ecole polytechnique, Institut Polytechnique de Paris, 91128 Palaiseau cedex, France}
\email{kleber.carrapatoso@polytechnique.edu}

\author[I. Gallagher]{Isabelle Gallagher}
\address[I. Gallagher]{DMA, \'Ecole normale sup\'erieure, CNRS, PSL University, 75005 Paris, France and UFR de math\'ematiques, Universit\'e Paris Cit\'e, 75013 Paris, France} 
\email{isabelle.gallagher@ens.fr}

\author[I. Tristani]{Isabelle Tristani}
\address[I. Tristani]{Université Côte d’Azur, CNRS, LJAD, Parc Valrose, F-06108 Nice, France} 
\email{isabelle.tristani@univ-cotedazur.fr}

\date{\today}

\begin{document}

\begin{abstract}
In this paper, we investigate the link between kinetic equations (including Boltzmann with or without cutoff assumption and Landau equations) and the incompressible Navier-Stokes equation. We work with strong solutions and we treat all the cases in a unified framework. The main purpose of this work is to be as accurate as possible in terms of functional spaces. More precisely, it is well-known that the Navier-Stokes equation can be solved in a lower regularity setting (in the space variable) than kinetic equations. Our main result allows to get a rigorous link between solutions to the Navier-Stokes equation with such low regularity data and kinetic equations. 
\end{abstract}

\maketitle

\tableofcontents

\section{Introduction}
In this paper, we are interested in a problem in the theory of  hydrodynamical limits: our goal is to obtain a rigorous result of convergence of solutions to various kinetic equations towards solutions to the incompressible Navier-Stokes equation. This problem can be seen as a part of the program initiated by the 6th problem of Hilbert in 1900 at the International Congress of Mathematicians. Indeed, the question is to understand the link between microscopic and macroscopic descriptions of a fluid, and deriving macroscopic equations from mesoscopic ones can be seen as an intermediate step of this program. We refer for instance to the book by Saint-Raymond~\cite{SRbook} for a detailed presentation of the subject and for mathematical results in the field. More specifically, in this paper, we seek to get a result on the convergence  of sequences of strong solutions  to the rescaled mesoscopic equations in which the connection between the kinetic and the fluid equations is as accurate as possible in terms of functional spaces.

\subsection{Kinetic equations}

At the kinetic level, we shall consider Boltzmann or Landau type equations for not too soft potentials. We denote by $F=F(t,x,v)$ the density of particles, which depends on time~$t \in \R^+$, position $x \in \T^3$ (the unit periodic box) and velocity $v \in \R^3$. The dimensionless version of our kinetic equation reads
	\[
	{\rm St} \, \partial_t F + v \cdot \nabla_x F = \frac1{\rm Kn} Q(F,F)\,,
	\]
where the Strouhal number ${\rm St}$ and  the Knudsen number ${\rm Kn}$ are dimensionless parameters which are natural in   kinetic problems.   Here and below, $Q$ can be the Boltzmann (with or without cutoff) collision operator or the Landau collision operator. 
The Boltzmann collision operator is an integral operator defined as
	\begin{equation} \label{def:Bolop}
	Q_B(f_1,f_2):=\int_{\R^3 \times \Sp^2} B(v-v_*,\sigma) \left((f_1)'_*(f_2)' - (f_1)_* f_2 \right) \,\d \sigma \,\d v_*\,.
	\end{equation}
Here and below, we are using the shorthand notations $f_2=f_2(v)$, $(f_1)_*=f_1(v_*)$, as well as~$(f_2)'=f_2(v')$ and~$(f_1)'_*=f_1(v'_*)$.
In this expression, $v$, $v_*$ and $v'$, $v'_*$ are the velocities of a pair of particles after and before collision.
We make a choice of parametrization of the set of solutions to the conservation of momentum and energy (physical laws of elastic collisions):
	\begin{equation} \label{eq:elasticlaws}
	\begin{aligned}
	v+v_*&=v'+v'_* \\
	|v|^2+|v_*|^2&=|v'|^2+ |v'_*|^2
	\end{aligned}
	\end{equation}
so that the pre-collisional velocities are given by
	\[
	v'=\frac{v+v_*}{2} + \frac{|v-v_*|}{2} \sigma\,, \quad v'_*=\frac{v+v_*}{2}- \frac{|v-v_*|}{2} \sigma\,, \qquad \sigma \in \S^{2}\,.
	\]
The Boltzmann collision kernel $B=B(v-v_*,\sigma)$ only depends on the relative velocity~$|v-v_*|$ and on the deviation angle $\vartheta$ through $\cos \vartheta = \la v-v_*, \sigma \ra / |v-v_*|$
where $\la \cdot, \cdot \ra$ is the usual scalar product in $\R^3$.
The form of the collision kernel depends on the type of collisions that occur between particles. In dimension~$3$  in the case where particles behave   as billiard balls, known as the hard-spheres case, the collision kernel is proportional to the norm of the relative velocity, namely
	\[
	B(v-v_*,\sigma) = C |v-v_*|\,, \qquad C>0\,.
	\]
When particles interact through inverse power law potentials of type 
	\begin{equation} \label{eq:potential}
	\phi(r)=r^{-(p-1)} \qquad \text{with} \qquad p \in (2,+\infty)\,,
	\end{equation}
the collision kernel cannot be computed explicitly but Maxwell~\cite{Maxwell} has shown that the collision kernel can be computed in terms of the interaction potential~$\phi$. More precisely, in dimension~$3$, 
the kernel~$B$ satisfies the following properties.
	\begin{itemize}[leftmargin=*]
	\item[--] It takes product form in its arguments as
		\begin{equation} \label{eq:product}
		{B(v-v_*,\sigma) = |v-v_*|^\gamma \, b(\cos \vartheta)\,.}
		\end{equation}
		\item[--] The angular function $b$ is locally smooth, and has a nonintegrable singularity for $\vartheta \rightarrow 0$:
	it satisfies for some $c_b>0$ and any $\vartheta \in (0,\pi/2]$, 
		\begin{equation} \label{eq:angularsing}
		\frac{c_b}{\vartheta^{1+2s}} \leq \sin \vartheta \, b(\cos \vartheta) \leq \frac{1}{c_b \,  \vartheta^{1+2s}}
		\qquad \text{with} \qquad s:=\frac{1}{p-1} \in (0,1)\,.
		\end{equation}
	\item[--] The parameter $\gamma$ is defined as
		\begin{equation} \label{eq:Phi}
		\gamma := \frac{p-5}{p-1}  \in (-3,1)\,.
		\end{equation}
	\end{itemize}
One traditionally calls {hard potentials} the case $p>5$ (for which $0<\gamma<1$), {Maxwell molecules} the case~$p=5$ (for which $\gamma=0$), {moderately soft potentials} the case corresponding with~$3\leq p<5$ (for which $-2s \leq \gamma <0$) and {very soft potentials} the case $2<p<3$ (for which $-3 < \gamma <-2s$). In this paper, we shall not consider the very soft potentials case, meaning we shall restrict to~$\gamma \geq -2s$ (see Remark~\ref{restrictions remark} for a discussion on this restriction). 

  Grad's cut-off assumption consists in additionally supposing that the angular kernel~$b$ is integrable on the sphere by removing its singularity for small deviation angles~$\vartheta$ (see~\eqref{eq:angularsing}). In that case, the Boltzmann collision operator is thus of the form~\eqref{def:Bolop} with 
	\[
	B(v-v_*,\sigma) = b(\cos \vartheta) |v-v_*|^\gamma
	\quad \text{with}  \quad \int_{\Sp^2} b(\cos \vartheta) \, \d\sigma < \infty 
	\quad \text{and}  \quad \gamma \in (-3,1]\,.
	\] 
Notice that this in particular includes the case of hard-sphere collisions by taking the angular kernel to be constant. Here again, we do not consider   the very soft potentials case, that is we restrict ourselves to $\gamma\geq 0$. 

In the case of the Coulomb potential ($s = 1$ and thus $\gamma=-3$), the Boltzmann operator does not make any  sense (see~\cite{BookVillani} for example). The Boltzmann operator has then to be replaced by the Landau one which can be obtained in the so-called grazing collision limit after having made a cut-off on the Coulomb interaction.
The Landau operator, defined in~1936 by Landau~\cite{Landau} (independently of the Boltzmann operator), is used in plasma physics and is an integro-differential operator  given by
	\begin{equation}\label{def:Landauop}
		Q_L(f_1,f_2)(v) := \partial_{v_i}  \int_{\R^d} a_{ij}(v-v_*) \left( f_1(v_*) \partial_{v_j} f_2(v) - f_2(v) \partial_{v_j} f_1(v_*) \right)  \d v_*\,,
	\end{equation}
where we use the convention of summation of repeated indices. The matrix $a_{ij}$ is symmetric, semi-positive and is given by
	\begin{equation}\label{def:aij}
		a_{ij}(v) := |v|^{\gamma+2}\left( \delta_{ij} - \frac{v_i v_j}{|v|^2}\right)\,, \quad -3 \le \gamma \le 1\,.
	\end{equation}
Similarly to the Boltzmann equation, we have the following classification according to the values of~$\gamma$: interactions are referred to as {hard potentials} if~$\gamma\in(0,1]$, {Maxwellian molecules} if~$\gamma=0$, {moderately soft potentials} if~$\gamma \in [-2,0)$, {very soft potentials} if~$\gamma \in (-3,-2)$ and {Coulomb potential} if~$\gamma=-3$.
We mention that only the case $\gamma=-3$ is relevant from a physical viewpoint and is the one that has been derived by Landau in~\cite{Landau}. Once more, we shall only consider not too soft potentials, which correspond to $\gamma \geq -2$. 

In the three cases (Boltzmann with and without cut-off assumption and Landau), weak formulations of the collision operators allow to obtain the following conservation laws:
	\begin{equation} \label{prop:conserv1}
	\int_{\R^3} Q(f,f)(v) \, \varphi(v) \,\d v = 0 \qquad \text{for} \qquad \varphi(v) = \,1,  v\,,  |v|^2,
	\end{equation}
as well as Boltzmann's H-theorem that asserts that Boltzmann's entropy of solutions to these equations, namely~$\displaystyle \int f \log f\,\d x\d v$,  is non-increasing along time. Moreover, the second part of the theorem states that any distribution minimizing the entropy is a local Maxwellian
distribution in velocity.

\subsection{Hydrodynamic limit}

All kinetic models leading to incompressible models are based on a regime in which both the Strouhal and the Knudsen numbers are small. 
In order to reach the incompressible Navier-Stokes equation, we shall work with ${\rm St} = {\rm Kn} = \eps \ll 1$ (see for example~\cite{BGL1}). Our kinetic equation then reads
	\begin{equation} \label{eq:scaledBol}
	\left\{
	\begin{array}{ll}
		\partial_t F^\eps + \eps^{-1} v \cdot \nabla_x F^\eps 
		=\eps^{-2} Q(F^\eps,F^\eps)
		&\text{in} \quad \R^+ \times \T^3 \times \R^3 \\[0.1cm]
		F^\eps_{|t=0}=F^\eps_{\rm in} &\text{in} \quad \T^3 \times \R^3 \,. 
	\end{array}
	\right. 
	\end{equation}
The Knudsen number is actually proportional to the inverse of the average number of collisions for each particle per unit of time. Taking $\eps$ small has thus the effect of enhancing the role of collisions.
To relate our kinetic models to the incompressible Navier-Stokes equation, we then look at equation~\eqref{eq:scaledBol} under the following linearization of order $\eps$:
	\[
	F^\eps = \mu + \eps \mu^{\frac12} f^\eps\,,
	\]
where $\mu$ is the global Maxwellian defined by
	$$
	\mu(v):=\frac1{(2\pi)^\frac 32}e^{-\frac{|v|^2}2}\, .
	$$
The equation we are going to study on the fluctuation $f^\eps$ is thus the following:
	\begin{equation} \label{eq:feps_intro}
	\left\{
	\begin{array}{ll}
		\partial_t f^\eps +\eps^{-1} v \cdot \nabla_x f^\eps = \eps^{-2} L f^\eps + \eps^{-1} \Gamma(f^\eps , f^\eps)
		&\text{in} \quad \R^+ \times \T^3 \times \R^3 \\[0.1cm]
		f^\eps_{|t=0}= f_{\rm in}^\eps := \eps^{-1}(F^\eps_{\rm in}-\mu)\,\mu^{-\frac12} &\text{in} \quad \T^3 \times \R^3 
	\end{array}
	\right. 
	\end{equation}
with 
	$$
	\Gamma (f_1,f_2) :=\mu^{-\frac12} Q (\mu^{\frac12} f_1 , \mu^{\frac12} f_2) 
	$$
and
	\begin{equation} \label{def:L}
	L f := \Gamma(\mu^{\frac12} , f) + \Gamma (f,\mu^{\frac12}) \,.
	\end{equation}
We say that a distribution $f=f(x,v)$ has global mass, momentum and energy when it satisfies 
\begin{equation}\label{eq:normalization}
\int_{\T^3} \! \int_{\R^3} f(x,v) \,\varphi(v) \mu^{\frac12}(v)\,\d v \, \d x = 0 \quad \text{for} \quad \varphi(v)= 1 , v , |v|^2\,.
\end{equation}
Conservation laws~\eqref{prop:conserv1} imply that  the perturbation $f^\eps$ satisfies~\eqref{eq:normalization} for all times $t \geq 0$ if~$F_{\rm in}^\eps$ satisfies 
	$$
	\int_{\T^3} \! \int_{\R^3} F_{\rm in}^\eps (x,v) \,\varphi(v) \, \d v \, \d x = 
	\int_{\R^3} \mu(v) \,\varphi(v) \, \d v  \quad \text{for} \quad \varphi(v)= 1 , v , |v|^2\,.
	$$

For every $f=f(x,v)$ we write the decomposition
	$$
	f = \P^\perp_0 f + \P_0 f\,,\, \qquad \P^\perp_0 := \mbox{Id} - \P_0\,,
	$$
where $\P_0$ is the orthogonal projection onto 
	\begin{equation}\label{eq:Ker L}
	\Ker L =\left \{\mu^{\frac12}(v) , v_1 \mu^{\frac12}(v), v_2 \mu^{\frac12}(v) , v_3 \mu^{\frac12}(v), |v|^2 \mu^{\frac12}(v)\right\}
	\end{equation}
given by  the so-called hydrodynamic modes
	\begin{equation} \label{def:P0}
	\P_0 f (x,v) := \left\{ \rho[f](x)  + u[f](x) \cdot v  + \theta[f](x) \frac{|v|^2-3}{2} \right\} \mu^{\frac12}(v)
	\end{equation} 
where
	$$
	\begin{aligned}
	\rho[f](x) &:= \int_{\R^3} f(x,v) \mu^{\frac12}(v) \, \d v \,,\\
	u[f](x) &:= \int_{\R^3}  f(x,v) v \mu^{\frac12}(v) \, \d v \,,\\
	\theta[f](x) &:= \int_{\R^3} f(x,v) \frac{|v|^2-3}{3} \mu^{\frac12}(v) \, \d v \,.
	\end{aligned}
	$$
 Returning to~(\ref{eq:feps_intro}), it is expected that as~$\eps$ goes to zero, the solution~$f^\eps$ should converge to an element of~$\Ker  L$. This is actually proved in  many situations (see Paragraph~\ref{state of the art} below), and in particular the hydrodynamic modes of the limit satisfy the  incompressible Navier-Stokes Fourier system
	\begin{equation} \label{eq:NSF}
	\left\{
\begin{aligned}
\partial_t u + u \cdot \nabla u- \nu_{\rm \small NS}  \Delta u &= - \nabla p \\
\partial_t \theta  + u \cdot \nabla \theta- \nu_{\rm \small heat}  \Delta \theta &= 0 \\
\mbox{div}\, u & = 0 \\
 \nabla(\rho+\theta) &= 0 \, .
\end{aligned}
\right.
\tag{NSF}
\end{equation}
 To define the viscosity coefficients, we introduce the two unique functions~{$\bf \Phi$ (which is a matrix function) and $\bf \Psi$} (which is a vectorial function) in~$(\operatorname{Ker} L)^\perp$ such that 
$$
\mu^{-\frac12}L\big(\mu^{\frac12} {\bf \Phi}\big) =  {|v|^2\over 3 } {\rm{Id}} -v\otimes v 
\quad \text{and} \quad 
\mu^{-\frac12}L\big(\mu^{\frac12} {\bf \Psi}\big) =v\Big(\frac{5}{2}-{|v|^2\over 2}\Big) \, .
$$
The viscosity coefficients are  then defined by
$$ 
\begin{aligned}
\nu_{\rm \small NS}  :=\frac{1}{10}\int{\bf \Phi} : L\big(\mu^{\frac12}{\bf \Phi}\big) \mu^{\frac12} \, \d v \quad \text{and} \quad
\nu_{\rm \small heat}   :=\frac{2}{15} \int {\bf \Psi} \cdot L\big(\mu^{\frac12}{\bf \Psi}\big) \mu^{\frac12} \, \d v  \, .
\end{aligned}
$$

\subsection{Functional framework and notation}\label{sect:notation}
%
In order to treat the three cases (Boltzmann with and without cut-off assumption and Landau equations) in a unified framework, we introduce the space $H^{s,*}_v$ with $s \in [0,1]$ by: for $s = 0$ (corresponding to the Boltzmann operator with cutoff)  
	\begin{equation}\label{eq:def:H0*v}
	\| f \|_{H^{0,*}_v} = \| \la v \ra^{\frac\gamma2}f \|_{L^2_v } \,,
	\end{equation}
for $s \in (0,1)$ (corresponding to the Boltzmann operator without cutoff) 
	\begin{equation}\label{eq:def:Hs*v}
	\begin{aligned}
	\| f \|_{H^{s,*}_v }^2 
	&:=  \int_{\R^3}\int_{\R^3}\int_{\S^2} b(\cos \vartheta) |v-v_*|^\gamma \mu(v_*)[f(v') - f(v)]^2 \, \d \sigma \, \d v_* \, \d v \\
	&\quad 
	+  \int_{\R^3}\int_{\R^3}\int_{\S^2} b(\cos \vartheta) |v-v_*|^\gamma  f(v_*)^2 [\mu^{\frac12}(v') - \mu^{\frac12}(v) ]^2 \, \d \sigma \, \d v_* \, \d v \,,
	\end{aligned}
	\end{equation}
and finally for $s=1$ (corresponding to the Landau operator) we define
	\begin{equation}\label{eq:def:H1v*}
	\begin{aligned}
	\| f \|^2_{H^{1,*}_{v}} 
	&:= \big\| \langle v \rangle^{\frac{\gamma}{2}+1}  f \big\|_{L^{2}_{v}}^{2}
	+ \big\| \langle v \rangle^{\frac{\gamma}{2}} \operatorname{pr}_{v}\nabla_{v}f \big\|_{L^{2}_{v}}^{2}
	+ \big\| \langle v \rangle^{\frac{\gamma}{2}+1}(\operatorname{Id}-\operatorname{pr}_{v})\nabla_{v}f \big\|_{L^{2}_{v}}^{2}\,,
	\end{aligned}
	\end{equation}
where $\operatorname{pr}_v$ stands for the projection on $v$, namely
 $$\forall w \in \R^3\, , \quad \operatorname{pr}_v w = \left(w\cdot\frac{v}{\vert v \vert}\right)\frac{v}{\vert v \vert} \, \cdotp$$ 
For every $s \in [0,1]$, we also define the dual space $(H^{s,*}_v)'$ endowed with the norm
	$$
	\| \phi \|_{(H^{s,*}_v)'} := \sup_{\| f \|_{H^{s,*}_v} \le 1 } \la \phi , f \ra \,.
	$$
It is worth mentioning that for $s \in [0,1]$ there holds (see \cite{AMUXY2,GS,HTT} for the case $s \in (0,1)$, the other cases being immediate), 
	$$
	\big\| \la v \ra^{\frac\gamma2+s} f \big\|_{L^2_v}
	+ \big\| \la v \ra^{\frac\gamma2} f \big\|_{H^{s}_v} 
	\lesssim \| f \|_{H^{s,*}_v} 
	\lesssim \big\| \la v \ra^{\frac\gamma2+s} f\big\|_{H^{s}_v}\,.
	$$
We recall that if~$\ell >3/2$, then~$H^\ell_x \subset L^\infty_x$. For $m \geq 0$, $T>0$ and when $E_v$ is a Lebesgue or Sobolev space in velocity, we define the space~$ \widetilde L^\infty\big([0,T], H^{m}_x E_v\big) $  (with the notation introduced in~\cite{cheminlerner}) through its norm
	\[
	\|f\|_{ \widetilde L^\infty\big([0,T], H^{m}_x E_v\big)}^2
	:= \sum_{k \in \Z^3}  \langle k\rangle^{2m}\big \|\widehat f(\cdot, k,\cdot)\big\|_{L^\infty([0,T],E_v)}^2  \, .
	\]
	We have denoted by~$(\widehat f(k))_{k \in \Z^3}$ the Fourier coefficients of~$f$  in the space variable. When more convenient, we will sometimes use the notation~$\mathcal F_xf$ for~$\widehat f$. To lighten notation, we shall often write~$   L^p_I H^{m}_x E_v $ for~$L^p(I, H^{m}_x E_v )$ and similarly~$ \widetilde L^\infty_I  H^{m}_x E_v $ for~$ \widetilde L^\infty\big(I, H^{m}_x E_v\big)$ when~$I$ is an interval of~$\R^+$. If~$I = [0,T]$ we will simply write~$ L^p_T H^{m}_x E_v$ and~$ \widetilde L^\infty_T  H^{m}_x E_v $.   Finally if~$T=\infty$ and in the absence of ambiguity we  write~$ L^p_t
  H^{m}_x E_v $ for~$L^p(\R^+, H^{m}_x E_v )$.

\medskip
We will   use the notation~$\mathbb P$ for the Leray projector onto divergence free vector fields.
For any triplet $(\rho_{\rm in}, u_{\rm in}, \theta_{\rm in})  $ defined on~$\T^3$ (considered as initial data, whence the subscript~``in'') we denote their projection onto incompressible/Boussinesq modes by 
	\begin{equation} \label{def:barhydro}
	\bar \rho_{\rm in} := \frac25 \rho_{\rm in} - \frac35 \theta _{\rm in}\,, \qquad
	\bar u_{\rm in}  := \mathbb P u_{\rm in} 	\qquad \text{and} \qquad
	\bar \theta_{\rm in}  := - \rho_{\rm in} \,. 
	\end{equation}
The kinetic counterpart of~$(\bar\rho_{\rm in}, \bar u_{\rm in},\bar \theta_{\rm in})  $  will be denoted
\begin{equation} \label{def:bargin}
	  g_{\rm in} (x,v) := 
	\Big\{ \bar \rho_{\rm in} (x)  + \bar u_{\rm in} (x) \cdot v  + \bar \theta_{\rm in} (x) \frac{|v|^2-3}{2} \Big\} \mu^{\frac12}(v) 
	\end{equation}
and if~$(\rho,u,\theta)$  solves~\eqref{eq:NSF}   with the  initial data $(\bar \rho_{\rm in}, \bar u_{\rm in}, \bar \theta_{\rm in})$  then we will write
		\begin{equation} \label{def:g}
 	g(t,x,v) := 
	\Big\{ \rho (t,x)  + u (t,x) \cdot v  +\theta (t,x) \frac{|v|^2-3}{2} \Big\} \mu^{\frac12}(v)\,.
	\end{equation}
	{We will say that the initial data~$  g_{\rm in} $ is well-prepared if writes under the form~(\ref{def:bargin}).}
{Note that  if~$\rho_{\rm in}, u_{\rm in}, \theta_{\rm in}$ lie in~$ H^{\frac12}(\T^3)$, 
then the function~$g$ belongs to the functional space~$\widetilde L^\infty_T H^{\frac12}_x L^2_v \cap  L^2_TH^{\frac32}_x L^2_v $ for  
	{all~$T<T^*$ where~$T^*$ is the life span of the solution to the Navier-Stokes-Fourier system: more properties are provided at the beginning of Section~\ref{sec:strategy} below. }
 Actually~(\ref{def:g}) shows that~$g$ also  belongs  to~$\widetilde L^\infty_T H^{\frac12}_x H^{s,*}_v\cap  L^2_TH^{\frac32}_x H^{s,*}_v $.}
 {
 In what follows, we shall denote by $C$ any multiplicative constant that depends only on fixed numbers and its value may change from line to line. }
	{The following shorthand notation will also be useful in the following: for any  real number~$m $, the Sobolev spaces~$H_x^{m+ 0}$ and~$H_x^{m- 0}$ are defined by 
	$$
	f \in H_x^{m\pm 0}   \Longleftrightarrow 	\exists \eta>0 \,, \quad f \in H_x^{m\pm\eta}   \, .
	$$
By abuse of notation 	we shall denote by~$\|f \|_{H_x^{m\pm 0} }$ the norm of~$f$ in~$H_x^{m\pm\eta} $ with~$\eta$ arbitrarily small. }

\subsection{State of the art}  \label{state of the art}
We give here  a short overview of the existing literature on the problem of deriving fluid equations from kinetic ones.

The first justifications of the link between kinetic and fluid equations were formal and based on asymptotic expansions by Hilbert~\cite{Hilbert}, Chapman, Enskog~\cite{ChapEns} and Grad~\cite{Gradhydro}. The first rigorous convergence proofs based also on  asymptotic expansions were given by Caflisch~\cite{Caflisch} (see also~\cite{Lachowicz} and~\cite{DeMasi-Esposito-Lebowitz}). In those papers, the limit is justified up to the first singular time for the fluid equation. By using his nonlinear energy method, Guo~\cite{GuoBNS}   justified the limit towards the Navier-Stokes equation and beyond in Hilbert's expansion from Boltzmann and Landau equations.

There have also been some convergence proofs based on spectral analysis in the framework of strong solutions close to equilibrium introduced by Grad~\cite{Grad} and Ukai~\cite{Ukai} for the Boltzmann equation. In this respect, we refer to the works by Nishida~\cite{Nishida},  Bardos and Ukai~\cite{Bardos-Ukai}. These results use the description of the spectrum of the linearized Boltzmann equation in Fourier space in the space variable performed in~\cite{Nicolaenko,CIP,Ellis-Pinsky} by respectively Nicolaenko; Cercignani, Illner and Pulvirenti; Ellis and Pinsky. The approach in the present paper as well as in~\cite{GT,CRT,Gervais1,Gervais2,Gervais-Lods,CC} are reminiscent of these ones.

Finally, let us mention that this problem has been extensively studied in the framework of weak solutions, the goal being to obtain solutions for the fluid models from renormalized solutions introduced by DiPerna and Lions in~\cite{DiPerna-Lions} for the Boltzmann equation. We shall not make an extensive presentation of this program as it is out of the realm of this paper, but let us mention that it was started by Bardos, Golse and Levermore at the beginning of the nineties in~\cite{BGL1,BGL2} and was continued by those  authors, Saint-Raymond, Masmoudi, Lions among others. We mention here  a (non exhaustive) list of papers which are part of this program \cite{GSR1,GSR2,Levermore-Masmoudi,Lions-Masmoudi,SRbook}. 

More recently, some uniform in $\eps$ estimates on kinetic equations have allowed to prove (at least) weak convergence towards the Navier-Stokes equation. Let us mention~\cite{Jiang-Xu-Zhao,Rachid2} in which the cases of the Boltzmann equation without cut-off and the Landau equations are treated by Jiang, Xu and Zhao on the one hand and by Rachid on the other hand.
In~\cite{Briant-BNS,BMM}, Briant and Briant, Merino-Aceituno and Mouhot have obtained  convergence to equilibrium results for the rescaled Boltzmann equation (and also the Landau equation in~\cite{Briant-BNS}) uniformly in the rescaling parameter using respectively hypocoercivity and enlargement methods. In~\cite{BMM}, the authors are able to weaken the assumptions on the data down to Sobolev spaces with polynomial weights (see also~\cite{ALT} for the inelastic Boltzmann equation). Notice that Briant~\cite{Briant-BNS} has combined this with the Ellis and Pinsky result~\cite{Ellis-Pinsky} to recover strong convergence in the case of the elastic Boltzmann equation. To end this part, we mention the works~\cite{CRT,CC} in which the authors also obtain uniform in $\eps$ estimates on the Landau equation and Boltzmann equation without cutoff respectively and also obtain a result of strong convergence towards the incompressible Navier-Stokes equation. 

Finally, let us bring up more recent works that have inspired the present paper. First, the paper~\cite{GT} in which the second and third authors proved that the life span of the solutions to the rescaled Boltzmann equation (for hard-spheres collisions) is bounded from below by that of the Navier-Stokes equation for~$\eps$ small enough. The main feature of the proof was to perform a fixed point argument by using information on the limit system since the starting point is the solution of the Navier-Stokes system (which is not the most common viewpoint). Gervais~\cite{Gervais1,Gervais2}   extended the functional framework in which this result holds. He   proved  a similar result in polynomially weighted spaces, his strategy being a combination of~\cite{GT} and of the one used in~\cite{BMM} by Briant, Merino-Aceituno and Mouhot in order to get uniform in $\eps$ estimates on solutions in polynomially weighted spaces. We also point out the paper by Gervais and Lods~\cite{Gervais-Lods} in which a unified framework  is also provided, which encompasses a large class of kinetic equations (including in particular the result in~\cite{GT}).


\subsection{Main result}
All the results mentioned in the previous paragraph concerning the convergence of strong solutions are stated in   functional spaces which are usual for the study of strong solutions to nonlinear kinetic problems, namely in which there is  an algebra structure in the space variable, typically $H^\ell_x$ with~$\ell>3/2$ (more generally $\ell>d/2$ in dimension~$d$).  Indeed the collision operator~$Q_B$
involves the product of~$f(x,v)$ and~$f(x,v')$ at the same point~$x$, so continuity of~$f$ seems to be required to make sense of the product (this requirement is of course too strong: it is actually possible to relax it in some cases, see the work by Arsénio in~\cite{Arsenio2} for example).
However it is well-known that the Navier-Stokes equations can be solved for initial data with
less regularity, namely~$H^\frac12_x$ ($H^{\frac d2-1}_x$ in dimension~$d$). Our goal in this work is to analyze to what extent the assumptions  one   makes on the initial data~$f^\eps_{\rm in}$ to the kinetic equation~(\ref{eq:feps_intro}) can reflect this discrepancy between the kinetic and the fluid frameworks.

The  main goal of our analysis is thus to show that  given an initial data in~$H^{\frac12}_x$ for the incompressible (NSF) system, the associate solution  to~(NSF), as long as it exists,  is the limit  of a sequence of solutions to the rescaled Boltzmann or Landau equation. More precisely we are able to construct, on the same life span as the solution to~(NSF), a sequence of solutions to the kinetic equation associated with   initial data whose 
  hydrodynamic part converges in~$H^{\frac12}_x$ to the given hydrodynamic profile, and whose microscopic part converges to zero in~$H^{\frac12}_x$  and may blow up (in a controled way) in~$H^{\ell }_x$ for~$\ell>3/2$.        Let us also underline that there is no smallness assumption  on the initial data of the fluid system, and we are able to treat the cases of non-global and global solutions to the fluid system in a unified framework.

\begin{theo} \label{theo:main}
Let~$3/2 < \ell \le 2$ be given. Consider $(\rho_{\rm in}, u_{\rm in}, \theta_{\rm in}) \in H^{\frac12}(\T^3)$ that are mean-free{, such that~$u_{\rm in}  $ is divergence free and~$\rho_{\rm in}+\theta_{\rm in} = 0$. Let~$(\rho,u,\theta)$ be the unique solution to~\eqref{eq:NSF} associated with the initial data $( \rho_{\rm in}, u_{\rm in},\theta_{\rm in})$} in the space~$\widetilde L^\infty_TH^{\frac12}_x \cap L^2_TH^\frac32_x $, for some~$T>0$.

Consider two real numbers~$\alpha <1/4$ and  $\beta<1/2$. Let $f_{\rm in}^\eps $ be a family of functions such that
	\[
	 \P_0 f^\eps_{\rm in}= \psi(\eps^\alpha |D_x|)g_{\rm in} 		\]
	 {and~$\P_0^\perp f_{\rm in}^\eps$ is arbitrary, going to zero in the sense that }
	\[
\|\P_0^\perp f_{\rm in}^\eps\|_{H^{\frac12}_x L^2_v} + \eps^\beta \|\P_0^\perp f_{\rm in}^\eps\|_{H^\ell_x L^2_v}
	\xrightarrow[\eps \to 0]{} 0
	\]
for some smooth, compactly supported function~$\psi$. Then, there is $\eps_0>0$ such that for any~$\eps \leq \eps_0$, there exists a unique solution $f^\eps$ to the kinetic equation~\eqref{eq:feps_intro} with initial data~$f_{\rm in}^\eps$, which belongs to the space~$\widetilde L^\infty_TH^{\ell}_x L^2_v  \cap L^2_TH^{\ell}_x H^{s,*}_v $, and it moreover satisfies, with  notation~{\rm(\ref{def:g})},
	\[
	\|f^\eps - g\|_{  \widetilde L^\infty_T H^{\frac12}_x L^2_v  }+\|f^\eps - g\|_{  L^2_TH^\frac32_x H^{s,*}_v  } \xrightarrow[\eps \to 0]{} 0 \,. 
	\]

\end{theo}
 \begin{rem}
 The restriction~$\ell\le 2 $ is purely technical, the result would hold for any~$\ell > 2$ up to some adaptations in the nonlinear estimates.  {Note that~$ \P_0 f^\eps_{\rm in}$  is a smoothened version of the well-prepared data~$g_{\rm in}$, with the higher regularity norms   allowed to blow up, in a controled way, with~$\varepsilon$.}
 The threshold value $1/4$ for the truncation parameter $\alpha$ comes from technical considerations that appear throughout the proof.  Note that such an assumption (the cut-off in frequency space) is reminiscent of the setting chosen in~\cite{DH} in the context of the incompressible limit.
 The additional parameter $\beta$ quantifies the possible blow up of the~$H^{\ell}_x H^{s,*}_v $ norm of the ``microscopic'' part of the initial data.    \end{rem}
 \begin{rem}\label{life span}
The proof of Theorem~\ref{theo:main} shows that 
if the solution~$(\rho,u,\theta)$ exists globally in time, regardless of the size of the initial data, the parameter~$\eps_0$ may be chosen uniformly in~$T$ (as is the case in~\cite{GT}).
  \end{rem}
\begin{rem}\label{restrictions remark}
Throughout this paper, we only consider the case of well-prepared data in the torus and also only the case of not too soft potentials for the kinetic equations. We believe that using the same method of proof combined with arguments and estimates of~\cite{GT,CC}, our analysis could be extended to a more general setting by considering the problem in the whole space (also including ill-prepared data in~$\R^3$) and very soft potentials for the kinetic equations. 
 \end{rem}
  
\subsection{Sketch of the proof and plan of the paper} 
The idea of the proof follows the method of~\cite{GT}, consisting in solving by a fixed point argument the equation obtained by taking  the difference between the kinetic and hydrodynamic equations, written in Duhamel form. The main interest of this equation is that it no longer involves the kinetic unknown but writes schematically as
 	\begin{equation} \label{eq:fixedpoint}
	\delta^{\eps}(t) 
	= \mathcal D^{\eps}(t) +  \mathcal S^\eps(t) + \mathcal L^{\eps}[\delta^{\eps}](t) +  \Psi^{\eps}[\delta^{\eps},\delta^{\eps}]  (t) \, , 
	\end{equation}
where~$\mathcal D^{\eps}(t)$ depends only on the initial data{~$g_{\rm in}$}  {(recall that $  g_{\rm in}$ is defined in~\eqref{def:bargin})}, $\mathcal S^\eps(t)$ is a source term depending only on the hydrodynamical solution{~$g$}  ,~$  \mathcal L^{\eps}[\cdot]$ is a linear operator depending on the hydrodynamic solution~$g$, and~$ \Psi^{\eps}[\cdot,\cdot] $ is the usual, Boltzmann bilinear operator (see~(\ref{eq:deltaeps})  below). The difficulty then consists in proving that~$\mathcal D^{\eps}(t) $ and $\mathcal S^\eps(t)$ are small, and that~$  \Psi^{\eps}$ is bilinear continuous,    in   
   a low regularity framework. An additional difficulty comes from the fact that~$ \mathcal L^{\eps}$ is not small if~$  g_{\rm in}$ is not small: since smallness is necessary for the fixed-point to work, we devise a Gronwall-type argument   to get round this difficulty (in this regard also, the proof differs from the one presented in~\cite{GT}).
   
In Section~\ref{sec:prelim}, we give some useful tools to estimate each part of equation~\eqref{eq:fixedpoint} (spectral decomposition, semi-group and nonlinear estimates). In Section~\ref{sec:strategy}, we  reduce the proof  of Theorem~\ref{theo:main} to a number of intermediate estimates. These estimates are proved   in Sections~\ref{sec:some results} and~\ref{sec:deltaeps}.

\medskip
 \noindent\textbf{Acknowledgments.}  KC was partially supported by the Project CONVIVIALITY ANR-23-CE40-0003 of the French National Research Agency (ANR).
 IT was supported by the French government through the France 2030 investment plan managed by the ANR, as part of the Initiative of Excellence Université Côte d’Azur under reference
number ANR-15-IDEX-01.

\section{Preliminaries}\label{sec:prelim}

 Our approach heavily relies on previous results on the spectral analysis of the linearized kinetic operator
$$
	\Lambda^\eps := \frac{1}{\eps^2}L - \frac{1}{\eps} v \cdot \nabla_x 	$$  in Fourier space for the space variable $x$ (see~\cite{Nicolaenko,Ellis-Pinsky,Gervais1,Gervais-Lods}), where we recall that $L$ is defined in~\eqref{def:L}.
We denote by~$U^\eps(t)$ the semi-group associated to~$	\Lambda^\eps$.
	
Taking the Fourier transform in the space variable, we denote, for all $ k \in  \Z^3$, 
	$$
	\widehat\Lambda^\eps( k) := \frac{1}{\eps^2} L - \frac{1}{\eps} i v \cdot  k
	$$
and
	$
	\widehat U^\eps (t, k) := e^{t \widehat\Lambda^\eps( k)},
	$
so that
	$
	U^\eps(t) = \FF_x^{-1} \widehat U^\eps(t,\cdot) \FF_x.
	$
We also denote
	\begin{equation}\label{defPsieps}
	\Psi^\eps[f_1,f_2] (t) := \frac{1}{\eps} \int_0^t U^\eps(t-t') \Gamma_{\mathrm{sym}}(f_1(t') , f_2(t')) \, \d t'\,,
	\end{equation}
	where $\Gamma_{\mathrm{sym}} (f_1,f_2) := \big(\Gamma(f_1,f_2) + \Gamma(f_2,f_1)\big)/2$ denotes the symmetrized form of $\Gamma$,	so that~(\ref{eq:feps_intro}) takes the Duhamel form
	\begin{equation} \label{eq:feps_intro Duhamel}
	f^\eps(t) =U^\eps(t) f^\eps_{\rm in} + \Psi^\eps[f^\eps,f^\eps] (t) \, .
	\end{equation}
In Fourier space we have
	$$
	\widehat \Psi^\eps  [f_1,f_2] (t, k) = \frac{1}{\eps} \int_0^t \widehat U^\eps(t-t',  k) \widehat \Gamma_{\mathrm{sym}}\big(f_1(t') , f_2(t')\big)( k) \, \d t'\,,
	$$
where  	
	$$
	\widehat \Gamma_{\mathrm{sym}}(f_1 , f_2)( k) := \sum_{ k'  \in \Z^3} \Gamma_{\mathrm{sym}}\big (\widehat f_1( k- k' ) , \widehat f_2( k' )\big) \,.
	$$
Observe that
	$$
	\Psi^\eps [f_1,f_2] (t) = \FF_x^{-1} \widehat \Psi^\eps[f_1,f_2] (t,\cdot) \FF_x\,.
	$$
It turns out that there is a complete description of the operator~$U^\eps$: this goes back to~\cite{Nicolaenko,Ellis-Pinsky} for the Boltzmann  hard-spheres kernel,~\cite{YY} for the Boltzmann non-cutoff (resp. Landau) kernels with hard and moderately soft potentials~$\gamma+2s \ge 0$ (resp. $\gamma+2 \geq 0$), and \cite{YY2} for the Boltzmann non-cutoff (resp.$\!$ Landau) kernels with very soft potentials~$\gamma+2s < 0$ (resp.~$\gamma+2<0$). For the not too soft potentials, we also refer to the paper~\cite{Gervais-Lods} in which the authors provide a more modern spectral approach. 

Let us start by noticing that
	\begin{equation}\label{def:Ueps}
	\widehat U^\eps(t, k) =  \widehat U^1\Big(\frac t{\eps^2},\eps  k\Big) \, .
	\end{equation} 
	Roughly speaking, for~$ | k| \leq \kappa $ small enough, the operator~$\widehat \Lambda^1( k):= L-iv \cdot  k$ can be seen as a perturbation of~$L$. In particular it can be proved (see \cite{Ellis-Pinsky}) that the~5-dimensional kernel of~$L$ recalled in~(\ref{eq:Ker L}) splits into 4 eigenvalues (the first one below is double) that satisfy for all~$ | k| \leq \kappa $
	\begin{equation}\label{NS heat eigenvalues}
	\begin{aligned}
	\lambda_{\rm \small NS}( k)&: =   - \nu_{\rm \small NS} | k|^2 + \gamma_{\rm \small NS}( k) \,,\quad 
	\nu_{\rm \small NS}>0 \, , \quad |\gamma_{\rm \small NS} ( k)|\le \frac {\nu_{\rm \small NS}}2 | k|^2   \\
	\lambda_{\rm \small heat}( k)&: =   - \nu_{\rm \small heat} | k|^2 + \gamma_{\rm \small heat}( k)  \,, \quad
	\nu_{\rm \small heat}>0 \, , \quad |\gamma_{\rm \small heat} ( k)|\le \frac {\nu_{\rm \small heat}}2 | k|^2 
	\end{aligned}
	\end{equation}
and
	\begin{equation}\label{wave eigenvalues}
	\begin{aligned}
	\lambda _{\rm \small wave\pm} ( k) &:= \pm i c | k| - \nu_{\rm \small wave\pm}   | k|^2 + \gamma _{\rm \small wave\pm}( k) \, , \\
	&\qquad\qquad c>0 \, , \, \,  \nu_{\rm \small wave\pm} >0 \, , \quad
	|\gamma _{\rm \small wave\pm}( k)|\le \frac { \nu_{\rm \small wave\pm} }2 | k|^2 \,.
	\end{aligned}
	\end{equation}
Moreover, the associate projectors~$\mathcal P_\star$ can be written (where~$\star$ stands for~$\rm \small NS, \rm \small heat,  $ or~$\rm wave\pm$)
\begin{equation}\label{dec Pstar}
	\mathcal P_\star  = \mathcal P_{\star}^0\left({ k \over | k|}\right) +| k|  \mathcal P_{\star}^1\left({ k \over | k|}\right) 
	+ | k|^2  \mathcal 	P_{\star}^2( k)\, ,
	\end{equation}
with  $ \mathcal P_{\star}^n$  bounded linear operators on $L^2_v$ with operator norms uniform for $| k| \le \kappa$. We even have that  $\mathcal P_\star^0( k/| k|)$, $\mathcal P_\star^1( k/| k|)$ and $\mathcal P_\star^2( k)$ are bounded from $(H^{s,*}_v)'$ into $H^{s,*}_v$ uniformly in $| k| \leq \kappa$. We refer to~\cite[Theorem~1.6-(2)]{Gervais-Lods} for this property (note the following correspondance of notation $H^\bullet = H^{s,*}_v$ and $H^\circ = (H^{s,*}_v)'$).  We also have that if $\star \neq \star'$, then  $ \mathcal P^0_\star  \mathcal P^0_{\star'}=0$ and the orthogonal  projector~$\mathbf P_0$ onto $\operatorname{Ker} L$ satisfies
\begin{equation}\label{remarkP0}		
	\mathbf P_0  = \sum_{\star \in \{\rm NS, heat, wave\pm\}}    \mathcal P_{\star}^0\Big({ k \over | k|}\Big)\, .
	\end{equation}
Actually $ \mathcal P_{\rm \small NS}^0( k/| k|)$ is the projection onto the $2$-dimensional space spanned by~$v-
\operatorname{pr}_ k v  $ for any $ k$ (this corresponds to the divergence free condition), and 
	$$
  	\mathcal P_{\rm \small heat}^0\Big({ k \over | k|}\Big) \hat f ( k,v) 
	= \frac{2}{5} \Big(-1+ \frac 12 (|v|^2-3)\Big)   \mu^{\frac12} (v)\int_{\R^3} \Big(-1+ \frac 12 (|w|^2-3)\Big)  \mu^{\frac12}  (w)\widehat f ( k,w)\, \d w \, .
 	$$ 
Finally 
 	$$
 	\begin{aligned}
 	& \mathcal P_{\rm wave\pm}^0\Big({ k \over | k|}\Big) \hat f ( k,v)\\
 	&  = \frac{3}{10} \Big(1 \pm \frac{ k}{| k|} \cdot v + \frac 13 (|v|^2-3)\Big) \mu^{\frac12}(v) \int_{\R^3} \Big(1 \pm \frac{ k}{| k|} \cdot w
	+ \frac 13 (|w|^2-3)\Big) \mu^{\frac12} (w)\widehat f  ( k,w)\, \d w \, .
 	\end{aligned}
 	$$ 
Thanks to~(\ref{def:Ueps}) and to this spectral study, we deduce as in~\cite{Bardos-Ukai,GT} that~$U^\eps$ can be decomposed as follows:
	\begin{equation}\label{decomposition Ueps}
	U^\eps(t) = U^{\eps,\flat}(t) + U^{\eps,\sharp}(t)
	\end{equation}
where~$ U^{\eps,\flat}(t) $  corresponds to the contribution of the low frequencies in the right part of the plane: 
	\begin{equation}\label{decomposition Uepsflat}
	\widehat U^{\eps,\flat}(t, k)  
	:= \chi \Big(\frac{\eps | k|}\kappa\Big)
	\sum_{\star \in \{\rm NS, heat, wave\pm\}}  e^{ \lambda_\star( \eps  k) \frac{t}{\eps^2} } \mathcal P_\star(\eps  k)\,,
	\end{equation}
	where~$\chi$ is a fixed smooth, compactly supported function.
 Moreover, since we consider not too soft potentials, there is $\lambda_0 >0$ such that uniformly in~$ k \in  \Z^3$  	
 	\begin{equation} \label{eq:decayUepssharp}
	\| \widehat U^{\eps,\sharp} (t, k) \|_{L^2_v \to L^2_v} \lesssim e^{- \lambda_0 \frac{ t}{\eps^2}}\, ,
	\qquad \forall  t \geq 0\,.
	\end{equation} 
{Notice that in the case of very soft potentials, the exponential decay should be replaced by an algebraic one (see for example~\cite{YY2}).}
In the study  of the limit~$\eps \to 0$ of~(\ref{eq:feps_intro}), it will be useful to decompose~$ U^{\eps,\flat}(t) $ into a part independent of~$\eps$ and a remainder, which will be shown to go to zero in a sense to be made precise later:
	\begin{equation}\label{decomposition Ueps1}
	 U^{\eps,\flat}  = U_{\rm NSF}  +\widetilde U^{\eps}_{\rm NSF}  + U^{\eps,\flat}_{\rm wave} 
 \, ,	 
	\end{equation}
where in Fourier variables
	\begin{equation}\label{decomposition Ueps NSF}
	\begin{aligned}
	\widehat U_{\rm NSF}(t, k) 
	&:= e^{ - \nu_{\rm \small NS} | k|^2 t} \mathcal P_{\rm \small NS}^0 \Big({ k \over | k|}\Big)+
	 e^{ - \nu_{\rm \small heat} | k|^2 t} \mathcal P_{\rm \small heat}^0\Big({ k \over | k|}\Big)\\
	\widehat U^{\eps,\flat}_{\rm wave}(t, k)
	&:=\chi \Big(\frac{\eps | k|}\kappa\Big) \sum_{\pm} 
	e^{ \lambda _{\rm \small wave\pm} (\eps  k) \frac t{\eps^2}} \mathcal P_{\rm wave\pm} (\eps  k) \, .
	\end{aligned}
	\end{equation}
According to~\eqref{defPsieps} and~\eqref{decomposition Ueps}, we can also decompose  
\begin{equation}\label{decomposition Psieps 1}
	\Psi^{\eps} = \Psi^{\eps,\flat} + \Psi^{\eps,\sharp} 
\end{equation}
where 
	\begin{equation}\label{decomposition Psieps 11}
	\mathcal{F}_x \left(  \Psi^{\eps  ,\sharp}[f_1,f_2](t)\right)( k) 
	:={1 \over \eps}  \int_0^t \widehat U^{\eps,\sharp}(t-t') \widehat  \Gamma_{\mathrm{sym}} \big(f_1(t'),f_2(t')\big) ( k)\, \d t'
	\end{equation}
and 
	\begin{equation}\label{decomposition Psieps 12}
	\mathcal{F}_x \left(  \Psi^{\eps ,\flat}[f_1,f_2](t)\right)( k) 
	:=\frac1\eps\chi \Big(\frac{\eps | k|}\kappa\Big)
	\sum_{\star}\int_0^t  e^{ \lambda_\star( \eps  k) \frac{t-t'}{\eps^2} }
	\mathcal P_\star(\eps  k) \widehat \Gamma_{\mathrm{sym}} \big(f_1(t'),f_2(t')\big) ( k)\, \d t' \,,
		\end{equation}
where  the sum runs over~$ \{\rm NS, heat, wave\pm\}$.
In the interest of the limit~$\eps\to 0$, this can be again decomposed as in~(\ref{decomposition Ueps1}), as follows:
\begin{equation}\label{decomposition Psieps 22}
	\Psi^{\eps,\flat}  
	= \Psi _{\rm NSF}+ \widetilde  \Psi_{\rm NSF}^{\eps}
 + \Psi^{\eps,\flat}_{\rm wave}  \end{equation}
 where writing~$\widehat\Psi_\star [f_1,f_2](t)=\mathcal{F}_x \big( \Psi _{\rm \star} [f_1,f_2](t)\big)  $  and recalling that~$\mathbf P_0 \Gamma_{\mathrm{sym}} = 0$,
 \begin{align*}
	\widehat \Psi _{\rm NSF} [f_1,f_2](t, k) 
	&:= \sum_{\star \in \{\rm NS, heat\}}
 \int_0^t e^{  -\nu _{\star}  (t-t') | k|^2} | k|\mathcal P_{\star} ^{1}\Big({ k \over | k|}\Big) \widehat \Gamma_{\mathrm{sym}} \big(f_1(t'),f_2(t')\big) ( k) \, \d t' 	\,,  \\
	\widehat \Psi^{\eps ,\flat} _{\rm wave} [f_1,f_2](t, k) 
	&:=  
 \chi\Big (\frac{\eps| k|}\kappa\Big)  \\
 &\hspace{0,5cm}
	\times \, \sum_{\pm}
	\int_0^t e^{  \lambda _{\rm \small wave\pm} (\eps  k)\frac {t-t'}{\eps^2}}   \mathcal P_{\rm \small wave\pm} (\eps k)   \widehat \Gamma_{\mathrm{sym}} \big(f_1(t'),f_2(t')\big) ( k)\, \d t' \,.
		\end{align*}
 It  can be checked that the solution~$g$ constructed in~(\ref{def:g}),  starting from~${ g_{\rm in}}$ as  defined in~(\ref{def:bargin}), satisfies
 \begin{equation}\label{Duhamelfluid}	 
 	g (t) = U_{\rm NSF} (t) g_{\rm in}+ \Psi_{\rm NSF} [ g , g ](t)\, .
 \end{equation}
		 
\section{Proof of the theorem} \label{sec:strategy}
 
 Let us start by presenting the functional framework in which we shall develop our proof. 
Let $3/2 < \ell \le 2$ be fixed. Recall that $\alpha <1/4$ and~$\beta<1/2$ have been introduced in Theorem~\ref{theo:main}. 
In the following, we shall assume without loss of generality that~$\alpha>0$ and~$\beta >  \alpha (\ell-1/2) $.

We now
  define for any interval~$I$ of~$\R^+$ the space
\begin{equation} \label{def:XT}
	\mathcal X_I^\eps := \left\{ f \in  
\widetilde L^\infty_IH^{\ell}_x L^2_v \cap
L^2_I H^{\ell}_x H^{s,*}_v \;\Big|\; \| f \|_{\mathcal X_I^\eps} < +\infty \right\}    	
	\end{equation}
which we endow with the norm
 	\begin{equation} \label{def:norm_XT}
	\begin{aligned}
\| f \|_{\mathcal X^\eps_I} 
&:= \| f \|_{\widetilde L^\infty_I H^{\frac12}_x L^2_v} + \| \mathbf P_0 f \|_{L^2_I H^\frac32_x H^{s,*}_v} + \frac{1}{\sqrt \eps} \| \mathbf P_0^\perp f \|_{L^2_I H^\frac32_x H^{s,*}_v} \\
&\quad
+ \eps^{\beta} \Big(\| f \|_{\widetilde L^\infty_I H^{\ell}_x L^2_v}+   \| \mathbf P_0 f \|_{L^2_I H^{\ell}_x H^{s,*}_v} 
+ \frac{1}{\sqrt\eps} \| \mathbf P_0^\perp f \|_{L^2_I H^{\ell}_x H^{s,*}_v}\Big)
\, .
\end{aligned}
	\end{equation}
In the following we write~$\mathcal X_T^\eps:=\mathcal X_{[0,T]}^\eps$.

\begin{rem}\label{remark truncation}
If~$f=f(x,v)$ is a function in~$H^{\frac12}_x L^2_v$ and if~${\psi}$ is a smooth, compactly supported function on~$\R^3$, then   the sequence $f^\eps:=\psi (\eps^\alpha |D_x|) f$ goes to zero in~$\eps^\beta H^\ell_xL^2_v$ because of the assumption~$\beta > \alpha (\ell-1/2)$: indeed
	\[
 	\eps^\beta\|f^\eps\|_{ H^{\ell}_x L^2_v }
	\lesssim 
	\eps^{\beta-\alpha (\ell-\frac12)}\|f\|_{ H^{\frac12}_x L^2_v }\, .
	\]
\end{rem}

\medskip
Recall that we  consider well-prepared initial data~$g_{\rm in}$ in~$H^{\frac12}_x L^2_v$ and the associated maximal fluid solution $g \in \widetilde L^\infty_T H^{\frac12}_x L^2_v \cap L^2_T H^\frac32_x L^2_v$ for~$T<T^\star$, where the maximal life span~$T^\star>0$ {satisfies}
$$
\lim_{T \to T^\star} \|g\|_{L^2_T H^\frac32_x L^2_v} = \infty\, .
$$
This solution satisfies
\begin{equation}\label{eq:g_bound}
\| g \|_{\widetilde L^\infty_T H^{\frac12}_x L^2_v} + \| g \|_{L^2_T H^\frac32_x L^2_v} \lesssim \| { g_{\mathrm{in}} }\|_{H^{\frac12}_x L^2_v}\, ,
\end{equation}
{where the constant may depend on~$T^\star$ but is uniform if~$T^\star = \infty$ (see~\cite{GIP}). We refer for instance to~\cite{BCD, lemarie, MR3469428} for more on the Navier-Stokes equations.}
Note that as mentioned in Section~\ref{sect:notation}, actually~$g$ belongs also to~$\widetilde L^\infty_T H^{\frac12}_x H^{s,*}_v \cap L^2_T H^\frac32_x H^{s,*}_v $, with the {same bound}.

We then build a family of initial data~$f_{\rm in}^\eps$ to Equation~(\ref{eq:feps_intro})
such that on the one hand~$  \P_0 f^\eps_{\rm in}= \psi(\eps^\alpha |D_x|)g_{\rm in}$ for some smooth, compactly supported function~$\psi$, and on the other hand~$\P_0^\perp f_{\rm in}^\eps$ is arbitrary but   goes to zero in~$H^{\frac12}_xL^2_v$ while~$\eps^\beta \P_0^\perp f_{\rm in}^\eps$ goes to zero in~$H^{\ell}_xL^2_v$. Note that as pointed out in Remark~\ref{remark truncation},~$  \P_0 f^\eps_{\rm in}$ actually goes to $0$ in~$ \eps^\beta H^{\ell}_xL^2_v$ (since~$\beta> \alpha(\ell-1/2)$). Our goal is to prove that the solution~$f^\eps$  of~(\ref{eq:feps_intro}) with data~$f_{\rm in}^\eps$ converges to~$g$ as stated in Theorem~\ref{theo:main}, on the same life span as~$g$.

\medskip

The first step consists in replacing~$g$ by a smooth solution to~\eqref{eq:NSF} in the following way: let us define
$$
 g^\eps(t,x,v) := 
	\Big\{ \rho^\eps (t,x)  + u^\eps (t,x) \cdot v  +\theta^\eps (t,x) \frac{|v|^2-3}{2} \Big\} \mu^{\frac12}(v)
$$
where~$(\rho^\eps ,u^\eps ,\theta^\eps )$  solves~\eqref{eq:NSF} with the initial data $ \psi(\eps^\alpha |D_x|) ( \rho_{\rm in}, u_{\rm in},\theta_{\rm in})$. It is classical (see for instance~\cite[Proposition B.5]{GT}, and~\cite{cheminSIAM,BCD,lemarie} for more), that for~$\eps$ small enough,~$(\rho^\eps ,u^\eps ,\theta^\eps )$  belongs to~$\widetilde L^\infty_TH^{\frac12}_x L^2_v \cap L^2_TH^\frac32_x H^{s,*}_v $, and there holds
\begin{equation} \label{L2T LinftyT geps1}
\| g^\eps-g\|_{ \widetilde L^\infty_TH^{\frac12}_x L^2_v } + \| g^\eps-g\|_{L^2_TH^\frac32_x H^{s,*}_v  } \xrightarrow[\eps \to 0]{}  0 \, .
\end{equation}
Note that in particular
	\begin{equation} \label{L2T LinftyT geps2}
	\|g^\eps\|_{\widetilde L^\infty_T H^{\frac12}_x L^2_v} \xrightarrow[\eps \to 0]{} 
	\|g\|_{\widetilde L^\infty_T H^{\frac12}_x L^2_v}
	\quad \text{and} \quad 
	\|g^\eps\|_{L^2_T H^\frac32_x H^{s,*}_v} \xrightarrow[\eps \to 0]{} \|g\|_{L^2_T H^\frac32_x H^{s,*}_v}\,.
	\end{equation}
To prove Theorem~\ref{theo:main}, it thus suffices to prove that
$$
\| g^\eps-f^\eps\|_{\widetilde L^\infty_TH^{\frac12}_x L^2_v} + 
  \| g^\eps-f^\eps\|_{L^2_TH^\frac32_x H^{s,*}_v  } 
\xrightarrow[\eps \to 0]{} 0 \, .
$$
Note that by propagation of regularity  (see again~\cite[Proposition B.5]{GT}) there holds, for any~$m>1/2$,
\begin{equation}\label{propag regularity}
\begin{aligned}
	\| g^\eps\|_{ \widetilde L^\infty_TH^{m}_x L^2_v  }  +	\| g^\eps\|_{   L^2_TH^{m+1}_x H^{s,*}_v }&\lesssim \|  \P_0 f^\eps_{\rm in}\|_{H^{m}_x L^2_v} \exp \Big( C\|g^\eps\|_{L^2_TH^\frac32_x H^{s,*}_v  } ^2 \Big)
\\
& \lesssim \|  \P_0 f^\eps_{\rm in}\|_{H^{m}_x L^2_v} \exp \Big( C\|g\|_{L^2_TH^\frac32_x H^{s,*}_v  }^2 \Big)
\\
& \lesssim \|  \P_0 f^\eps_{\rm in}\|_{H^{m}_x L^2_v} \exp \Big( C\| { g_{\mathrm{in}}} \|_{H^{\frac12}_x L^2_v }^2 \Big)
	\end{aligned}
\end{equation}
due to~(\ref{L2T LinftyT geps2}) and \eqref{eq:g_bound}.
In   particular~$ g^\eps$ satisfies
	\begin{equation} \label{smallness geps}
	\begin{aligned}
	&	\| g^\eps\|_{ \widetilde L^\infty_TH^{m}_x L^2_v} + \| g^\eps\|_{   L^2_TH^{m+1}_x H^{s,*}_v }\\
	&\qquad \lesssim \eps^{ - \alpha(m-\frac12)} \|{g_{\rm in}}\|_{H^{\frac12}_x L^2_v} \exp \Big( C\| {g_{\rm in}}\|_{H^{\frac12}_x L^2_v }^2 \Big) \, .
	\end{aligned}
	\end{equation}
	By the standard interpolation inequality
	\begin{equation}\label{eq:interpolation}
	{\| h \|_{\widetilde L^4_T H^{n}_x} 
	\lesssim \| h \|_{\widetilde L^\infty_T H^{n-\frac12}_x}^{\frac12} \| h \|_{L^2_T H^{n+\frac12}_x}^{\frac12} \,, \qquad \forall  n \geq \frac12\,,}
	\end{equation}
one also has
\begin{equation}\label{estimateL4}
 \| g^\eps\|_{\widetilde   L^4_TH^{m+\frac12}_x H^{s,*}_v }\lesssim \eps^{ - \alpha(m-\frac12)} \| g_{\rm in}\|_{H^{\frac12}_x L^2_v} \exp \Big( C\|{g_{\rm in}} \|_{H^{\frac12}_x L^2_v }^2 \Big) \, .\end{equation}It is also worth recalling  that $g^\eps = \P_0 g^\eps$ so that $\|g^\eps(t,x,\cdot)\|_{H^{s,*}_v} \lesssim \|g^\eps(t,x,\cdot)\|_{L^2_v}$. 

In what follows, we shall look for a solution $f^\eps$ to \eqref{eq:feps_intro} under the form
	$
	f^\eps =  g^\eps + \delta^\eps.
	$
Since as recalled in~(\ref{Duhamelfluid})
	$$
	g^\eps(t) = U_{\rm NSF} (t)   \P_0 f^\eps_{\rm in} + \Psi_{\rm NSF} [ g^\eps, g^\eps](t)\, ,
	$$	
elementary algebraic computations lead to the following equation on~$\delta^{\eps}$:
	\begin{equation} \label{eq:deltaeps} 
	\begin{aligned}
	\delta^{\eps}(t) 
	&= \big(U^{\eps}(t) -U_{\rm NSF}(t) \big)\P_0f_{\mathrm{in}}^\eps  
	+ U^\eps(t) \P_0^\perp f_{\rm in}^\eps \\
	&\quad 
	+ \Psi^{\eps }[  g^\eps, g^\eps](t)-  \Psi_{\rm NSF} [ g^\eps, g^\eps](t) \\
	&\quad 
	+ 2\Psi^{\eps}[g^\eps,\delta^\eps](t)
	+ \Psi^\eps[\delta^\eps,\delta^\eps](t)\,.
	\end{aligned}
	\end{equation}

As we shall see, the main point is to be able to solve the equation on~$\delta^{\eps}$ although the initial data blows up  (in a controled way) as $\eps \to 0$. Our method of proof will enable us to prove that the equation has a unique solution on the same time interval as~$g^\eps$ hence as~$g$, at least for~$\eps$ small enough. In doing so we shall also prove that $\delta^{\eps}$ converge to~$0$ in~$\widetilde L^\infty_T H^{\frac12}_x L^2_v \cap 
	 L^2_T H^\frac32_x L^2_v.$

The method will rely on the following  fixed point lemma.  
\begin{lem} \label{lem:Picard}
There is a constant~$C_0>0$ such that the following holds. Let~$X$ be a Banach space, ${\mathcal L}$   a  continuous linear map from~$X$ to~$X$,  and ${\mathcal B}$   a bilinear map from~$X\times X$ to~$X$.
Let us define
	\[
	\|{\mathcal L}\|  := \sup_{\|x\|=1} \|{\mathcal L}x\|
	\quad\hbox{and}\quad
	\|{\mathcal B}\| := \sup_{\|x\|=\|y\|=1} \|{\mathcal B}(x,y)\| \, .
	\]
If~$\|{\mathcal L}\| <1$, then for any~$x_{0}$ in~$X$ such that
	\begin{equation}\label{small data}
	\|x_{0}\|_{X}< \frac{(1-\|{\mathcal L}\| )^2} {4\|{\mathcal B}\| } 
	\end{equation}
the equation
	$
	x=x_{0}+{\mathcal L}x+{\mathcal B}(x,x)
	$
has a unique solution in the ball of center~$0$ and radius~$\displaystyle \frac {1-\|{\mathcal L}\| }{2\|{\mathcal B}\|} $  and  there holds~$
	\|x\|\le C _0\|x_0\| \, .
	$
\end{lem}

In the next sections, we shall provide all the necessary estimates in order to implement this fixed-point argument to solve~\eqref{eq:deltaeps}, which we re-write in the following form:
	$$
	\delta^{\eps}(t) 
	= \mathcal D^{\eps}(t) +  \mathcal S ^{\eps}(t) +  \mathcal L^{\eps}[\delta^{\eps}](t) +  \Psi^{\eps}[\delta^{\eps},\delta^{\eps}]  (t) \, , 
	$$
where the data~$ \mathcal D^{\eps}$, source~$  \mathcal S ^{\eps}$ and linear~$  \mathcal L^{\eps}[\delta^{\eps}]$   terms are defined by
	\begin{equation}\label{eq:def:calDSL}
	\begin{aligned}
	\mathcal D^{\eps}(t)&:= \big(U^{\eps}(t) -U_{\rm NSF}(t) \big)\P_0f_{\mathrm{in}}^\eps + U^\eps(t) \P_0^\perp f_{\rm in}^\eps\, , \\
	\mathcal S ^{\eps}(t) & := \Psi^{\eps }[  g^\eps, g^\eps](t)-  \Psi_{\rm NSF} [ g^\eps, g^\eps](t)  \, , \\
	\mathcal L^{\eps}[\delta^{\eps}](t)   &:=2 \Psi^{\eps}[g^\eps,\delta^{\eps}](t)     \, . 	
	\end{aligned}
	\end{equation}
 Sections~\ref{sec:some results} and~\ref{sec:deltaeps} will be devoted to the proof of the following result.  
\begin{prop}\label{estimates on delta eps}
Under the assumptions of  Theorem~{\rm\ref{theo:main}}, the following holds.
\begin{enumerate}

\item 
For any $t \in (0,T)$ there holds
$$
\| U^\eps(\cdot - t) F(t) \|_{\mathcal X^\eps_{[t,T]}}
\lesssim \| F \|_{\mathcal X^\eps_{[0,t]}}.
$$

\item The data term goes to zero globally in time: 
	$$
	\begin{aligned}
	 \|\mathcal D^{\eps}\|_{\mathcal X^\eps_\infty} 
	\lesssim
 \eps^{\frac12-\alpha} \|{ g_{\rm in}}\|_{H^{\frac12}_x L^2_v}
+ \|\P_0^\perp f_{\rm in}^\eps\|_{H^{\frac12}_xL^2_v} 
+  \eps^\beta \|\P_0^\perp f_{\rm in}^\eps\|_{H^{\ell}_xL^2_v}
	\xrightarrow[\eps \to 0]{}  0  \, .
	\end{aligned}
	$$

\item The source term goes to zero in~$\mathcal X^\eps_T$: there 
exists a nonnegative increasing function~$\Phi$ 
such that 
	$$
	 \|\mathcal S^{\eps}\| _{\mathcal X^\eps_T}  \leq \eps^{\frac12-2\alpha} \, \Phi \Big( \|{ g_{\rm in}}\|_{H^{\frac12}_x L^2_v} \Big) \xrightarrow[\eps \to 0]{}  0   \, .	
	 $$
\item The linear term satisfies the following continuity estimate for~$\eps $ small enough: for all intervals~$I$, 
$$
\| \mathcal L^{\eps}[f] \|_{\mathcal X^\eps_I}
\lesssim \|f \|_{\mathcal X^\eps_I} 
\Big(\| g^\eps \|_{\widetilde L^4_I H^{1}_x L^2_v} + \| g^\eps \|_{L^2_I H^\frac32_x L^2_v} + \eps^{\beta} \|g^\eps \|_{\widetilde L^\infty_I H^{\ell}_x L^2_v}
+ \eps^{\beta} \|g^\eps \|_{L^2_I H^{\ell}_x L^2_v} \Big)	\, .
$$
	
 \item The nonlinear term satisfies the following continuity estimate: for all intervals~$I$, 
	$$
	\| \Psi^{\eps}[f_1,f_2] \|_{\mathcal X^\eps_I}\lesssim \|f_1\|_{\mathcal X^\eps_I}  \|f_2\|_{\mathcal X^\eps_I}\, .
	$$
\end{enumerate}
\end{prop}

\smallskip

Let us investigate how  Proposition~\ref{estimates on delta eps} ensures the wellposedness of~(\ref{eq:deltaeps}) in~$\mathcal X^\eps_T$ and the convergence of~$\delta^\eps$ to zero, thus proving Theorem~\ref{theo:main}.

\begin{proof}[Proof of Theorem~{\rm\ref{theo:main}}]
We shall check that~(\ref{eq:deltaeps}) takes the form required by Lemma~\ref{lem:Picard}. 
Thanks to Proposition~\ref{estimates on delta eps}--(2),(3) and the assumptions of Theorem~\ref{theo:main} we have
$$
\|\mathcal D^{\eps}\| _{\mathcal X^\eps_\infty} +  \|\mathcal S^{\eps}\| _{\mathcal X^\eps_T}\xrightarrow[\eps \to 0]{} 0 \, .
$$
Due to Proposition~\ref{estimates on delta eps}--(5), (\ref{small data}) will be satisfied as soon as we have a hold on the continuity constant on~$\mathcal L^\eps$: we need  the linear operator~$ \mathcal L^{\eps}$ to be a  a contraction in~$\mathcal X^\eps_T$. As can be seen from Proposition~\ref{estimates on delta eps}--(4) along with~(\ref{smallness geps}) and~(\ref{estimateL4}), 
for that to be the case one needs~${  g_{\rm in}}$ to be small, which we do not assume here. 
  
  In order to get around this difficulty, we shall apply  Lemma~\ref{lem:Picard} iteratively on small time intervals. Note that due to Proposition~\ref{estimates on delta eps}--(4) and~\eqref{smallness geps}, there is a constant~$C>0$ and~$\eps_0>0$  such that for all~$\eps \leq \eps_0$
  $$
  \| \mathcal L^{\eps}[f] \|_{\mathcal X^\eps_T}\leq C \|f \|_{\mathcal X^\eps_T} \Big(\frac1{4C} + \| g \|_{\widetilde L^4_T H^{1}_x L^2_v} + \| g \|_{L^2_T H^\frac32_x L^2_v}  	   \Big)\, .
  $$
Now 
thanks to~(\ref{eq:g_bound}) and~(\ref{eq:interpolation})
there exists~$K>0$ and times~$t_1:=0<t_2<\dots<t_K:=T$ such that
  $$
  \forall 1 \leq i \leq K-1
 \, , \quad  \| g \|_{\widetilde L^4([t_i,t_{i+1} ]; H^{1}_x L^2_v)} 
 + \| g \|_{L^2\big([t_i,t_{i+1} ]; H^\frac32_x L^2_v\big)} \leq \frac1{4C}\, \cdotp$$
 Then in particular
\begin{equation}\label{L contraction 2}
  \| \mathcal L^{\eps}[f] \|_{\mathcal X^\eps_{t_2}} \le \frac 12  \|f \|_{\mathcal X^\eps_{t_2}}\, \cdotp
 \end{equation}
Applying Lemma~\ref{lem:Picard} on~$[0,t_2]$   implies that there is a unique solution~$\delta^\eps$ to~(\ref{eq:deltaeps})  in~$\mathcal X^\eps_{t_2}$, which satisfies
\begin{equation}\label{bound t2}
\|\delta^\eps\|_{\mathcal X^\eps_{t_2}} \leq C_0 \Big(D_{\rm in}^\eps+  \|\mathcal S^{\eps}\| _{\mathcal X^\eps_{t_2}}\Big) \xrightarrow[\eps \to 0]{}0\, ,
\end{equation}
with thanks to Proposition~\ref{estimates on delta eps}--(2)
\begin{equation}\label{limit data3}
D_{\rm in}^\eps:= \eps^{\frac12-\alpha}   \|{g_{\rm in}}\|_{H^{\frac12}_x L^2_v} + \|\P_0^\perp f_{\rm in}^\eps\|_{H^{\frac12}_x L^2_v} +  \eps^{\beta} \|\P_0^\perp f_{\rm in}^\eps\|_{H^{\ell}_xL^2_v} \, .
	\end{equation}
	Then we solve~(\ref{eq:deltaeps}) on~$[t_2,t_3]$. We recall that~(\ref{eq:deltaeps})  writes
$$
\forall t \in [t_2,t_3] \, , \quad \delta^{\eps}(t) = \mathcal D^{\eps}( t) +   \mathcal S ^{\eps } ( t) +  \mathcal L^{\eps}[\delta^{\eps}](t) +  \Psi^{\eps}[\delta^{\eps},\delta^{\eps}]  (t) \, , 
$$
with $\mathcal D^{\eps}$, $\mathcal S^{\eps}$ and $\mathcal L^{\eps}$ defined in \eqref{eq:def:calDSL}.
	We want to recast this equation in a form suited to a fixed point on~$[t_2,t_3]$. 
According to~(\ref{defPsieps}) and since~$U^\eps$ is a semigroup, 
we can write for all~$t \geq t_2$
$$
\begin{aligned}
	\Psi^\eps[f,g] (t) &= \frac{1}{\eps} U^\eps(t-t_2) \int_0^{t_2} U^\eps(t_2-t') \Gamma_{\mathrm{sym}}(f(t') , g(t')) \, \d t' \\
	&\qquad \qquad \qquad \qquad \qquad 
	+  \frac{1}{\eps}  \int_{t_2} ^{t} U^\eps(t-t') \Gamma_{\mathrm{sym}}(f(t') , g(t')) \, \d t' \\
	&=:   \frac{1}{\eps} U^\eps(t-t_2) \int_0^{t_2} U^\eps(t_2-t') \Gamma_{\mathrm{sym}}(f(t') , g(t')) \, \d t' + \Psi^\eps[f,g] (t_2;t)
\, .
\end{aligned}$$
We also define the operator
$$
 \mathcal L^{\eps}[f](t_2;t):=\Psi^{\eps}[ g^\eps,f] (t_2;t)
$$
and we set 
$$
\mathcal D ^\eps_{2}(t):=\mathcal D ^\eps(t) - U^{\eps}(t-t_2)\mathcal D ^\eps(t_2)
 $$
and
$$
\mathcal S ^\eps_{2}(t):= \mathcal S ^\eps(t)  - U^{\eps}(t-t_2)\mathcal S ^\eps(t_2) \, . $$
Then~(\ref{eq:deltaeps}) can be recast on~$[t_2,t_3]$ as follows:
$$
 \delta^{\eps}(t) =U^\eps(t-t_2)  \delta^{\eps}(t_2) +\mathcal D ^\eps_{2}(t)+   \mathcal S ^\eps_{2}(t)   + \mathcal L^{\eps}[\delta^{\eps}](t_2;t) +  \Psi^{\eps}[\delta^{\eps},\delta^{\eps}]  (t_2;t) \, .
$$
Thanks to Proposition~\ref{estimates on delta eps}--(1), (2) and~(3),~$ \mathcal D^{\eps}_2   $  and~$\mathcal S^{\eps}_2$   go  to zero in~$\mathcal X^{\eps}_{[t_2,t_3]}$,  with for some universal constant~$C_1$
$$
\| \mathcal D^{\eps}_2\|_{\mathcal X^\eps_{[t_2,t_3]}}  \leq  C_1D_{\rm in}^\eps
$$
with notation~(\ref{limit data3}), and $$
\| \mathcal S^{\eps}_2\|_{\mathcal X^\eps_{[t_2,t_3]}}  \leq C_1  \eps^{\frac12-2\alpha} \Phi \Big(\|{ g_{\rm in}}\|_{H^{\frac12}_xL^2_v} \Big) 	  .
$$
The linear operator~$\mathcal L^{\eps}[\delta^{\eps}](t_2;t) $ is dealt with exactly as above to produce similarly to~(\ref{L contraction 2}), for~$\eps$ small enough,
$$
 \| \mathcal L^{\eps}[f] \|_{\mathcal X^\eps_{[t_2,t_3]}} \le \frac 12  \|f \|_{\mathcal X^\eps_{[t_2,t_3]}}\, \cdotp
$$
Finally thanks to Proposition~\ref{estimates on delta eps}--(1) and~(\ref{bound t2}), we have for some universal constant~$C_2>0$ that
$$
	\begin{aligned}
  \| U^\eps(t-t_2)  \delta^{\eps}(t_2) \|_{\mathcal X^\eps_{[t_2,t_3]}} & \lesssim \|\mathcal D^{\eps}\| _{\mathcal X^\eps_\infty} +  \|\mathcal S^{\eps}\| _{\mathcal X^\eps_{t_2}}  \\
 & \leq C_2 \Big[ D_{\rm in}^\eps+\eps^{\frac12-2\alpha}   \Phi \left(\|{ g_{\rm in}}\|_{H^{\frac12}_xL^2_v} \right) \Big]  \, .	\end{aligned}
$$
We can therefore apply Lemma~\ref{lem:Picard} which implies that $$
 \begin{aligned}\|\delta^\eps\|_{\mathcal X^\eps_{[t_2,t_3]}} 
&\leq C_0 
 \Big(\|U^\eps(\cdot-t_2)  \delta^{\eps}(t_2)\|_{\mathcal X^\eps_{[t_2,t_3]}}  +\|\mathcal D ^\eps_{2} \|_{\mathcal X^\eps_{[t_2,t_3]}} + \|  \mathcal S ^\eps_{2} \|_{\mathcal X^\eps_{[t_2,t_3]}} \Big)\\
 &\leq C_0 
  (C_1+C_2) \Big[ D_{\rm in}^\eps+ \eps^{\frac12-2\alpha}   \Phi \left(\|{ g_{\rm in}}\|_{H^{\frac12}_xL^2_v} \right) \Big] . 
\end{aligned}
$$
 Iterating this argument~$K$ times and  noticing that~$K$ is of the order of
 $$K(g):= \|g\|_{\widetilde L^4_TH^{1}_xL^2_v} + \|g\|_{L^2_TH^\frac32_x L^2_v} \, , $$we find that there  is a unique solution $\delta^\eps \in \mathcal X^\eps_T$ to~(\ref{eq:deltaeps})  on~$[0,T]$ which satisfies for some universal constant~$C\geq 2$
$$
 \|\delta^\eps\|_{\mathcal X^\eps_{T}} \lesssim C^{K(g)} \left[ D_{\rm in}^\eps+ \eps^{\frac12-2\alpha}   \Phi \left(\|{ g_{\rm in}}\|_{H^{\frac12}_xL^2_v} \right) \right]
 \xrightarrow[\eps \to 0]{} 0 \, .  
$$
Theorem~\ref{theo:main} is proved. 
\end{proof}

\section{Some   results on the operators $U^\eps$ and $\Psi^\eps$}\label{sec:some results}
In this section, we provide useful continuity estimates on~$U^\eps$,~$\Psi^\eps$ and~$\Gamma$.

\subsection{Estimates on $U^\eps$ and $\Psi^\eps$}
The first series of estimates (Propositions~\ref{prop:estimate_Ueps},~\ref{prop:estimate_Uepssharp}, \ref{prop:estimate_Uepsflat} and Corollary~\ref{cor:estimate_Psieps}) are very close to the ones established in~\cite{CC} (and in~\cite{Gervais-Lods}) and are based on hypocoercive energy estimates (see Appendix~\ref{app:hypocoercivity} for a presentation of hypocoercivity results). Since the functional framework is a little different, we reformulate them in our functional setting. Some key elements of proofs are provided in Appendix~\ref{app:hypocoercivity}.

\begin{prop}\label{prop:estimate_Ueps}
Let $m \ge 0$ and~$T>0$. There holds:
\begin{enumerate}

\item Let $f \in H^m_x L^2_v$ and assume $f$ verifies~\eqref{eq:normalization}. Then
	\begin{equation*}
	\| U^\eps (\cdot) f \|_{\widetilde L^\infty_t H^m_x L^2_v } 
		+ \left\|  \P_0 U^\eps (\cdot) f \right\|_{L^2_t H^m_x L^2_v} +\frac{1}{\eps} \| \P_0^\perp  U^\eps (\cdot) f \|_{L^2_t H^m_x H^{s,*}_v} 
	\lesssim \| f \|_{H^m_x L^2_v}\, ,
	\end{equation*}
and moreover $U^\eps (t) f$ verifies~\eqref{eq:normalization} for all $t \ge 0$.

\item Let $S=S(t,x,v)$ satisfy $\P_0 S = 0$ and $S \in L^2_T H^m_x (H^{s,*}_v)'$, then for any~$t \leq T$,
	\begin{multline*}
	\left\| \int_0^t U^\eps (t-t') S(t') \, \d t' \right\|_{\widetilde L^\infty_T H^m_x L^2_v } 
	+  \left\|   \P_0 \int_0^t U^\eps (t-t') S(t') \, \d t' \right\|_{L^2_T H^m_x L^2_v} 
\\
	+\frac{1}{\eps} \left\| \P_0^\perp \int_0^t U^\eps (t-t') S(t') \, \d t' \right\|_{L^2_T H^m_x H^{s,*}_v}	\lesssim \eps \| S \|_{L^2_T H^m_x (H^{s,*}_v)'}\,.
	\end{multline*}

\end{enumerate}

\end{prop}
From \cite[Lemmas 4.8 and 4.9]{Gervais-Lods} we also have estimates for the semi-group $U^{\eps,\sharp}$.

\begin{prop}\label{prop:estimate_Uepssharp}
Let $m \ge 0$ and $T>0$. There holds: 
\begin{enumerate}

\item Let $f \in H^m_x L^2_v$, then
$$
\| U^{\eps,\sharp} (\cdot) f \|_{\widetilde L^\infty_T H^m_x L^2_v } 
+\frac{1}{\eps} \|   U^{\eps,\sharp} (\cdot) f \|_{L^2_T H^m_x H^{s,*}_v}  
\lesssim \| f \|_{H^m_x L^2_v}\, .
$$

\item Let $S=S(t,x,v)$ satisfy $S \in L^2_T H^m_x (H^{s,*}_v)'$. Then for any $t \leq T$,
	\begin{multline*}
	\left\| \int_0^t U^{\eps,\sharp} (t-t') S(t') \, \d t' \right\|_{\widetilde L^\infty_T H^m_x L^2_v } 
	+\frac{1}{\eps} \left\|  \int_0^t U^{\eps,\sharp} (t-t') S(t') \, \d t' \right\|_{L^2_T H^m_x H^{s,*}_v} \\ 
	\lesssim \eps \| S \|_{L^2_T H^m_x (H^{s,*}_v)'}\,.
	\end{multline*}

\end{enumerate}
\end{prop}
From the two previous propositions, since $\Gamma_{\mathrm{sym}}$ is such that $\P_0 \Gamma_{\mathrm{sym}} = 0$, it is straightforward to deduce the following result.
\begin{cor} \label{cor:estimate_Psieps}
Consider $m \geq 0$, $f_1$, $f_2$ such that $\Gamma_{\mathrm{sym}}(f_1,f_2) \in L^2_T H^m_x (H^{s,*}_v)'$ for some given~$T>0$. Then, there holds:
	\begin{multline*}
	\|\Psi^\eps[f_1,f_2]\|_{\widetilde L^\infty_T H^m_x L^2_v} 
	+ \left\|  \P_0 \Psi^\eps[f_1,f_2] \right\|_{L^2_T H^m_x L^2_v} 
\\
	+\frac{1}{\eps} \left\| \P_0^\perp \Psi^\eps[f_1,f_2] \right\|_{L^2_T H^m_x H^{s,*}_v}  	\lesssim  \| \Gamma_{\mathrm{sym}}(f_1,f_2) \|_{L^2_T H^m_x (H^{s,*}_v)'}
	\end{multline*}
and
	\[
	\|\Psi^{\eps,\sharp}[f_1,f_2]\|_{\widetilde L^\infty_T H^m_x L^2_v} 
	+\frac{1}{\eps} \| \Psi^{\eps,\sharp}[f_1,f_2] \|_{L^2_T H^m_x H^{s,*}_v} 
	\lesssim  \| \Gamma_{\mathrm{sym}}(f_1,f_2) \|_{L^2_T H^m_x (H^{s,*}_v)'}\,.
	\]
\end{cor}
{The  above statements show that there is some kind of smoothing effet in the velocity variable after time integration. However there  is no such effect in the space variable in general, except when it comes to the operator~$ U^{\eps,\flat} $, as shown in the following}
 straighforward estimate for $U^{\eps,\flat}$.
\begin{prop}\label{prop:estimate_Uepsflat}
Let $m \ge 0$. For any $f \in H^m_x L^2_v$ there holds
$$
\|   U^{\eps,\flat} (\cdot) f \|_{L^2_t H^{m+1}_x H^{s,*}_v}  
\lesssim \| f \|_{H^m_x L^2_v}\, .
$$
\end{prop}

  \subsection{Refined estimates on $\Psi^\eps$}
We recall that as introduced in Section~\ref{sec:prelim}
\begin{equation}\label{decomposition Psieps 111}
\Psi^{\eps}   = \Psi^{\eps,\flat}  
+\Psi^{\eps,\sharp}  \ .
\end{equation}
In what follows, we give estimates on~$\Psi^{\eps,\flat}  $ and~$\Psi^{\eps,\sharp}$.

\begin{prop} \label{prop:Psiepsflat}
Consider $T>0$ and~$m \in \R$. For any smooth enough functions $f_1$ and~$f_2$, we have
\begin{equation} \label{eq:estimate_Psiepsflat_Linfty}
\begin{aligned}
\big\| \Psi^{\eps,\flat} [f_1,f_2]\big\|_{\widetilde L^\infty_T H^{m}_x L^2_v}
\lesssim \| \Gamma_{\mathrm{sym}}(f_1,f_2) \|_{\widetilde L^\infty_T H^{m-1}_x L^2_v} \, ,
\end{aligned}
\end{equation}
\begin{equation} \label{eq:estimate_Psiepsflat_Linfty2}
\begin{aligned}
\big\| \Psi^{\eps,\flat} [f_1,f_2]\big\|_{\widetilde L^\infty_T H^{m}_x L^2_v}
\lesssim \| \Gamma_{\mathrm{sym}}(f_1,f_2) \|_{\widetilde L^4_T H^{m-\frac12}_x L^2_v} \, ,
\end{aligned}
\end{equation}
and also
\begin{equation} \label{eq:estimate_Psiepsflat1}
\begin{aligned}
\big\| \Psi^{\eps,\flat} [f_1,f_2]\big\|_{L^2_T H^{m}_x H^{s,*}_v}
\lesssim \| \Gamma_{\mathrm{sym}}(f_1,f_2) \|_{L^2_T H^{m-1}_x (H^{s,*}_v)'} \, .
\end{aligned}
\end{equation}

\end{prop}

\begin{proof}
Recalling~(\ref{decomposition Psieps 12}), for any $ k \in \Z^3$, there holds
	\begin{align*}
	&\widehat{\Psi^{\eps ,\flat}} [f_1,f_2](t, k) \\
	&\qquad
	 =\frac1\eps\chi \Big(\frac{\eps | k|}\kappa\Big)
	\sum_{\star\in \{\rm NS, heat,wave\pm\}}\int_0^t  e^{ \lambda_\star( \eps  k) \frac{t-t'}{\eps^2} }
	\mathcal P_\star(\eps  k) \widehat \Gamma_{\mathrm{sym}} \big(f_1(t'),f_2(t')\big) ( k) \, \d t' \,.
	\end{align*}
	Due to the form~(\ref{NS heat eigenvalues})-(\ref{wave eigenvalues}) of~$\lambda_\star$ and to the fact that~$\mathbf P_0 \Gamma_{\mathrm{sym}} = 0$, there is a constant~$\lambda_1>0$ such that
	$$
 \begin{aligned}
&\Big \|	 \widehat{\Psi^{\eps ,\flat}} [f_1,f_2](t, k) 
\Big \|_{L^2_v }\\
&\qquad \lesssim  | k|  \chi \Big(\frac{\eps | k|}\kappa\Big)
\int_0^t  e^{-  \lambda_1 | k|^2  {(t-t')} }\Big \|\mathcal P^1\Big({ k \over | k|}\Big) \widehat \Gamma_{\mathrm{sym}} (f_1,f_2) (t', k) \Big \|_{L^2_v}\, \d t' \\
&\qquad \quad  + \eps  | k|^2  \chi \Big(\frac{\eps | k|}\kappa\Big)
\int_0^t  e^{-  \lambda_1 | k|^2  {(t-t')} }\Big \|\mathcal P^2(\eps k) \widehat \Gamma_{\mathrm{sym}} (f_1,f_2) (t', k) \Big \|_{L^2_v}\, \d t'\\
&\qquad
\lesssim  | k|  \chi \Big(\frac{\eps | k|}\kappa\Big)
\int_0^t  e^{-  \lambda_1 | k|^2  {(t-t')} }\Big \|\mathcal P^1\Big({ k \over | k|}\Big) \widehat \Gamma_{\mathrm{sym}} (f_1,f_2) (t', k) \Big \|_{L^2_v }\, \d t' \\
&\qquad \quad  + | k|  \chi \Big(\frac{\eps | k|}\kappa\Big)
\int_0^t  e^{-  \lambda_1 | k|^2  {(t-t')} }\Big \|\mathcal P^2(\eps k) \widehat \Gamma_{\mathrm{sym}} (f_1,f_2) (t', k) \Big \|_{L^2_v }\, \d t'
	 \end{aligned}
$$
where we used that $\eps  k$ lies in a compact set to get the last inequality and where~$\mathcal P^1$ and~$\mathcal P^2$ are   bounded from $L^2_v$ into~$L^2_v $ uniformly in $\eps | k| \leq \kappa$. We then have 
	\begin{align*}
	\Big \|\mathcal P^{1}\Big({ k \over | k|}\Big) 
	\widehat \Gamma_{\mathrm{sym}} (f_1, f_2) (t', k) \Big \|_{L^2_v} + \Big \|\mathcal P^{2}(\eps k) 
	\widehat \Gamma_{\mathrm{sym}} (f_1, f_2) (t', k) \Big \|_{L^2_v} \\
	\lesssim  \big \| \widehat \Gamma_{\mathrm{sym}} (f_1, f_2) (t', k,\cdot) \big \|_{L^2_v }\, .
	\end{align*}
We denote $A(t',k) :=  \| \widehat \Gamma_{\mathrm{sym}} (f_1, f_2) (t', k,\cdot) \|_{L^2_v }$, and we use the fact that Young's inequality in time implies~$L^1_T \star L^\infty_T\subset L^\infty_T$ to estimate
	\[
	\| \widehat{\Psi^{\eps,\flat}} [f_1,f_2] (k) \|_{L^\infty_T  L^2_v}
	\lesssim 
	\left\|  \int_0^t  |k|^2 e^{-\lambda_1(t-t') | k|^2}  
	    | k|^{-1} A(t',k,\cdot)\, \d t'  \right\|_{L^\infty_T} 
	 \lesssim  \big\| |k|^{-1} A(\cdot,k) \big\|_{L^\infty_T}\,.
	\]
Therefore we obtain
	\[
	\| \Psi^{\eps,\flat} [f_1,f_2] \|_{\widetilde L^\infty_T H^{m}_x L^2_v}
	\lesssim 
	\left\|    
	 \langle   k\rangle ^{m-1} A  \right\|_{\ell^2_ k L^\infty_T} 
	\lesssim  \| \Gamma_{\mathrm{sym}} (f_1, f_2)  \|_{\widetilde L^\infty_T H^{m-1}_x L^2_v }	\,,
	\]
which concludes the proof of \eqref{eq:estimate_Psiepsflat_Linfty}. Using instead Young's inequality in time~$L^{4/3}_T \star L^4_T\subset L^\infty_T$, we also obtain 
	\[
	\| \widehat{\Psi^{\eps,\flat}} [f_1,f_2] (k) \|_{L^\infty_T  L^2_v}
	\lesssim 
	\left\|  \int_0^t  |k|^{\frac32} e^{-\lambda_1(t-t') | k|^2}  
	    | k|^{-\frac12} A(t',k,\cdot)\, \d t'  \right\|_{L^\infty_T} 
	 \lesssim  \big\| |k|^{-\frac12} A(\cdot,k) \big\|_{L^4_T}\,,
	\]
from which we deduce \eqref{eq:estimate_Psiepsflat_Linfty2} arguing as before.

For the $L^2_T$ estimate \eqref{eq:estimate_Psiepsflat1}, we observe that $\mathcal P^1$ and~$\mathcal P^2$ are   bounded from $(H^{s,*}_v)' $ into~$H^{s,*}_v $ uniformly in $\eps | k| \leq \kappa$. Therefore we obtain, denoting $B(t',k) :=  \| \widehat \Gamma_{\mathrm{sym}} (f_1, f_2) (t', k,\cdot) \|_{(H^{s,*}_v)' }$ and using Young's inequality in time~$L^1_T \star L^2_T\subset L^2_T$, that
	\[
	\| \widehat{\Psi^{\eps,\flat}} [f_1,f_2] (k) \|_{L^2_T  H^{s,*}_v}
	\lesssim 
	\left\|  \int_0^t  |k|^2 e^{-\lambda_1(t-t') | k|^2}  
	    | k|^{-1} B(t',k,\cdot)\, \d t'  \right\|_{L^2_T} 
	 \lesssim  \big\| |k|^{-1} B(\cdot,k) \big\|_{L^2_T}\,.
	\]
We then conclude the proof of \eqref{eq:estimate_Psiepsflat1} by arguing as before.  Proposition~\ref{prop:Psiepsflat}
is proved.
\end{proof}
We now give an estimate on $\Psi^{\eps,\sharp}$ that will be useful when both entries are macroscopic. 
\begin{prop} \label{prop:Psiepssharp}
Consider $T>0$ and $m \in \R $. For any smooth enough functions $f_1$ and~$f_2$ we have:
	\[
	\big\| \Psi^{\eps,\sharp} [f_1,  f_2] \big\|_{\widetilde L^\infty_T H^{m}_x L^2_v}
	\lesssim \eps \| \Gamma_{\mathrm{sym}} ( f_1 ,  f_2) \|_{\widetilde L^\infty_T H^m_x L^2_v} \,. 	
	\]
\end{prop}
\begin{proof}
We first write for any $ k \in \Z^3$,
	\begin{equation*}
	\Big\|\widehat \Psi^{\eps,\sharp} [ f_1,  f_2] (k)\Big\|_{L^\infty_T L^2_v}
	\lesssim
	\frac{1}{\eps} 
	\left\| \int_0^t \Big\|\widehat U^{\eps,\sharp}(t-t', k) \widehat \Gamma_{\mathrm{sym}}( f_1,  f_2 )(t', k) \Big\|_{L^2_v} \, \d t' \right\|_{L^\infty_T}  \,.
	\end{equation*}
Using then~\eqref{eq:decayUepssharp}, we deduce that 
	\begin{equation*}
	\Big\|\widehat \Psi^{\eps,\sharp} [ f_1,  f_2] (k)\Big\|_{L^\infty_T L^2_v}
	\lesssim
	\frac{1}{\eps}
	\left\| \int_0^t e^{-\lambda_0 \frac{t-t'}{\eps^2}} \Big\|\widehat \Gamma_{\mathrm{sym}} ( f_1,  f_2)(t', k)\Big\|_{L^2_v} \, \d t' \right\|_{L^\infty_T} \,,
	\end{equation*}
and thus Young's inequality in time yields
	\[
	\Big\|\widehat \Psi^{\eps,\sharp}[ f_1,   f_2]( k)\Big\|_{L^\infty_T L^2_v} \\
	\lesssim \eps
	 \Big\|\widehat \Gamma_{\mathrm{sym}} ( f_1,  f_2)(k)\Big\|_{L^\infty_T L^2_v}\,. 
	\] 
This concludes the proof of Proposition~\ref{prop:Psiepssharp}. \end{proof}

\subsection{Nonlinear estimates} 
We now provide nonlinear estimates that are central to estimate the nonlinear collisional operator $\Gamma$ in various functional spaces.
It is well-known   (see~\cite{MR2013332} for cutoff Boltzmann, \cite{GS,AMUXY5} for non-cutoff Boltzmann, \cite{MR1946444} for Landau) that
	$$
	\big | \la \Gamma(f_1,f_2) , f_3 \ra_{L^2_v}\big | \lesssim \| f_1 \|_{L^2_v} \| f_2 \|_{H^{s,*}_v} \| f_3 \|_{H^{s,*}_v}\,,
	$$
from which we  obtain
	\begin{equation}\label{eq:nonlinear_Hs*v'}
	\| \Gamma(f_1,f_2) \|_{(H^{s,*}_v)'} 
	:= \sup_{\| \phi \|_{H^{s,*}_v} \le 1} \la \Gamma(f_1,f_2), \phi \ra_{L^2_v} \lesssim  \| f_1 \|_{L^2_v} \| f_2 \|_{H^{s,*}_v}\,.
	\end{equation}

\begin{prop}\label{prop:nonlinear}
Let $m\ge 0$. For any $r_1,r_2 \neq  3/2 $, any~$ p_1, q_1 , p_2 , q_2 \in [1,\infty]$ that are such that~$1/p_1+1/q_1 = 1/p_2+1/q_2 = 1/2$, and any smooth enough functions $f_1$, $f_2$ there holds: 
$$
\begin{aligned}
\| \Gamma (f_1,f_2) \|_{L^2_T H^{m}_x (H^{s,*}_v)'} 
&\lesssim 
\| f_1 \|_{\widetilde L^{p_1}_T H^{m+(\frac32-r_1)_+}_x L^2_v} \| f_2 \|_{\widetilde L^{q_1}_T  H^{r_1}_x H^{s,*}_v} \\
&\quad
+\| f_1 \|_{\widetilde L^{p_2}_T  H^{r_2}_x L^2_v} \| f_2 \|_{\widetilde L^{q_2}_T H^{m+(\frac32-r_2)_+}_x H^{s,*}_v} \,	.
\end{aligned}
$$
\end{prop}

\begin{proof}
To simplify we write $F_1(t,k) = \| \hat f_1(t,k, \cdot) \|_{L^2_v}$ and $F_2(t,k) = \| \hat f_2(t,k,\cdot) \|_{H^{s,*}_v}$.
By~\eqref{eq:nonlinear_Hs*v'} we have, for any $k \in \Z^3$,
$$
\begin{aligned}
\| \widehat \Gamma (f_1,f_2) (k) \|_{L^2_T (H^{s,*}_v)'}
&\lesssim \Big\{ \int_0^T \Big( \sum_{n \in \Z^3} F_1(t,k-n) F_2(t,n) \Big)^2 \d t \Big\}^{\frac12} ,
\end{aligned}
$$
and applying Minkowski's inequality yields
$$
\| \widehat \Gamma (f_1,f_2) (k) \|_{L^2_T (H^{s,*}_v)'}
\lesssim \sum_{n \in \Z^3}
 \left\{ \int_0^T  |F_1(t,k-n)|^2 |F_2(t,n)|^2 \,  \d t \right\}^{\frac12} .
$$
We now follow \cite[Lemma~7.3]{MR3469428}. We first split
$$
\| \widehat \Gamma (f_1,f_2) (k) \|_{L^2_T (H^{s,*}_v)'}
\lesssim I_1(k) + I_2(k)
$$
with
$$
I_1(k) = \sum_{n \in \Z^3}
 \mathbf{1}_{|n| < |k-n|}
 \left\{ \int_0^T  |F_1(t,k-n)|^2 |F_2(t,n)|^2 \,  \d t \right\}^{\frac12}
$$
and
$$
I_2(k) = \sum_{n \in \Z^3}
 \mathbf{1}_{|n| \le |k-n|}
\left\{ \int_0^T  |F_1(t,n)|^2 |F_2(t,k-n)|^2 \,  \d t \right\}^{\frac12}.
$$
We now estimate the term $I_1$. Thanks to H\"older's inequality in time, we obtain
\begin{equation}\label{eq:I1_FLp_GLq}
I_1 (k)
\lesssim \sum_{n \in \Z^3}
 \mathbf{1}_{|n| < |k-n|}
 \| F_1(\cdot,k-n) \|_{L^{p_1}_T} \| F_2(\cdot, n) \|_{L^{q_1}_T} , 
\end{equation}
where $1/p_1+1/q_1 = 1/2$. To simply notation we introduce $\mathcal{F}_1 (k) = \| F_1(\cdot,k) \|_{L^{p_1}_T} $ and~$\mathcal{F}_2(k) = \| F_2(\cdot, k) \|_{L^{q_1}_T}$. By the Cauchy-Schwarz inequality it follows that
$$
\begin{aligned}
I_1(k) \lesssim \| \la \cdot \ra^{r_1} \mathcal{F}_2 \|_{\ell^2(\Z^3)} \left\{  \sum_{n \in \Z^3}
 \mathbf{1}_{|n| < |k-n|}  \la n \ra^{-2r_1} \mathcal{F}_1(k-n)^2 \right\}^{\frac12}
\end{aligned}
$$
where we recall that $r_1 \neq 3/2$. 
Multypliying $I_1(k)$ by $\la k \ra^{m}$ then taking the square and summing it gives
$$
\begin{aligned}
\sum_{k \in \Z^3} \la k \ra^{2m} I_1(k)^2
\lesssim \| \la \cdot \ra^{r_1} \mathcal{F}_2 \|_{\ell^2(\Z^3)}^2   \sum_{k \in \Z^3}  \sum_{n \in \Z^3}  
 \mathbf{1}_{|n| < |k-n|} \la k \ra^{2m}  \la n \ra^{-2r_1} \mathcal{F}_1(k-n)^2\,.
\end{aligned}
$$
Using that $\mathbf{1}_{|n| < |k-n|} \la k \ra^{2m} \lesssim \mathbf{1}_{|n| < |k-n|} \la k-n \ra^{2m}$, the above sum can be bounded by
$$
\sum_{n' \in \Z^3} \left\{ \sum_{n \in \Z^3} \mathbf{1}_{|n| < |n'|} \la n \ra^{-2r_1} \right\} \la n' \ra^{2m} \mathcal{F}_1(n')^2  
$$
and we observe by standard arguments that
$$
\sum_{n \in \Z^3} \mathbf{1}_{|n| < |n'|} \la n \ra^{-2r_1} \lesssim 
\begin{cases}
\la n' \rangle^{3-2r_1}  \; &\text{if} \quad r_1 < \frac32 \,, \\
1 \; &\text{if} \quad r_1> \frac32\,.
\end{cases}
$$
This implies
$$
\begin{aligned}
\sum_{k \in \Z^3} \la k \ra^{2m} I_1(k)^2
&\lesssim \| \la \cdot \ra^{r_1} \mathcal{F}_2 \|_{\ell^2(\Z^3)}^2   \| \la \cdot \ra^{m+(\frac32-r_1)_+} \mathcal{F}_1 \|_{\ell^2(\Z^3)}^2 \\
&= \| f_1 \|_{\widetilde{L}^{p_1}_T H^{m+(\frac32-r_1)_+}_x L^2_v}^2 \| f_2 \|_{\widetilde{L}^{q_1}_T  H^{r_1}_x H^{s,*}_v}^2.
\end{aligned}
$$
The term $I_2$ can be estimated in a similar fashion, by exchanging the role of $f_1$ and $f_2$. Indeed, we first apply H\"older's inequality in time with $1/p_2+1/q_2 = 1/2$ to obtain
\begin{equation}\label{eq:I1_FLp2_GLq2}
I_2 (k)
\lesssim \sum_{n \in \Z^3}
 \mathbf{1}_{|n| \le |k-n|}
 \| F_1(\cdot,n) \|_{L^{p_2}_T} \| F_2(\cdot, k-n) \|_{L^{q_2}_T} . 
\end{equation}
Denoting $\mathcal{F}'_1 (k) = \| F_1(\cdot,k) \|_{L^{p_2}_T} $ and $\mathcal{F}'_2(k) = \| F_2(\cdot, k) \|_{L^{q_2}_T}$, the Cauchy-Schwarz inequality yields
$$
\begin{aligned}
I_2(k) \lesssim \| \la \cdot \ra^{r_2} \mathcal{F}'_1 \|_{\ell^2(\Z^3)} \left\{  \sum_{n \in \Z^3}
 \mathbf{1}_{|n| \le |k-n|}  \la n \ra^{-2r_2} \mathcal{F}'_2(k-n)^2 \right\}^{\frac12}.
\end{aligned}
$$
Arguing as above, it follows
$$
\begin{aligned}
\sum_{k \in \Z^3} \la k \ra^{2m} I_2(k)^2
\lesssim \| f_1 \|_{\widetilde{L}^{p_2}_T  H^{r_2}_x L^2_v}^2 \| f_2 \|_{\widetilde{L}^{q_2}_T  H^{m+(\frac32-r_2)_+}_x H^{s,*}_v}^2,
\end{aligned}
$$
which completes the proof.
\end{proof}

We give another estimate on $\Gamma$ in the specific case where both entries are macroscopic (which in particular implies that there is no loss of regularity in the velocity variable). 

\begin{prop} \label{prop:nonlinear_macro}
Let $m \ge 0$. For any smooth enough functions~$f_1$,~$f_2$ we have: 
\begin{enumerate}
\item For any $ r_1,r_2 \neq 3/2 $ there holds
	\[
	\begin{aligned}
	\|\Gamma (\mathbf P_0 f_1, \mathbf P_0 f_2) \|_{\widetilde L^\infty_T H^{m}_x L^2_v} 
	&\lesssim \| \mathbf P_0 f_1 \|_{\widetilde L^\infty_T H^{m+(\frac32-r_1)_+}_x L^2_v} \| \mathbf P_0 f_2 \|_{\widetilde L^\infty_T H^{r_1}_x L^2_v}\\
	&\quad
	+ \| \mathbf P_0 f_1 \|_{\widetilde L^\infty_T H^{r_2}_x L^2_v} \| \mathbf P_0 f_2 \|_{\widetilde L^\infty_T H^{m+(\frac32-r_2)_+}_x L^2_v}\,. 
	\end{aligned}
	\]
	
\item For any $ r_1,r_2 \neq 3/2 $ and $p_1,q_1,p_2,q_2 \in [1,\infty]$ such that $1/p_1+1/q_1=1/p_2+1/q_2=1/4$, there holds
	\[
	\begin{aligned}
	\|\Gamma (\mathbf P_0 f_1, \mathbf P_0 f_2) \|_{\widetilde L^4_T H^{m}_x L^2_v} 
	&\lesssim \| \mathbf P_0 f_1 \|_{\widetilde L^{p_1}_T H^{m+(\frac32-r_1)_+}_x L^2_v} \| \mathbf P_0 f_2 \|_{\widetilde L^{q_1}_T H^{r_1}_x L^2_v}\\
	&\quad
	+ \| \mathbf P_0 f_1 \|_{\widetilde L^{p_2}_T H^{r_2}_x L^2_v} \| \mathbf P_0 f_2 \|_{\widetilde L^{q_2}_T H^{m+(\frac32-r_2)_+}_x L^2_v}\,. 
	\end{aligned}
	\]

\end{enumerate}
\end{prop}

\begin{proof}
Using the regularization properties of $\P_0$, thanks to \cite{Strain,CRT} respectively for the noncutoff Boltzmann and Landau equations, and the fact that $\| \la v \ra^p \mathbf P_0 \phi \|_{H^{q}_v } \lesssim \| \mathbf P_0 \phi \|_{L^2_v}$ for all~$p,q \ge 0$, we have
\begin{equation}\label{eq:Gamma_L2v}
\| \Gamma (\mathbf P_0 f_1 , \mathbf P_0 f_2) \|_{L^2_v } \lesssim \| \mathbf P_0 f_1 \|_{L^2_v} \| \mathbf P_0 f_2 \|_{L^2_v} \,.
\end{equation}
Therefore we have for any $ k \in \Z^3$, 
	$$
	\| \widehat \Gamma (\mathbf P_0 f_1, \mathbf P_0 f_2) ( k) \|_{L^\infty_T L^2_v} 
	\lesssim   \sum_{ n  \in \Z^3} 
	\| \widehat{\mathbf P_0 f_1}(k- n ) \|_{L^\infty_T L^2_v} \| \widehat{\mathbf P_0 f_2} (n) \|_{L^\infty_T L^2_v} \,.
	$$
We then conclude the proof of~(1) as in the proof of Proposition~\ref{prop:nonlinear}. The proof of~(2) is similar by writing
$$
\begin{aligned}
	\| \widehat \Gamma (\mathbf P_0 f_1, \mathbf P_0 f_2) ( k) \|_{L^4_T L^2_v} 
	&\lesssim   \sum_{ n  \in \Z^3} \mathbf{1}_{|n|<|k-n|}
	\| \widehat{\mathbf P_0 f_1}(k- n ) \|_{L^{p_1}_T L^2_v} \| \widehat{\mathbf P_0 f_2} (n) \|_{L^{q_1}_T L^2_v} \\
&\quad
+ \sum_{ n  \in \Z^3} \mathbf{1}_{|k-n|\ge n}
	\| \widehat{\mathbf P_0 f_1}(n ) \|_{L^{p_2}_T L^2_v} \| \widehat{\mathbf P_0 f_2} (k-n) \|_{L^{q_2}_T L^2_v} \,,
\end{aligned}
	$$	
and then arguing as in the proof of Proposition~\ref{prop:nonlinear}.
\end{proof}

\section{The equation on $\delta^{\eps}$: proof of Proposition~\ref{estimates on delta eps}} \label{sec:deltaeps}

This section is devoted to the proof of Proposition~\ref{estimates on delta eps}.

\subsection{Continuity estimates for $U^\eps$}\label{sct:continuity}
Let us prove  Proposition~\ref{estimates on delta eps}--(1). From Proposition~\ref{prop:estimate_Ueps} we have
$$
\eps^\beta \| U^\eps(\cdot - t) F(t) \|_{\widetilde L^\infty_{[t,T]} H^\ell_x L^2_v}
\lesssim \eps^\beta \| F(t) \|_{H^\ell_x L^2_v} \,,
$$
and $$
\| U^\eps(\cdot - t) F(t) \|_{\widetilde L^\infty_{[t,T]} H^{\frac12}_x L^2_v} \lesssim  \| F(t) \|_{H^{\frac12}_x L^2_v} \, .
$$
Let us turn to the~$L^2$ norm in time. First we note that
$$ 
\frac{\eps^\beta}{\sqrt\eps} \| \mathbf P_0^\perp U^\eps(\cdot - t) F(t) \|_{ L^2_{[t,T]} H^\ell_x H^{s,*}_v} 
\lesssim \eps^{\beta+\frac12} \| F(t) \|_{H^\ell_x L^2_v} \,,
$$
and $$
\frac{1}{\sqrt\eps}\| \P_0^\perp U^\eps(\cdot - t) F(t) \|_{L^2_{[t,T]} H^\frac32_x H^{s,*}_v} \lesssim \sqrt\eps \| F(t) \|_{H^\frac32_x L^2_v}\,.
$$
We now decompose $\P_0 U^\eps = \P_0 U^{\eps,\sharp} + \P_0 U^{\eps,\flat}$ as in~(\ref{decomposition Ueps}). By Proposition~\ref{prop:estimate_Uepssharp}--(1) we get
$$
\eps^\beta \| \P_0 U^{\eps,\sharp}(\cdot - t) F(t) \|_{L^2_{[t,T]} H^\ell_x H^{s,*}_v} \lesssim \eps^{1+\beta} \| F(t) \|_{H^\ell_x L^2_v}\,,
$$
and also
$$
\| \P_0 U^{\eps,\sharp}(\cdot - t) F(t) \|_{L^2_{[t,T]} H^\frac32_x H^{s,*}_v} \lesssim \eps \| F(t) \|_{H^\frac32_x L^2_v}\,.
$$
Furthermore, Proposition~\ref{prop:estimate_Uepsflat} yields
$$
\eps^\beta \| \P_0 U^{\eps,\flat}(\cdot - t) F(t) \|_{L^2_{[t,T]} H^\ell_x H^{s,*}_v} \lesssim  \eps^\beta \| F(t) \|_{H^{\ell-1}_x L^2_v} \, ,
$$
as well as
$$
\| \P_0 U^{\eps,\flat}(\cdot - t) F(t) \|_{L^2_{[t,T]} H^\frac32_x H^{s,*}_v} \lesssim  \| F(t) \|_{H^{\frac12}_x L^2_v}\, .
$$
Gathering the previous estimates and using that $\beta<1/2$ and~$\ell>3/2$, it follows that
$$
\begin{aligned}
&\| U^\eps(\cdot - t) F(t) \|_{\mathcal X^\eps_{[t,T]}} \\
&\qquad
\lesssim 
\eps^\beta \| F(t) \|_{H^\ell_x L^2_v} 
+ \sqrt\eps \| F(t) \|_{H^\frac32_x L^2_v} 
+ \eps^{\beta} \| F(t) \|_{H^{\ell-1}_x L^2_v}
+ \| F(t) \|_{H^{\frac12}_x L^2_v} \\
&\qquad
\lesssim \eps^\beta \| F \|_{\widetilde L^\infty_{[0,t]} H^\ell_x L^2_v} + \| F \|_{\widetilde L^\infty_{[0,t]} H^{\frac12}_x L^2_v} \lesssim \| F \|_{\mathcal X^\eps_{[0,t]}} \,.
\end{aligned}
$$
This concludes the proof of Proposition~\ref{estimates on delta eps}--(1).

\subsection{Contribution of the data $\mathcal D^{\eps}$}\label{sct:data}
Let us prove  Proposition~\ref{estimates on delta eps}--(2). Recall that
\begin{equation}\label{decDepsproof}
\mathcal D^{\eps}(t)= \big(U^{\eps}(t) -U_{\rm NSF}(t) \big)\P_0f_{\mathrm{in}}^\eps + U^\eps(t) \P_0^\perp f_{\rm in}^\eps \, .
\end{equation}
Let us first prove that
\begin{equation}\label{limit data1}
\big\|  \big(U^{\eps}(\cdot) -U_{\rm NSF}(\cdot) \big)\P_0f_{\mathrm{in}}^\eps  \big\| _{\mathcal X^\eps_\infty}
\lesssim
\eps^{\frac12-\alpha}
  \|{ g_{\rm in}}\|_{ H^{\frac12}_x L^2_v} \, .
\end{equation}
We start that by recalling that by~(\ref{decomposition Ueps}) and~(\ref{decomposition Ueps1}) there holds
$$
\begin{aligned}
 \big(U^{\eps}(t) -U_{\rm NSF}(t) \big)\P_0f_{\mathrm{in}}^\eps& = \big(  {\widetilde U^{\eps}_{\rm NSF}}+ U^{\eps,\flat}_{\rm wave}	 + U^{\eps,\sharp}\big)(t)  \P_0f_{\mathrm{in}}^\eps   \, .
\end{aligned}
$$
We notice that since~$  \P_0f_{\mathrm{in}}^\eps $ is well-prepared, then~$U^{\eps,\flat}_{\rm wave}	(t)  \P_0f_{\mathrm{in}}^\eps\equiv0$ so
$$
 \big(   {\widetilde U^{\eps}_{\rm NSF}}+ U^{\eps,\flat}_{\rm wave}	 + U^{\eps,\sharp}\big)(t)  \P_0f_{\mathrm{in}}^\eps =  \big(  \widetilde U^{\eps}_{\rm NSF}	 + U^{\eps,\sharp}\big)(t)    \P_0f_{\mathrm{in}}^\eps\, .
$$
Let us start by considering~$ U^{\eps,\sharp}(t) \P_0f_{\mathrm{in}}^\eps$. Thanks to Proposition~\ref{prop:estimate_Uepssharp}--(1) we have 
$$
\| U^{\eps ,\sharp}(\cdot)  \P_0 f_{\mathrm{in}}^\eps \|_{L^2_tH^m_x H^{s,*}_v}
\lesssim \eps \|   \P_0f_{\mathrm{in}}^\eps\|_{ H^{m}_x L^2_v}
$$
for any $m \ge 0$. To deal with the~$\widetilde L^\infty$ norm in time, we shall   follow the arguments of~\cite{GT} (see in particular the proofs of Lemmas 3.3 and~3.5). We notice as in~\cite[Lemma~6.2]{Bardos-Ukai} that
$$
U^{\eps ,\sharp}(t )f = U^\eps(t)U^{\eps ,\sharp}(0 ) f= U^\eps(t) \Big[\mathcal{F}_x^{-1} \Big({\rm Id}-\chi\Big (\frac{\eps| k|}\kappa\Big)\sum_{\star \in \{\rm NS, heat, wave\pm\}}  \mathcal P_\star(\eps  k)  \Big){\widehat f ( k)} \Big]
$$
so in particular
$$U^{\eps,\sharp}(t )\mathbf P_0f_{\mathrm{in}}^\eps= U^\eps(t) \Bigg[\mathcal{F}_x^{-1} \Bigg( \Big({\rm Id}-\chi\Big (\frac{\eps| k|}\kappa\Big)\Big) - \eps | k| \chi\Big (\frac{\eps| k|}\kappa\Big)\sum_{\star \in \{\rm NS, heat, wave\pm\}}  \mathcal P_\star(\eps  k)\Bigg) \widehat{\mathbf P_0f_{\mathrm{in}}^\eps }( k)\Bigg] \,.
$$
The $ H^{m}_x L^2_v$-norm of the first term in the right-hand side can be estimated using  
\begin{equation}\label{estimatechi}
  \Big|\chi\Big (\frac{\eps| k|}\kappa\Big)-1\Big| \lesssim \eps | k|\, ,
\end{equation}
and thanks to the fact that the projectors $\mathcal P_\star $ are bounded from~$L^2_v$ to~$L^2_v$. The same   holds   for the terms coming from the second part of the right-hand side, so we find that
 $$
\begin{aligned}
\|U^{\eps ,\sharp}(\cdot)  \P_0f_{\mathrm{in}}^\eps \|_{ \widetilde L^\infty_tH^{m }_x L^2_v }
  &\lesssim \eps\|   \P_0f_{\mathrm{in}}^\eps\|_{ H^{m+1}_x L^2_v} .\end{aligned}
 $$
Using this estimate, along with Proposition~\ref{prop:estimate_Uepssharp}--(1) and Remark~\ref{remark truncation} and the constraints on~$\alpha,\beta$ and~$\ell$, we deduce that
$$
\begin{aligned}
&\eps^\beta\|U^{\eps ,\sharp}(\cdot)  \P_0f_{\mathrm{in}}^\eps \|_{ \widetilde L^\infty_tH^{\ell }_x L^2_v }
+\frac{\eps^{\beta}}{\sqrt\eps} \| \mathbf P_0^\perp U^{\eps ,\sharp}(\cdot)  \P_0 f_{\mathrm{in}}^\eps \|_{L^2_tH^\ell_x H^{s,*}_v} 
+ \eps^{\beta} \| \mathbf P_0 U^{\eps ,\sharp}(\cdot)  \P_0f_{\mathrm{in}}^\eps \|_{L^2_tH^\ell_x H^{s,*}_v} \\
&\quad\lesssim 
\eps^{1+\beta }\|   \P_0f_{\mathrm{in}}^\eps\|_{ H^{\ell+1}_x L^2_v} 
+\left( \eps^{\beta+\frac12} + \eps^{\beta+1} \right) \|   \P_0f_{\mathrm{in}}^\eps\|_{ H^{\ell}_x L^2_v} \\
&\quad\lesssim 
\left( \eps^{1+\beta -\alpha(\ell+\frac12)} + \eps^{\beta+\frac12 -\alpha(\ell-\frac12)}  \right)   \|{ g_{\rm in}}\|_{ H^{\frac12}_x L^2_v} \\
&\quad\lesssim 
\eps^{ \frac12 -\alpha }   \|{ g_{\rm in}}\|_{ H^{\frac12}_x L^2_v} \, ,
\end{aligned}
$$
as well as
$$
\begin{aligned}
&\|U^{\eps ,\sharp}(\cdot)  \P_0f_{\mathrm{in}}^\eps \|_{ \widetilde L^\infty_tH^{\frac12 }_x L^2_v }
+\frac{1}{\sqrt\eps} \| \mathbf P_0^\perp U^{\eps ,\sharp}(\cdot)  \P_0 f_{\mathrm{in}}^\eps \|_{L^2_tH^\frac32_x H^{s,*}_v} 
+  \| \mathbf P_0 U^{\eps ,\sharp}(\cdot)  \P_0f_{\mathrm{in}}^\eps \|_{L^2_tH^\frac32_x H^{s,*}_v} \\
&\quad\lesssim 
\left( \sqrt \eps + \eps \right) \|   \P_0f_{\mathrm{in}}^\eps\|_{ H^\frac32_x L^2_v} \\
&\quad\lesssim 
\eps^{ \frac12 -\alpha }  \|{ g_{\rm in}}\|_{ H^{\frac12}_x L^2_v} \, .
\end{aligned}
$$
We conclude that
\begin{equation}\label{eq:Ueps_sharp}
\big\|U^{\eps ,\sharp}(\cdot)\P_0f_{\mathrm{in}}^\eps  \big\| _{\mathcal X^\eps_\infty}
\lesssim 
 \eps^{\frac12-\alpha}  \|{ g_{\rm in}}\|_{H^{\frac12}_x L^2_v} \, .
\end{equation}
Now let us turn to~$ \widetilde U^{\eps}_{\rm NSF}(t)\P_0f_{\mathrm{in}}^\eps$ as defined in~(\ref{decomposition Ueps1}). By construction  it is made of three terms, defined  in Fourier variables by
$$
\begin{aligned}
\widehat{  \widetilde U^{\eps}_{\rm NSF}}(t, k)&= \widehat{  \widetilde U^{\eps,\sharp}_{\rm NSF}}(t, k) + \widehat{  \widetilde U^{\eps 1,\flat}_{\rm NSF}}(t, k) +    \widehat{  \widetilde U^{\eps 2,\flat}_{\rm NSF}}(t, k)\\
	&:= \Big(1-\chi \Big(\frac{\eps | k|}\kappa\Big)\Big)\sum_{\star \in \{\rm NS, heat\}}
 e^{ - \nu_{\star} | k|^2 t} \, \mathcal P_{\star}^0 \Big({ k \over | k|}\Big)\\
 & \qquad + \chi \Big(\frac{\eps | k|}\kappa\Big)\sum_{\star \in \{\rm NS, heat\}}
	\Big(e^{   \lambda_{\star} (\eps  k) \frac t{\eps^2}}  -  e^{ - \nu_{\star} | k|^2 t}\Big) \mathcal P_{\star}^0  \Big({ k \over | k|}\Big)   \\
	&\qquad + \chi \Big(\frac{\eps | k|}\kappa\Big)\sum_{\star \in \{\rm NS, heat\}} e^{   \lambda_{\star} (\eps  k) \frac t{\eps^2}} \left[ \eps |k|  \mathcal P_{\star}^1 (\tfrac{k}{|k|}) + \eps^2 |k|^2 \mathcal P_{\star}^2 (\eps  k)  \right] \,.
\end{aligned}
$$
We shall study the three contributions in turn, starting with~$ \widetilde U^{\eps,\sharp}_{\rm NSF}(t)$.  We recall  again that the projectors $\mathcal P_\star^j$ are bounded from~$L^2_v$ to~$L^2_v$ and from~$H^{s,*}_v$ to~$H^{s,*}_v$. 
%
Using the fact that~${\mathbf P}_0 f_{\mathrm{in}}^\eps$ is mean free which implies that there is no contribution to~$ k = 0$,   and the fact that the integral in time of the exponential term provides a factor~$ | k|^{-2}$,  we have for any~$m \ge 0$,
$$
\begin{aligned}
\big\| \widetilde U^{\eps,\sharp}_{\rm NSF}(\cdot)\P_0f_{\mathrm{in}}^\eps\big\|^2_{  L^2_tH^{m }_x H^{s,*}_v } &\lesssim\sum_{\star \in \{\rm NS, heat\}}
\sum_{ k \in \Z^3 \setminus\{0\}} \langle  k\rangle^{2m} | k|^{-2} \Big(1-\chi \Big(\frac{\eps | k|}\kappa\Big)\Big)^2 \Big\| \widehat{{\mathbf P}_0 f_{\mathrm{in}}^\eps} ( k,\cdot)\Big\|_{H^{s,*}_v}^2   .
\end{aligned}
$$
Using~\eqref{estimatechi} and the fact that $\| \mathbf P_0 f \|_{H^{s,*}_v} \lesssim \| \mathbf P_0 f \|_{L^2_v}$, we get
$$
\begin{aligned}
\big\| \widetilde U^{\eps,\sharp}_{\rm NSF}(\cdot)\P_0f_{\mathrm{in}}^\eps\big\|_{  L^2_tH^{m }_x H^{s,*}_v } 
&\lesssim \eps \left( \sum_{ k \in \Z^3} \langle  k\rangle^{2m}   \Big\| \widehat{{\mathbf P}_0 f_{\mathrm{in}}^\eps} ( k,\cdot)\Big\|_{L^2_v}^2 \right)^{\frac12} 
\lesssim 
\eps \| \mathbf P_0  f_{\mathrm{in}}^\eps \|_{H^m_x L^2_v} \, .
\end{aligned}
$$
Similar computations give 
$$
\big\| \widetilde U^{\eps,\sharp}_{\rm NSF}(\cdot)\P_0f_{\mathrm{in}}^\eps\big\|_{\widetilde L^\infty_tH^{m}_x L^2_v }  \lesssim 
\eps  \| \mathbf P_0  f_{\mathrm{in}}^\eps \|_{H^{m+1}_xL^2_v}\, .
$$
Therefore, arguing as for obtaining estimate~\eqref{eq:Ueps_sharp} for the term $U^{\eps,\sharp}$, we also get
\begin{equation}\label{eq:widetildeUeps_sharp}
\big\| \widetilde U^{\eps,\sharp}_{\rm NSF} (\cdot) \P_0 f_{\mathrm{in}}^\eps  \big\| _{\mathcal X^\eps_\infty}
\lesssim  
 \eps^{\frac12-\alpha}   \|{ g_{\rm in}}\|_{H^{\frac12}_x L^2_v} \, .
\end{equation}
Next we turn to~$ \widetilde U^{\eps 1,\flat}_{\rm NSF}(t)$. We write
$$
\begin{aligned}
 \chi \Big(\frac{\eps | k|}\kappa\Big)\Big|e^{   \lambda_{\star} (\eps  k) \frac t{\eps^2}}  -  e^{ - \nu_{\star} | k|^2 t} \Big| &\lesssim
  \chi \Big(\frac{\eps | k|}\kappa\Big) e^{ -\frac{ \nu_{\star}}2 | k|^2 t}t \eps | k|^3\\
  &\lesssim
   e^{ -\frac{ \nu_{\star}}4 | k|^2 t}  \eps | k| \, .
  \end{aligned}
$$
 The same argument gives 
$$
\big\| \widetilde U^{\eps 1,\flat} _{\rm NSF}(\cdot) \P_0f_{\mathrm{in}}^\eps\big\|_{\widetilde L^\infty_tH^{m}_x L^2_v }  
+ \big\| \widetilde U^{\eps 1,\flat} _{\rm NSF}(\cdot) \P_0f_{\mathrm{in}}^\eps\big\|_{L^2_tH^{m}_x L^2_v }
\lesssim 
\eps    \| \mathbf P_0  f_{\mathrm{in}}^\eps \|_{H^{m+1}_xL^2_v}\,  .
$$
We can then again argue as in the case of the bound~\eqref{eq:Ueps_sharp} to deduce
\begin{equation}\label{eq:widetildeUeps_1}
\big\| \widetilde U^{\eps 1,\flat}_{\rm NSF} (\cdot) \P_0 f_{\mathrm{in}}^\eps  \big\| _{\mathcal X^\eps_\infty}
\lesssim \eps^{\frac12-\alpha}  \|g_{\rm in}\|_{H^{\frac12}_x L^2_v}
\, .
\end{equation}
The computations for~$ \widetilde U^{\eps 2,\flat}_{\rm NSF}(t)$ are very similar: we find
\begin{equation}\label{eq:widetildeUeps_2}
\begin{aligned}
\big\| \widetilde U^{\eps 2,\flat}_{\rm NSF}  (\cdot) \P_0 f_{\mathrm{in}}^\eps  \big\| _{\mathcal X^\eps_\infty}
&\lesssim\eps^{\frac12-\alpha}
 \|g_{\rm in}\|_{H^{\frac12}_x L^2_v}
\, .
  \end{aligned}
\end{equation}
Putting together the estimates on~$\widetilde U^{\eps,\sharp}_{\rm NSF}(t)\P_0f_{\mathrm{in}}^\eps$,~$\widetilde U^{\eps 1,\flat}_{\rm NSF} (t)\P_0f_{\mathrm{in}}^\eps$ and~$\widetilde U^{\eps 2,\flat}_{\rm NSF} (t)\P_0f_{\mathrm{in}}^\eps$ gives~(\ref{limit data1}).

\medskip
\noindent  {Recalling~(\ref{decDepsproof})}, it remains to prove that
\begin{equation}\label{limit data2}
\big\|  U^{\eps}(t)\P_0^\perp f_{\mathrm{in}}^\eps  \big\| _{\mathcal X^\eps_\infty}
\lesssim 
\eps^{\beta} \|\P_0^\perp f_{\rm in}^\eps\|_{H^\ell_x L^2_v} 
+ \|\P_0^\perp f_{\rm in}^\eps\|_{H^{\frac12}_x L^2_v} 
\, .
\end{equation}
From Proposition~\ref{prop:estimate_Ueps}--(1), we have 
	\[
	\big\|  U^{\eps}(t)\P_0^\perp f_{\mathrm{in}}^\eps  \big\| _{\widetilde L^\infty_T H^{\frac12}_x L^2_v} 
	\lesssim \|\P_0^\perp f_{\rm in}^\eps\|_{H^{\frac12}_xL^2_v} \, ,
	\]
as well as
$$
\begin{aligned}
\frac{1}{\sqrt \eps} \big\| \mathbf P_0^\perp U^{\eps}(t)\P_0^\perp f_{\mathrm{in}}^\eps  \big\| _{L^2_T H^\frac32_x H^{s,*}_v} 
\lesssim \sqrt \eps  \|\P_0^\perp f_{\rm in}^\eps\|_{H^\frac32_xL^2_v}
\end{aligned}
$$
Similarly, 
$$
	\eps^\beta \big\|  U^{\eps}(t)\P_0^\perp f_{\mathrm{in}}^\eps  \big\| _{\widetilde L^\infty_T H^{\ell}_x L^2_v} 
	\lesssim \eps^\beta \|\P_0^\perp f_{\rm in}^\eps\|_{H^{\ell}_x L^2_v} \, ,
$$
and
$$
\begin{aligned}
\frac{\eps^\beta}{\sqrt\eps} \big\| \mathbf P_0^\perp U^{\eps}(t)\P_0^\perp f_{\mathrm{in}}^\eps  \big\| _{L^2_T H^{\ell}_x H^{s,*}_v} 
\lesssim  \eps^{\beta+\frac12}  \|\P_0^\perp f_{\rm in}^\eps\|_{H^{\ell}_xL^2_v}\,  .
\end{aligned}
$$
It remains to control~$ \P_0 U^\eps  \P_0^\perp $ in~$L^2_T H^m_x H^{s,*}_v$ for~$m=\ell$ or~$
m=3/2$.
From the decomposition~\eqref{decomposition Ueps}, for any $ k \in \Z^3$ and $t \ge 0$, we have
	\[
	\|\widehat U^\eps(t, k) \P_0^\perp f_{\mathrm{in}}^\eps  \|_{L^2_v \to L^2_v} 
	\lesssim \eps | k| e^{-t| k|^2} + e^{-\lambda_0 \frac{t}{\eps^2}}\, .
	\]
Using that $\mathbf P_0$ is bounded from  $L^2_v$ into $H^{s,*}_v$  and that $\P_0^\perp \P_0^\perp = \P_0^\perp$, we obtain
$$
\| \mathbf P_0 \widehat U^\eps(\cdot, k) \P_0^\perp \widehat f_{\rm in}^\eps ( k) \|_{L^2_T H^{s,*}_v} 
	\lesssim \eps \|\P_0^\perp  \widehat f_{\rm in}^\eps ( k) \|_{L^2_v}\, .
$$
We thus deduce
$$
\| \mathbf P_0  U^\eps(\cdot) \P_0^\perp  f_{\rm in}^\eps  \|_{L^2_T H^\frac32_x H^{s,*}_v} 
	\lesssim \eps \|\P_0^\perp   f_{\rm in}^\eps  \|_{H^\frac32_x L^2_v}\, ,
$$
as well as
$$
\eps^{\beta} \| \mathbf P_0  U^\eps(\cdot) \P_0^\perp  f_{\rm in}^\eps  \|_{L^2_T H^{\ell}_x H^{s,*}_v} 
	\lesssim \eps^{\beta+1} \|\P_0^\perp   f_{\rm in}^\eps  \|_{H^{\ell}_x L^2_v}\, .
$$
Gathering the previous estimates and using that $\sqrt \eps  \|\P_0^\perp f_{\rm in}^\eps\|_{H^\frac32_xL^2_v} \lesssim \eps^\beta \|\P_0^\perp f_{\rm in}^\eps\|_{H^{\ell}_xL^2_v}$ ends the proof of~\eqref{limit data2}.
The estimates~\eqref{limit data1} and~\eqref{limit data2} together give Proposition~\ref{estimates on delta eps}--(2).

\subsection{Contribution of the source term $\mathcal S^{\eps}$}\label{sct:source}
In this paragraph, we are going to prove Proposition~\ref{estimates on delta eps}--(3). 
We recall that 
	$$ 
	\mathcal S ^{\eps}(t)  = \Psi^{\eps }[  g^\eps, g^\eps](t)-  \Psi_{\rm NSF} [ g^\eps, g^\eps](t) \, ,
	$$
	and we want to prove that
$$
	 \|\mathcal S^{\eps}\| _{\mathcal X^\eps_T}  \leq{\color{black} \eps^{\frac12-2\alpha} \Phi(\|{ g_{\rm in}}\|_{H^{\frac12}_xL^2_v}) }	   
$$
for a nonnnegative increasing function $\Phi$. 
We recall decompositions~(\ref{decomposition Psieps 1}) and~(\ref{decomposition Psieps 22}).  We can further expand~$\Psi^{\eps,\flat}$ by writing
\begin{equation}\label{decomposition Psieps 222}
	\Psi^{\eps,\flat}  
	= \Psi_{\rm NSF}  +  \Psi^{\eps1,\sharp}_{\rm NSF}
	+ \Psi^{\eps1,\flat}_{\rm wave} 
	 +\widetilde  \Psi^{\eps1,\flat}
	+  \Psi^{\eps2,\flat}
\end{equation}
 where writing~$\widehat\Psi_\star [f_1,f_2](t)=\mathcal{F}_x \big( \Psi _{\rm \star} [f_1,f_2](t)\big)  $  and recalling that~$\mathbf P_0 \Gamma_{\mathrm{sym}} = 0$,
 \begin{align*}
	\widehat \Psi _{\rm NSF} [f_1,f_2](t, k) 
	&:= \sum_{\star \in \{\rm NS, heat\}}
 \int_0^t e^{  -\nu _{\star}  (t-t') | k|^2} | k|\mathcal P_{\star} ^{1}\Big({ k \over | k|}\Big) \widehat \Gamma_{\mathrm{sym}} \big(f_1(t'),f_2(t')\big) ( k) \, \d t' 	\,,  \\
		\widehat \Psi^{\eps1,\sharp} _{\rm NSF}  [f_1,f_2](t, k) 
	&:=   \Big(\chi\Big (\frac{\eps| k|}\kappa\Big)-1\Big) \\
	&\hspace{0,5cm}
	\times \, \sum_{\star \in \{\rm NS, heat  \}}
	\int_0^t e^{   -  \nu_{\star}   | k|^2  (t-t')}   | k|  \mathcal P_{\star}^1\Big({ k \over | k|}\Big)   \widehat \Gamma_{\mathrm{sym}} \big(f_1(t'),f_2(t')\big) ( k) \, \d t' \,,\\
	\widehat \Psi^{\eps1,\flat} _{\rm wave} [f_1,f_2](t, k) 
	&:=  
 \chi\Big (\frac{\eps| k|}\kappa\Big)  \\
 &\hspace{-0,3cm}
	\times \, \sum_{\pm}
	\int_0^t e^{  ( \pm i c \eps| k| -  \nu_{\rm \small wave\pm}  \eps^2  | k|^2  )\frac {t-t'}{\eps^2}}   | k|  \mathcal P_{\rm \small wave\pm}^1\Big({ k \over | k|}\Big)   \widehat \Gamma_{\mathrm{sym}} \big(f_1(t'),f_2(t')\big) ( k)\, \d t' \,,\\
 \widehat{\widetilde \Psi}\vphantom\Psi^{\eps1,\flat}   [f_1,f_2](t, k) 
	&:=  \chi\Big (\frac{\eps| k|}\kappa\Big) \sum_{\star \in \{\rm NS, heat,wave\pm  \}}
\int_0^t e^{  ( \pm i c_\star \eps| k| -  \nu_{\star} \eps^2  | k|^2  )\frac {t-t'}{\eps^2}} \\
	&\hspace{0,5cm}
	\times \, \Big(e^{(t-t') \frac{\gamma_\star(\eps| k|)}{\eps^2}}-1 \Big)  | k|  \mathcal P_{\star}^1\Big({ k \over | k|}\Big)   \widehat \Gamma_{\mathrm{sym}} \big(f_1(t'),f_2(t')\big) ( k)\, \d t' \,,\\
		\widehat{  \Psi}^{\eps2,\flat} [f_1,f_2](t, k) 
	&:=  \chi\Big (\frac{\eps| k|}\kappa\Big) \\
	&\hspace{0,4cm}
	\times \, \sum_{\star \in \{\rm NS, heat,wave\pm\}}
	\int_0^t
	e^{   \lambda_{\star}(\eps  k) \frac {t-t'}{\eps^2}}
	 \eps | k|^2  \mathcal P_{\star}^2(\eps k)  \widehat \Gamma_{\mathrm{sym}} \big(f_1(t'),f_2(t')\big) ( k)\, \d t'   \, . 
	\end{align*}
 We have used the notation~$c_\star:=0$ if~$\star \in  \{\rm NS, heat\}$ {and $c_\star := c$ if $\star \in \{{\rm wave} \pm\}$}.
So let us write
\begin{equation}\label{eq:source_decomposition}
 \mathcal S ^{\eps}(t) =\Big(  \Psi^{\eps,\sharp }  + \Psi_{\rm NSF}  +  \Psi^{\eps1,\sharp}_{\rm NSF}
	+ \Psi^{\eps1,\flat}_{\rm wave} 
	 +\widetilde  \Psi^{\eps1,\flat}
	+  \Psi^{\eps2,\flat}\Big)
 [  g^\eps, g^\eps](t)\, ,
\end{equation}
and let us estimate each term separately.
We start with the first term in \eqref{eq:source_decomposition}. 
\begin{lem}\label{lem:source_Psi_eps_sharp}
There holds
$$
\begin{aligned}
\|  \Psi^{\eps,\sharp }  [  g^\eps, g^\eps]\|_{\mathcal X^\eps_T}
&\lesssim \eps^{1+\beta}  \| \Gamma (g^\eps, g^\eps) \|_{\widetilde L^\infty_T H^\ell_x L^2_v}
+ \eps^{\frac12+\beta}  \| \Gamma (g^\eps, g^\eps) \|_{L^2_T H^\ell_x (H^{s,*}_v)'} \\
&\quad
+ \eps \| \Gamma (g^\eps, g^\eps) \|_{\widetilde L^\infty_T H^{\frac12}_x L^2_v}  
+ \sqrt \eps \| \Gamma (g^\eps, g^\eps) \|_{L^2_T H^\frac32_x (H^{s,*}_v)'} \, .
\end{aligned}
$$
\end{lem}

\begin{proof}
From Proposition~\ref{prop:Psiepssharp} 
and Corollary~\ref{cor:estimate_Psieps}
we have directly that$$
\begin{aligned}
& \eps^{\beta} \|  \Psi^{\eps,\sharp }[  g^\eps, g^\eps]\|_{\widetilde L^\infty_T H^\ell_x L^2_v} 
+\frac{\eps^\beta}{\sqrt\eps}\| \mathbf P_0^\perp \Psi^{\eps,\sharp }[  g^\eps, g^\eps]\|_{L^2_T H^\ell_x H^{s,*}_v}
+\eps^{\beta} \| \mathbf P_0 \Psi^{\eps,\sharp }[  g^\eps, g^\eps]\|_{L^2_T H^\ell_x H^{s,*}_v} \\
&\quad
\lesssim  \eps^{1+\beta} \| \Gamma (g^\eps , g^\eps) \|_{\widetilde L^\infty_T H^\ell_x {L^2_v}}
+ \left( \eps^{\beta+\frac12} + \eps^{\beta+1} \right)  \| \Gamma (g^\eps , g^\eps) \|_{L^2_T H^\ell_x (H^{s,*}_v)'} 
\end{aligned}
$$
and  
$$
\begin{aligned}
& \|  \Psi^{\eps,\sharp }[  g^\eps, g^\eps]\|_{\widetilde L^\infty_T H^{\frac12}_x L^2_v} 
+\frac{1}{\sqrt \eps}\| \mathbf P_0^\perp \Psi^{\eps,\sharp }[  g^\eps, g^\eps]\|_{L^2_T H^\frac32_x H^{s,*}_v}
+\| \mathbf P_0 \Psi^{\eps,\sharp }[  g^\eps, g^\eps]\|_{L^2_T H^\frac32_x H^{s,*}_v} \\
&\quad
\lesssim \eps \| \Gamma (g^\eps,g^\eps) \|_{\widetilde L^\infty_T H^{\frac12}_x L^2_v} + (\sqrt \eps + \eps)  \| \Gamma (g^\eps , g^\eps) \|_{L^2_T H^\frac32_x (H^{s,*}_v)'}\, ,
\end{aligned}
$$
and we conclude the proof by gathering these estimates.
\end{proof}
Before looking at the other contributions, let us remark that from~(\ref{decomposition Psieps 222}), for $ k = 0$, we have 
	\[
	\widehat \Psi^\eps[g^\eps,g^\eps](t,0) = \widehat \Psi^{\eps,\sharp}[g^\eps,g^\eps](t,0)\,,
	\]
it is thus enough to analyze the other contributions in \eqref{eq:source_decomposition} for $ k \in \Z^3 \setminus \{0\}$ i.e. for $| k| \geq 1$.

For the   term~$ \Psi^{\eps1,\flat}_{\rm wave} $ in \eqref{eq:source_decomposition}, we follow the arguments of~\cite{GT,CC}:  one needs to exploit the oscillations of the phase by integrations by parts in time. Thus with   notation inspired from~\cite{GT,CC} we define
$$
H_\pm^\eps(t,t',x):=  \mathcal{F}_x^{-1} \left( \chi\Big (\frac{\eps| k|}\kappa\Big) e^{-  \nu_{\rm \small wave\pm} (t-t') | k|^2} | k|\mathcal P_{\rm \small wave\pm}^1\Big({ k \over | k|}\Big) \widehat {\Gamma}( g^\eps, g^\eps)(t', k)\right)
$$
 	so that after an integration by parts in time
	\begin{align*}
		&\widehat \Psi^{\eps1, {\flat}} _{\rm wave} [g^\eps,g^\eps](t, k) \\
 &\quad =  
\sum_{\pm} \frac{\eps}{ic| k| } \Big(
	\int_0^t e^{  \pm i c  | k|  \frac {t-t'}{\eps}}  \partial_{t'} \widehat  H_\pm^\eps(t,t', k) \, \d t' 
	 - \widehat  H_\pm^\eps(t,t, k) + e^{  \pm i c  | k|  \frac {t}{\eps}}  \widehat  H_\pm^\eps(t,0, k)\Big)\, .
 	\end{align*}
	 Let us define
	 \begin{equation}\label{eq:Jeps}
	 \widehat{J^\eps_\pm}(t,k):=  \chi\Big (\frac{\eps| k|}\kappa\Big) \frac{\eps}{ic }\int_0^t e^{  \pm i c  | k|  \frac {t-t'}{\eps}-  \nu_{\rm \small wave
	\pm} (t-t') | k|^2} \mathcal P_{\rm \small wave\pm}^1\Big({ k \over | k|}\Big) \partial_{t'} \widehat {\Gamma}( g^\eps, g^\eps)(t', k) \, \d t'
	\end{equation}
	and
		 \begin{equation}\label{eq:Ieps}
  \widehat{I^\eps_\pm}(t,k)	:= \widehat \Psi^{\eps1, {\flat}}_{\rm wave} [  g^\eps, g^\eps] (t,k)-	 \widehat{J^\eps_\pm}(t,k)\, ,
 \end{equation}
which will be estimated separately.

For the term $I^\eps_{\pm}$ we have: 
\begin{lem}\label{lem:source_Psi_eps1_wave_I}
There holds
$$
\begin{aligned}
\|  I^{\eps}_{\pm} \|_{\mathcal X^\eps_T}
&\lesssim \eps^{1+\beta} \| \Gamma (g^\eps, g^\eps) \|_{\widetilde L^\infty_T H^\ell_x L^2_v} 
+  \eps^{\frac12+\beta}  \left(\| \Gamma (g^\eps, g^\eps) \|_{L^2_T H^{\ell}_x (H^{s,*}_v)'} + \| \Gamma (g^\eps_{\mathrm{in}}, g^\eps_{\mathrm{in}}) \|_{ H^{\ell}_x (H^{s,*}_v)'}\right)\\
&\quad
+ \eps \| \Gamma (g^\eps, g^\eps) \|_{\widetilde L^\infty_T H^{\frac12}_x L^2_v}
+ \sqrt \eps \left( \| \Gamma (g^\eps, g^\eps) \|_{L^2_T H^\frac32_x (H^{s,*}_v)'}+ \| \Gamma (g^\eps_{\mathrm{in}}, g^\eps_{\mathrm{in}}) \|_{ H^\frac32_x (H^{s,*}_v)'} \right) \,  .
\end{aligned}
$$
\end{lem}
 
\begin{proof} 
 
Since $\mathcal P_{\rm \small wave\pm}^1$ is bounded from $L^2_v$ into $L^2_v$ as well as from $(H^{s,*}_v)'$ into $H^{s,*}_v$, we obtain from ~\cite[Proof of Lemma~6.5]{CC} that, for all $t \in [0,T]$ and $k \in \Z^3 \setminus \{ 0 \}$,
$$
\begin{aligned}
\| \widehat{I^\eps_{\pm}} (t,k) \|_{L^2_v} 
&\lesssim \eps  \int_0^t |k|^2 e^{-\nu_{\rm \small wave\pm} (t-t') | k|^2}  \|\widehat \Gamma( g^\eps, g^\eps)(t',k)  \|_{ L^2_v } \, \d t'  \\
&\quad
+\eps  \|\widehat \Gamma( g^\eps, g^\eps)(t,k)  \|_{ L^2_v }
+ \eps   e^{-\nu_{\rm \small wave\pm} t| k|^2}  \|\widehat \Gamma (g^\eps_{\mathrm{in}}, g^\eps_{\mathrm{in}}) (k)  \|_{ L^2_v } \,  ,
\end{aligned}
$$
and
$$
\begin{aligned}
\| \widehat{I^\eps_{\pm}} (t,k) \|_{H^{s,*}_v} 
&\lesssim \eps  \int_0^t |k|^2 e^{-\nu_{\rm \small wave\pm} (t-t') | k|^2}  \|\widehat \Gamma( g^\eps, g^\eps)(t',k)  \|_{ (H^{s,*}_v)' } \, \d t'  \\
&\quad
+\eps  \|\widehat \Gamma( g^\eps, g^\eps)(t,k)  \|_{(H^{s,*}_v)' }
+ \eps   e^{-\nu_{\rm \small wave\pm} t| k|^2}  \|\widehat \Gamma( g^\eps_{\mathrm{in}}, g^\eps_{\mathrm{in}})(k)  \|_{ (H^{s,*}_v)' } \,  .
\end{aligned}
$$
We recall that~$k \neq 0$.
Applying Young's convolution in time $L^1_T * L^\infty_T \subset L^\infty_T $ and, respectively, $L^1_T * L^2_T \subset L^2_T $, we therefore obtain
$$
\begin{aligned}
\| \widehat{I^\eps_{\pm}} (k) \|_{L^\infty_T L^2_v} 
&\lesssim \eps \|\Gamma( g^\eps, g^\eps) (k) \|_{ L^\infty_T L^2_v }
\end{aligned}
$$
and
$$
\begin{aligned}
\| \widehat{I^\eps_{\pm}} (k) \|_{L^2_T H^{s,*}_v} 
&\lesssim \eps \|\Gamma( g^\eps, g^\eps) (k) \|_{ L^2_T (H^{s,*}_v)' }
+ \eps  \|\widehat \Gamma( g^\eps_{\mathrm{in}}, g^\eps_{\mathrm{in}})(k)  \|_{ (H^{s,*}_v)' }\,  .
\end{aligned}
$$ 
This implies
$$
\begin{aligned}
& \eps^{\beta} \|  I^\eps_{\pm} \|_{\widetilde L^\infty_T H^\ell_x L^2_v} 
+\frac{\eps^\beta}{\sqrt\eps}\| \mathbf P_0^\perp I^\eps_{\pm} \|_{L^2_T H^\ell_x H^{s,*}_v}
+\eps^{\beta} \| \mathbf P_0 I^\eps_{\pm} \|_{L^2_T H^\ell_x H^{s,*}_v} \\
&\quad
\lesssim 
\eps^{1+\beta} \| \Gamma (g^\eps,g^\eps) \|_{\widetilde L^\infty_T H^{\ell}_x L^2_v} + 
\left( \eps^{\frac12+\beta} + \eps^{1+ \beta} \right)  \left( \| \Gamma (g^\eps , g^\eps) \|_{L^2_T H^\ell_x (H^{s,*}_v)'} + \| \Gamma (g^\eps_{\mathrm{in}}, g^\eps_{\mathrm{in}}) \|_{ H^{\ell}_x (H^{s,*}_v)'}\right)\,  , 
\end{aligned}
$$
as well as
$$
\begin{aligned}
&  \|  I^\eps_{\pm} \|_{\widetilde L^\infty_T H^{\frac12}_x L^2_v} 
+\frac{1}{\sqrt\eps}\| \mathbf P_0^\perp I^\eps_{\pm} \|_{L^2_T H^\frac32_x H^{s,*}_v}
+ \| \mathbf P_0 I^\eps_{\pm} \|_{L^2_T H^\frac32_x H^{s,*}_v} \\
&\quad
\lesssim 
\eps \| \Gamma (g^\eps,g^\eps) \|_{\widetilde L^\infty_T H^{\frac12}_x L^2_v} + \left( \sqrt \eps + \eps \right)  \left( \| \Gamma (g^\eps , g^\eps) \|_{L^2_T H^\frac32_x (H^{s,*}_v)'} + \| \Gamma (g^\eps_{\mathrm{in}}, g^\eps_{\mathrm{in}}) \|_{ H^\frac32_x (H^{s,*}_v)'}\right)\,  .
\end{aligned}
$$ 
Lemma~\ref{lem:source_Psi_eps1_wave_I} is proved. \end{proof}
Recalling the definition of $J^\eps_{\pm}$ in \eqref{eq:Jeps}, we have the following result.
\begin{lem}\label{lem:source_Psi_eps1_wave_J}
There holds
$$
\begin{aligned}
\|  J^{\eps}_{\pm}  \|_{\mathcal X^\eps_T}
&\lesssim \eps^{1+\beta} \| \Gamma_{\mathrm{sym}} (g^\eps, \partial_t g^\eps) \|_{\widetilde L^\infty_T H^{\ell-2}_x L^2_v} 
+  \eps^{\frac12+\beta}  \| \Gamma_{\mathrm{sym}} (g^\eps, \partial_t g^\eps) \|_{L^2_T H^{\ell-2}_x (H^{s,*}_v)'} \\
&\quad
+ \eps \| \Gamma_{\mathrm{sym}} (g^\eps, \partial_t g^\eps) \|_{\widetilde L^\infty_T H^{-\frac32}_x L^2_v}
+ \sqrt \eps  \| \Gamma_{\mathrm{sym}} (g^\eps, \partial_t g^\eps) \|_{L^2_T H^{-\frac12}_x (H^{s,*}_v)'}  \, .
\end{aligned}
$$
\end{lem}
\begin{proof}
Starting from \eqref{eq:Jeps}, we use the fact that $\mathcal P_{\rm \small wave\pm}^1$ is bounded from $L^2_v$ into $L^2_v$ as well as from $(H^{s,*}_v)'$ into $H^{s,*}_v$ to obtain that, for all $t \in [0,T]$ and $k \in \Z^3 \setminus \{ 0 \}$,
$$
\| \widehat{J^\eps_\pm}(t,k) \|_{L^2_v} 
\lesssim  \eps \int_0^t |k|^{2} e^{- \nu_{\rm \small wave \pm} (t-t') | k|^2} |k|^{-2} \| \partial_{t} \widehat {\Gamma}( g^\eps, g^\eps)(t', k) \|_{L^2_v} \, \d t'
$$
and
$$
\| \widehat{J^\eps_\pm}(t,k) \|_{H^{s,*}_v} 
\lesssim  \eps \int_0^t |k|^{2} e^{- \nu_{\rm \small wave \pm} (t-t') | k|^2} |k|^{-2}\| \partial_{t} \widehat {\Gamma}( g^\eps, g^\eps)(t', k) \|_{(H^{s,*}_v)'} \, \d t' \, .
$$
Using Young's inequality for convolutions as in the proof of Lemma~\ref{lem:source_Psi_eps1_wave_I} together with the fact that $\partial_{t} \widehat {\Gamma}( g^\eps, g^\eps) = \widehat {\Gamma}( \partial_{t} g^\eps, g^\eps) +  \widehat {\Gamma}( g^\eps, \partial_{t} g^\eps)$ , we thus deduce
$$
\| \widehat{J^\eps_\pm}(t,k) \|_{L^\infty_T L^2_v} 
\lesssim  \eps |k|^{-2} \| \widehat{\Gamma}_{\mathrm{sym}}( g^\eps, \partial_{t} g^\eps)(k) \|_{L^\infty_T L^2_v}
$$
and
$$
\| \widehat{J^\eps_\pm}(t,k) \|_{L^2_T H^{s,*}_v} 
\lesssim  \eps |k|^{-2} \| \widehat{\Gamma}_{\mathrm{sym}}( g^\eps, \partial_{t} g^\eps)(k) \|_{L^2_T (H^{s,*}_v)'} \, .
$$
These two inequalities imply
$$
\begin{aligned}
& \eps^{\beta} \|  J^\eps_{\pm} \|_{\widetilde L^\infty_T H^\ell_x L^2_v} 
+\frac{\eps^\beta}{\sqrt\eps}\| \mathbf P_0^\perp J^\eps_{\pm} \|_{L^2_T H^\ell_x H^{s,*}_v}
+\eps^{\beta} \| \mathbf P_0 J^\eps_{\pm} \|_{L^2_T H^\ell_x H^{s,*}_v} \\
&\quad
\lesssim 
\eps^{1+\beta} \| \Gamma_{\mathrm{sym}} (g^\eps,\partial_t g^\eps) \|_{\widetilde L^\infty_T H^{\ell-2}_x L^2_v} + 
\left( \eps^{\frac12+\beta} + \eps^{1+ \beta} \right)  \| \Gamma_{\mathrm{sym}} (g^\eps , \partial_t g^\eps) \|_{L^2_T H^{\ell-2}_x (H^{s,*}_v)'} \, ,
\end{aligned}
$$
and also
$$
\begin{aligned}
&  \| J^\eps_{\pm} \|_{\widetilde L^\infty_T H^{\frac12}_x L^2_v} 
+\frac{1}{\sqrt\eps}\| \mathbf P_0^\perp J^\eps_{\pm} \|_{L^2_T H^\frac32_x H^{s,*}_v}
+ \| \mathbf P_0 J^\eps_{\pm} \|_{L^2_T H^\frac32_x H^{s,*}_v} \\
&\quad
\lesssim 
\eps \| \Gamma_{\mathrm{sym}} (g^\eps, \partial_t g^\eps) \|_{\widetilde L^\infty_T H^{-\frac32}_x L^2_v} + \left( \sqrt \eps + \eps \right)   \| \Gamma_{\mathrm{sym}} (g^\eps , \partial_t g^\eps) \|_{L^2_T H^{-\frac12}_x (H^{s,*}_v)'} \, .
\end{aligned}
$$ 
Lemma~\ref{lem:source_Psi_eps1_wave_J} is proved. \end{proof}
The contributions of the terms~$ \Psi^{\eps 1, \sharp}_{\rm NSF}$, $\widetilde \Psi^{\eps1,\flat}$ and~$\Psi^{\eps2,\flat}$ in \eqref{eq:source_decomposition} are easier to   obtain.

\begin{lem}\label{lem:source_remainder}
There holds
$$
\begin{aligned}
&\|  \Psi^{\eps 1, \sharp}_{\rm NSF}[g^\eps,g^\eps] \|_{\mathcal X^\eps_T} 
+ \| \widetilde \Psi^{\eps1,\flat} [g^\eps,g^\eps] \|_{\mathcal X^\eps_T}
+ \|  \Psi^{\eps2,\flat} [g^\eps,g^\eps] \|_{\mathcal X^\eps_T} \\
&\quad
\lesssim \eps^{1+\beta} \| \Gamma (g^\eps, g^\eps) \|_{\widetilde L^\infty_T H^\ell_x L^2_v} 
+  \eps^{\frac12+\beta}  \| \Gamma (g^\eps, g^\eps) \|_{L^2_T H^{\ell}_x (H^{s,*}_v)'} \\
&\quad\quad
+ \eps \| \Gamma (g^\eps, g^\eps) \|_{\widetilde L^\infty_T H^{\frac12}_x L^2_v}
+ \sqrt \eps  \| \Gamma (g^\eps, g^\eps) \|_{L^2_T H^\frac32_x (H^{s,*}_v)'} \,   .
\end{aligned}
$$
\end{lem}

\begin{proof}
Recall that $\mathcal P_{\rm \small wave\pm}^1$ is bounded from $L^2_v$ into $L^2_v$ as well as from $(H^{s,*}_v)'$ into $H^{s,*}_v$. Let $t \in [0,T]$ and $k \in \Z^3 \setminus \{ 0 \}$. 
Remarking that
$$
\left| \chi \left( \frac{\eps |k|}{\kappa} \right) - 1 \right| \lesssim \min \left( 1 , \eps |k| \right),
$$
we obtain that, where we denote $\nu_0 := \min (\nu_{\mathrm{NS}} , \nu_{\mathrm{heat}}) >0$,
$$
\begin{aligned}
\| \widehat \Psi^{\eps 1, \sharp}_{\rm NSF}[g^\eps,g^\eps] (t,k) \|_{L^2_v} 
&\lesssim \eps  \int_0^t |k|^2 e^{-\nu_{0} (t-t') | k|^2}  \|\widehat \Gamma( g^\eps, g^\eps)(t',k)  \|_{ L^2_v } \, \d t'  ,
\end{aligned}
$$
and
$$
\begin{aligned}
\| \widehat \Psi^{\eps 1, \sharp}_{\rm NSF}[g^\eps,g^\eps] (t,k) \|_{H^{s,*}_v} 
&\lesssim \eps  \int_0^t |k|^2 e^{-\nu_{0} (t-t') | k|^2}  \|\widehat \Gamma( g^\eps, g^\eps)(t',k)  \|_{ (H^{s,*}_v)' } \, \d t' . \end{aligned}
$$
Using that
$$
	\chi\Big (\frac{\eps| k|}\kappa\Big) e^{- \nu_\star t | k|^2} \left|e^{t \frac{\gamma_\star(\eps| k|)}{\eps^2}}-1 \right| | k|
	\lesssim \chi\Big (\frac{\eps| k|}\kappa\Big) e^{-  \frac{\nu_\star  }2 t | k|^2} t \eps  | k|^4
	\lesssim \eps  | k|^2 \chi\Big (\frac{\eps| k|}\kappa\Big) e^{-\frac{\nu_\star }{4} t | k|^2}\,,
$$
we also get, denoting  $\nu_1 = \min (\nu_{\mathrm{NS}} , \nu_{\mathrm{heat}} , \nu_{\mathrm{wave} \pm}) >0$,
$$
\begin{aligned}
\|   \widehat{\widetilde \Psi}\vphantom\Psi^{\eps1,\flat} [g^\eps,g^\eps] (t,k) \|_{L^2_v} 
&\lesssim \eps  \int_0^t |k|^2 e^{-\frac{\nu_{1}}{4} (t-t') | k|^2}  \|\widehat \Gamma( g^\eps, g^\eps)(t',k)  \|_{ L^2_v } \, \d t'  ,
\end{aligned}
$$
and
$$
\begin{aligned}
\|  \widehat{\widetilde \Psi}\vphantom\Psi^{\eps1,\flat} [g^\eps,g^\eps] (t,k) \|_{H^{s,*}_v} 
&\lesssim \eps  \int_0^t |k|^2 e^{-\frac{\nu_{1}}{4} (t-t') | k|^2}  \|\widehat \Gamma( g^\eps, g^\eps)(t',k)  \|_{ (H^{s,*}_v)' } \, \d t' . \end{aligned}
$$
Finally, observing that
$$
	\chi\Big (\frac{\eps| k|}\kappa\Big) e^{- \nu_\star t | k|^2 + t \frac{\gamma_\star(\eps| k|)}{\eps^2}} \eps | k|^2
	\lesssim \chi\Big (\frac{\eps| k|}\kappa\Big) e^{-  \frac{\nu_\star }2 t | k|^2} \eps  | k|^2
	\lesssim \eps | k|^2 \chi\Big (\frac{\eps| k|}\kappa\Big) e^{-\frac{\nu_\star}{2} t | k|^2} ,
$$
we also deduce
$$
\begin{aligned}
\|   \widehat{\widetilde \Psi}\vphantom\Psi^{\eps1,\flat}[g^\eps,g^\eps] (t,k) \|_{L^2_v} 
&\lesssim \eps  \int_0^t |k|^2 e^{-\frac{\nu_{1}}{2} (t-t') | k|^2}  \|\widehat \Gamma( g^\eps, g^\eps)(t',k)  \|_{ L^2_v } \, \d t'  ,
\end{aligned}
$$
and
$$
\begin{aligned}
\|   \widehat{\widetilde \Psi}\vphantom\Psi^{\eps1,\flat}[g^\eps,g^\eps] (t,k) \|_{H^{s,*}_v} 
&\lesssim \eps  \int_0^t |k|^2 e^{-\frac{\nu_{1}}{2} (t-t') | k|^2}  \|\widehat \Gamma( g^\eps, g^\eps)(t',k)  \|_{ (H^{s,*}_v)' } \, \d t' . \end{aligned}
$$
{The term $\Psi^{\eps 2,\flat}$ is treated in the same way using the boundedness properties of $\mathcal P^2_\star(\eps k)$ uniformly in $\eps |k| \leq \kappa$.} We can then conclude by arguing as in the proof of Lemma~\ref{lem:source_Psi_eps1_wave_I}.
\end{proof}

We can now gather the contributions of Lemmas~\ref{lem:source_Psi_eps_sharp}, \ref{lem:source_Psi_eps1_wave_I}, \ref{lem:source_Psi_eps1_wave_J} and \ref{lem:source_remainder} to conclude the proof of Proposition~\ref{estimates on delta eps}--(3).
We first observe that from Proposition~\ref{prop:nonlinear_macro}, since $\P_0 g^\eps = g^\eps$, we obtain 
$$
\begin{aligned}
\| \Gamma (g^\eps,g^\eps) \|_{\widetilde L^\infty_T H^{\ell}_x L^2_v} 
&\lesssim \| g^\eps \|_{\widetilde L^\infty_T H^{\ell}_x L^2_v}^2 \,,
\end{aligned}
$$
and from Proposition~\ref{prop:nonlinear} we have
$$
\begin{aligned}
\| \Gamma (g^\eps , g^\eps) \|_{L^2_T H^\ell_x (H^{s,*}_v)'}
&\lesssim \| g^\eps \|_{\widetilde L^\infty_T H^\ell_x L^2_v} \| g^\eps \|_{L^2_T H^\ell_x L^2_v}.
\end{aligned}
$$
Moreover from Proposition~\ref{prop:nonlinear_macro} again, we obtain 
$$
\begin{aligned}
\| \Gamma (g^\eps,g^\eps) \|_{\widetilde L^\infty_T H^{\frac12}_x L^2_v} 
&\lesssim \| g^\eps \|_{\widetilde L^\infty_T H^{\frac12}_x L^2_v} \|g^\eps\|_{\widetilde L^\infty_T H^\frac32_x L^2_v} \,,
\end{aligned}
$$
and thanks to Proposition~\ref{prop:nonlinear} again, we also get
$$
\begin{aligned}
\| \Gamma (g^\eps , g^\eps) \|_{L^2_T H^\frac32_x (H^{s,*}_v)'} 
&\lesssim \| g^\eps \|_{\widetilde L^\infty_T H^{\frac12}_x L^2_v} \|g^\eps\|_{L^2_T H^{\frac52}_x L^2_v} \, .
\end{aligned}
$$
By Proposition~\ref{prop:nonlinear_macro} there holds
$$
\begin{aligned}
\| \Gamma (g^\eps_{\mathrm{in}} , g^\eps_{\mathrm{in}}) \|_{H^\ell_x (H^{s,*}_v)'}
&\lesssim \| g^\eps_{\mathrm{in}} \|_{H^\ell_x L^2_v}^2 \, .
\end{aligned}
$$
and
$$
\begin{aligned}
\| \Gamma (g^\eps_{\mathrm{in}} , g^\eps_{\mathrm{in}} ) \|_{H^\frac32_x (H^{s,*}_v)'} 
&\lesssim \| g^\eps \|_{H^{\frac12}_x L^2_v} \|g^\eps\|_{H^{\frac52}_x L^2_v} \, .
\end{aligned}
$$
Therefore estimates~\eqref{eq:g_bound} and~\eqref{smallness geps} together with Lemma~\ref{lem:source_Psi_eps_sharp} yield
\begin{equation} \label{limit source1}
\begin{aligned}
\|  \Psi^{\eps,\sharp }  [  g^\eps, g^\eps]\|_{\mathcal X^\eps_T}
&\lesssim  \left( \eps^{\frac12+\beta -2\alpha(\ell-1)} + \eps^{\frac12-\alpha} \right) \|{ g_{\rm in}}\|_{H^{\frac12}_x L^2_v}^2  \exp \left( C \|{ g_{\rm in}}\|_{H^{\frac12}_x L^2_v}^2 \right) .
\end{aligned}
\end{equation}
Moreover, estimates~\eqref{eq:g_bound} and~\eqref{smallness geps} together with Lemma~\ref{lem:source_Psi_eps1_wave_I} and Lemma~\ref{lem:source_remainder} yield
\begin{equation} \label{limit source2}
\begin{aligned}
&\| I^\eps_{\pm} \|_{\mathcal X^\eps_T} 
+ \|  \Psi^{\eps 1, \sharp}_{\rm NSF}[g^\eps,g^\eps] \|_{\mathcal X^\eps_T} 
+ \|  \Psi^{\eps1,\flat} [g^\eps,g^\eps] \|_{\mathcal X^\eps_T}
+ \|  \Psi^{\eps2,\flat} [g^\eps,g^\eps] \|_{\mathcal X^\eps_T} \\
&\quad
\lesssim  \left(  
\eps^{\frac12+\beta -2\alpha(\ell-1)} 
+ \eps^{\frac12-\alpha} \right) \|{ g_{\rm in}}\|_{H^{\frac12}_x L^2_v}^2  \exp \left( C \|{ g_{\rm in}}\|_{H^{\frac12}_x L^2_v}^2 \right) \\
&\quad\quad
+ \left( \eps^{\frac12+\beta-2\alpha (\ell-\frac12)} + \eps^{\frac12- 2 \alpha} \right) \|{ g_{\rm in}}\|_{H^{\frac12}_x L^2_v}^2  \, .
\end{aligned}
\end{equation}
It only remains to investigate the contribution of the term $J^\eps_{\pm}$, and we recall that $3/2 < \ell \le 2$. 
We write
$$
\begin{aligned}
&\| \Gamma_{\mathrm{sym}} (g^\eps, \partial_t g^\eps) \|_{\widetilde L^\infty_T H^{\ell-2}_x L^2_v} + \| \Gamma_{\mathrm{sym}} (g^\eps, \partial_t g^\eps) \|_{\widetilde L^\infty_T H^{-\frac32}_x L^2_v} \\
&\quad\lesssim 
\| \Gamma_{\mathrm{sym}} (g^\eps, \partial_t g^\eps) \|_{\widetilde L^\infty_T L^2_x L^2_v} \\
&\quad\lesssim 
\| g^\eps \|_{\widetilde L^\infty_T H^{\ell}_x L^2_v} \| \partial_t g^\eps \|_{\widetilde L^\infty_T L^2_x L^2_v} 
\end{aligned}
$$
thanks to Proposition~\ref{prop:nonlinear_macro}, and also
$$
\begin{aligned}
&\| \Gamma_{\mathrm{sym}} (g^\eps , \partial_t g^\eps) \|_{L^2_T H^{\ell-2}_x (H^{s,*}_v)'} + \| \Gamma_{\mathrm{sym}} (g^\eps , \partial_t g^\eps) \|_{L^2_T H^{-\frac12}_x (H^{s,*}_v)'} \\
&\quad\lesssim 
\| \Gamma_{\mathrm{sym}} (g^\eps , \partial_t g^\eps) \|_{L^2_T L^2_x (H^{s,*}_v)'} \\
&\quad\lesssim 
\| g^\eps \|_{\widetilde L^\infty_T H^\ell_x L^2_v} \| \partial_t g^\eps \|_{L^2_T L^2_x L^2_v} \, ,
\end{aligned}
$$
using Proposition~\ref{prop:nonlinear}.
We now observe from \eqref{eq:NSF} that, for all $t \ge 0$ and $k \in \Z^3$,
$$
|\partial_t \widehat{g^\eps} (t,k)| \lesssim |k|^2 |\widehat{g^\eps} (t,k)| + |k| \left| \left( \widehat{g^\eps} (t,\cdot) * \widehat{g^\eps} (t,\cdot) \right) (k) \right|\, .
$$
We hence compute
$$
\| \partial_t g^\eps \|_{\widetilde L^\infty_T L^2_x L^2_v} \lesssim \| g^\eps \|_{\widetilde L^\infty_T H^{2}_x L^2_v} 
+ \| (g^\eps)^2 \|_{\widetilde L^\infty_T H^{1}_x L^2_v}\,  .
$$
Arguing as in the proof of Proposition~\ref{prop:nonlinear}, we have
$$
\| (g^\eps)^2 \|_{\widetilde L^\infty_T H^{1}_x L^2_v} 
\lesssim \| g^\eps \|_{\widetilde L^\infty_T H^{\frac12}_x L^2_v} \| g^\eps \|_{\widetilde L^\infty_T H^{2}_x L^2_v}\, ,
$$
and thus
$$
\| \partial_t g^\eps \|_{\widetilde L^\infty_T L^2_x L^2_v} 
\lesssim \left( 1+ \| g^\eps \|_{\widetilde L^\infty_T H^{\frac12}_x L^2_v} \right) \| g^\eps \|_{\widetilde L^\infty_T H^{2}_x L^2_v} \, .
$$
Similarly we obtain
$$
\| \partial_t g^\eps \|_{L^2_T L^2_x L^2_v} 
\lesssim \| g^\eps \|_{L^2_T H^{2}_x L^2_v} 
+ \| (g^\eps)^2 \|_{L^2_T H^{1}_x L^2_v}\,  ,
$$
then we get, as in the proof of Proposition~\ref{prop:nonlinear},
$$
\| (g^\eps)^2 \|_{L^2_T H^1_x L^2_v} 
\lesssim \| g^\eps \|_{\widetilde L^\infty_T H^{\frac12}_x L^2_v} \| g^\eps \|_{L^2_T H^{2}_x L^2_v} \, ,
$$
and finally
$$
\| \partial_t g^\eps \|_{L^2_T L^2_x L^2_v} 
\lesssim \left( 1+ \| g^\eps \|_{\widetilde L^\infty_T H^{\frac12}_x L^2_v} \right) \| g^\eps \|_{L^2_T H^{2}_x L^2_v} \, .
$$
Therefore estimates~\eqref{eq:g_bound} and~\eqref{smallness geps} together with Lemma~\ref{lem:source_Psi_eps1_wave_J} yield
\begin{equation} \label{limit source3}
\begin{aligned}
\| J^\eps_{\pm} \|_{\mathcal X^\eps_T}
&\lesssim  \left(  
\eps^{1-\alpha(\ell+1)} 
+ \eps^{\frac12-\alpha\ell}  \right) \\
&\qquad \times
\|{ g_{\rm in}}\|_{H^{\frac12}_x L^2_v}^2 \left(1+ \|{ g_{\rm in}}\|_{H^{\frac12}_x L^2_v} \right)  \exp \left( C \|{ g_{\rm in}}\|_{H^{\frac12}_x L^2_v}^2 \right) \,  .
\end{aligned}
\end{equation}
Finally gathering \eqref{limit source1}, \eqref{limit source2} and \eqref{limit source3} provides  (using the restrictions on~$\alpha,\ell$)
\begin{equation}\label{eq:source_final}
\begin{aligned}
\| \mathcal S^\eps \|_{\mathcal X^\eps_T} 
&\le 
\eps^{\frac12 - 2 \alpha} 
 \, \Phi \Big( \|{ g_{\rm in}}\|_{H^{\frac12}_x L^2_v} \Big)
\end{aligned}
\end{equation}
where $\Phi ( z) = C (1+z) z^2 e^{C z^2}$ and we have used the conditions on $\alpha$ and $\ell$. We hence obtain Proposition~\ref{estimates on delta eps}--(3).

\subsection{Estimates on the linear term $\mathcal L^{\eps}[\cdot]$}\label{sct:linear}

In this paragraph we prove Proposition~\ref{estimates on delta eps}--(4).
Consider $f \in \mathcal X_I^\eps$. 

\subsubsection{Linear estimates: high regularity}

First we have, from Corollary~\ref{cor:estimate_Psieps} and Proposition~\ref{prop:nonlinear} and  using additionally the fact that $g^\eps = \P_0 g^\eps$ and $H^{s,*}_v \subset L^2_v$,
$$
\begin{aligned}
&\frac{1}{\eps} \| \mathbf P_0^\perp \Psi^{\eps}[f,g^\eps] \|_{L^2_I H^\ell_x H^{s,*}_v}
+ \frac{1}{\eps} \| \Psi^{\eps,\sharp}[f,g^\eps] \|_{L^2_I H^\ell_x H^{s,*}_v}\\
&\qquad
\lesssim \| \Gamma_{\mathrm{sym}} [f,g^\eps]   \|_{L^2_I H^\ell_x (H^{s,*}_v)'} \\
&\qquad
\lesssim \| f \|_{L^2_I H^\ell_x H^{s,*}_v} \| g^\eps \|_{\widetilde L^\infty_I H^\ell_x L^2_v}\, .
\end{aligned}
$$
We then write $\mathbf P_0 \Psi^\eps = \mathbf P_0 \Psi^{\eps,\sharp} + \mathbf P_0 \Psi^{\eps,\flat}$. We compute, thanks to Proposition~\ref{prop:Psiepsflat} and Proposition~\ref{prop:nonlinear},
$$
\begin{aligned}
&\| \Psi^{\eps,\flat} [f,g^\eps] \|_{L^2_I H^\ell_x H^{s,*}_v} \\
&\qquad
\lesssim \| \Gamma_{\mathrm{sym}} [\mathbf P_0^\perp f,g^\eps]  + \Gamma_{\mathrm{sym}} [\mathbf P_0 f,g^\eps]   \|_{L^2_I H^{\ell-1}_x (H^{s,*}_v)'} \\
&\qquad
\lesssim \| \mathbf P_0^\perp f \|_{L^2_I H^\ell_x H^{s,*}_v} \| g^\eps \|_{\widetilde L^\infty_I H^{\frac12}_x L^2_v} 
+ \| \mathbf P_0^\perp f \|_{\widetilde L^\infty_I H^{\frac12}_x L^2_v} \| g^\eps \|_{L^2_I H^\ell_x L^2_v} \\
&\qquad \quad+  \| \mathbf P_0^\perp f \|_{L^2_I H^\frac32_x H^{s,*}_v} \| g^\eps \|_{\widetilde L^\infty_I H^{\ell}_x L^2_v} 
+ \| \mathbf P_0 f \|_{\widetilde L^4_I H^{\ell-\frac12}_x L^2_v} \| g^\eps \|_{\widetilde L^4_I H^{1}_x L^2_v} \\
&\qquad\quad
+ \| \mathbf P_0 f \|_{\widetilde L^\infty_I H^{\frac12}_x L^2_v} \| g^\eps \|_{L^2_I H^\ell_x L^2_v} \\
&\qquad
\lesssim 
\| \mathbf P_0^\perp f \|_{L^2_I H^\ell_x  { H^{s,*}_v }L^2_v} \| g^\eps \|_{\widetilde L^\infty_I H^{\frac12}_x L^2_v} 
+  \| \mathbf P_0^\perp f \|_{L^2_I H^\frac32_x H^{s,*}_v} \| g^\eps \|_{\widetilde L^\infty_I H^{\ell}_x L^2_v} 
\\
&\qquad\quad
+ \| f \|_{\widetilde L^\infty_I H^{\frac12}_x L^2_v} \| g^\eps \|_{L^2_I H^\ell_x L^2_v}
+ \| \mathbf P_0 f \|_{\widetilde L^\infty_I H^{ \ell-1}_x L^2_v}^{\frac12} \| \mathbf P_0 f \|_{L^2_I H^\ell_x H^{s,*}_v}^{\frac12} \| g^\eps \|_{\widetilde L^4_I H^{1}_x L^2_v} \,  ,
\end{aligned}
$$
where we have used the interpolation inequality~(\ref{eq:interpolation})  {as well as the fact that}
{
$$
\begin{aligned}
&\|g^\eps\|_{\widetilde L^\infty_I H^{\ell-1+0}_x L^2_v} \|\mathbf P_0^\perp f\|_{L^2_I H^{\frac32-0}_x H^{s,*}_v}
+ \|g^\eps\|_{\widetilde L^\infty_I H^{\ell-1}_x L^2_v} \|\mathbf P_0^\perp f\|_{L^2_I H^{\frac32}_x H^{s,*}_v} \\
&\quad \lesssim
\|g^\eps\|_{\widetilde L^\infty_I H^{\ell}_x L^2_v} \|\mathbf P_0^\perp f\|_{L^2_I H^{\frac32}_x H^{s,*}_v}\,.
\end{aligned}
$$
} 
Therefore we get
$$
\begin{aligned}
&\frac{\eps^\beta}{\sqrt\eps} \| \mathbf P_0^\perp \Psi^{\eps}[f,g^\eps] \|_{L^2_I H^\ell_x H^{s,*}_v} + \eps^\beta \| \mathbf P_0 \Psi^{\eps}[f,g^\eps] \|_{L^2_I H^\ell_x H^{s,*}_v} \\
&\qquad
\lesssim \eps^{\beta} \| f \|_{L^2_I H^\ell_x H^{s,*}_v} \sqrt\eps \| g^\eps \|_{\widetilde L^\infty_I H^\ell_x L^2_v}
+ \frac{\eps^\beta}{\sqrt\eps} \| \mathbf P_0^\perp f \|_{L^2_I H^\ell_x H^{s,*}_v} \sqrt\eps \| g^\eps \|_{\widetilde L^\infty_I H^{\frac12}_x L^2_v}\\
&\qquad\quad
+ \| f \|_{\widetilde L^\infty_I H^{\frac12}_x L^2_v} \eps^\beta \| g^\eps \|_{L^2_I H^\ell_x L^2_v}
+ \frac{1}{\sqrt\eps} \| \mathbf P_0^\perp f \|_{L^2_I H^\frac32_x H^{s,*}_v} \eps^{\beta+\frac12} \| g^\eps \|_{\widetilde L^\infty_I H^\ell_x L^2_v}\\
&\qquad\quad
+ \eps^{\frac{\beta}{2}} \| \mathbf P_0 f \|_{\widetilde L^\infty_I H^\ell_x L^2_v}^{\frac12} \eps^{\frac{\beta}{2}} \| \mathbf P_0 f \|_{L^2_I H^\ell_x H^{s,*}_v}^{\frac12}\| g^\eps \|_{\widetilde L^4_I H^{1}_x L^2_v}\, .
\end{aligned}
$$
We now investigate the $\widetilde L^\infty_I$ norm by writing $$\Psi^\eps[f,g^\eps] = \Psi^\eps[\mathbf P_0^\perp f,g^\eps] + \Psi^{\eps,\sharp}[\mathbf P_0 f,g^\eps] + \Psi^{\eps,\flat}[\mathbf P_0 f,g^\eps]\, .$$ 
From Corollary~\ref{cor:estimate_Psieps} and Proposition~\ref{prop:nonlinear} we have
$$
\begin{aligned}
\eps^\beta \| \Psi^\eps [\mathbf P_0^\perp f,g^\eps] \|_{\widetilde L^\infty_I H^\ell_x L^2_v} 
&\lesssim \eps^\beta \| \Gamma_{\mathrm{sym}} [\mathbf P_0^\perp f,g^\eps]  \|_{L^2_I H^\ell_x (H^{s,*}_v)'} \\
&\lesssim \| \mathbf P_0^\perp f \|_{L^2_I H^\ell_x H^{s,*}_v} \eps^\beta \| g^\eps \|_{\widetilde L^\infty_I H^\ell_x L^2_v}\\
&\lesssim \frac{\eps^\beta}{\sqrt{\eps}}  \| \mathbf P_0^\perp f \|_{L^2_I H^\ell_x H^{s,*}_v} \eps^\beta \| g^\eps \|_{\widetilde L^\infty_I H^\ell_x L^2_v}
\end{aligned}
$$
since $\beta<1/2$.
Moreover from Proposition~\ref{prop:Psiepssharp} and using again that $\beta<1/2$, we have
$$
\begin{aligned}
\eps^\beta \| \Psi^{\eps,\sharp} [\mathbf P_0 f,g^\eps] \|_{\widetilde L^\infty_I H^\ell_x L^2_v} 
&\lesssim  \eps^{\beta+1} \| \Gamma_{\mathrm{sym}} [\mathbf P_0 f,g^\eps]  \|_{\widetilde L^\infty_I H^\ell_x L^2_v} \\
&\lesssim \eps^\beta \| \mathbf P_0 f \|_{\widetilde L^\infty_I H^\ell_x L^2_v} \eps^\beta \| g^\eps \|_{\widetilde L^\infty_I H^\ell_x L^2_v}\, .
\end{aligned}
$$
Applying Proposition~\ref{prop:Psiepsflat} and Proposition~\ref{prop:nonlinear_macro}--(2) together with interpolation inequality~\eqref{eq:interpolation}, we get
$$
\begin{aligned}
\eps^\beta \| \Psi^{\eps,\flat} [\mathbf P_0 f,g^\eps] \|_{\widetilde L^\infty_I H^\ell_x L^2_v} 
&\lesssim  \eps^\beta \| \Gamma_{\mathrm{sym}}[\mathbf P_0 f,g^\eps]  \|_{\widetilde L^4_I H^{\ell-\frac12}_x L^2_v} \\
&\lesssim  \eps^\beta \| \mathbf P_0 f \|_{\widetilde L^\infty_I H^\ell_x L^2_v} \| g^\eps \|_{\widetilde L^4_I H^1_x L^2_v}
+ \eps^\beta \| \mathbf P_0 f \|_{\widetilde L^4_I H^1_x L^2_v} \| g^\eps \|_{\widetilde L^\infty_I H^{\ell}_x L^2_v} \\
&\lesssim 
\eps^\beta \| \mathbf P_0 f \|_{\widetilde L^\infty_I H^\ell_x L^2_v} 
\| g^\eps \|_{\widetilde L^4_I H^{1}_x L^2_v} \\
&\quad
+ \| \mathbf P_0 f \|_{\widetilde L^\infty_I H^{\frac12}_x L^2_v}^{\frac12}
\| \mathbf P_0 f \|_{L^2_I H^\frac32_x L^2_v}^{\frac12}
\eps^\beta \| g^\eps \|_{\widetilde L^\infty_I H^\ell_x L^2_v}\,  .
\end{aligned}
$$
Gathering the previous estimates and using the fact that $\beta<1/2$ finally yield
\begin{equation}\label{linear_high}
\begin{aligned}
&\eps^\beta \| \Psi^\eps[f,g^\eps] \|_{\widetilde L^\infty_I H^\ell_x L^2_v} + \frac{\eps^\beta}{\sqrt\eps} \| \mathbf P_0^\perp \Psi^{\eps}[f,g^\eps] \|_{L^2_I H^\ell_x H^{s,*}_v} +\eps^{\beta} \| \mathbf P_0 \Psi^{\eps}[f,g^\eps] \|_{L^2_I H^\ell_x H^{s,*}_v} \\
&\qquad
\lesssim
\| f \|_{\mathcal X^\eps_I}  \Big( \eps^{\beta} \| g^\eps \|_{\widetilde L^\infty_I H^\ell_x L^2_v} 
+ \eps^{\beta}\| g^\eps \|_{L^2_I H^{\ell}_x L^2_v} 
+\| g^\eps \|_{\widetilde L^4_I H^{1}_x L^2_v} \Big) \, .
\end{aligned}
\end{equation}

\subsubsection{Linear estimates: low regularity}

We have from Corollary~\ref{cor:estimate_Psieps} and Proposition~\ref{prop:nonlinear}, as well as from Proposition~\ref{prop:Psiepsflat},
$$
\begin{aligned}
&\| \Psi^\eps[f,g^\eps] \|_{\widetilde L^\infty_I H^{\frac12}_x L^2_v} 
+ \| \Psi^{\eps,\flat} [f,g^\eps] \|_{L^2_I H^\frac32_x H^{s,*}_v} \\
&\quad
\lesssim \| \Gamma_{\mathrm{sym}} [\mathbf P_0^\perp f , g^\eps ]  + \Gamma_{\mathrm{sym}} [\mathbf P_0 f , g^\eps]   \|_{L^2_I H^{\frac12}_x (H^{s,*}_v)'} \\
&\quad
\lesssim \| \mathbf P_0^\perp f \|_{L^2_I H^\frac32_x H^{s,*}_v} \| g^\eps \|_{\widetilde L^\infty_I H^{\frac12}_x L^2_v} 
+ \| \mathbf P_0^\perp f \|_{\widetilde L^\infty_I H^{\frac12}_x L^2_v} \| g^\eps \|_{L^2_I H^\frac32_x L^2_v} \\
&\quad\quad
+\| \mathbf P_0^\perp f \|_{L^2_I H^{\frac32-0}_x H^{s,*}_v} \| g^\eps \|_{\widetilde L^\infty_I H^{\frac12+0}_x L^2_v}   \\
&\quad\quad
+\| \mathbf P_0 f \|_{\widetilde L^4_I H^{1}_x H^{s,*}_v} \| g^\eps \|_{\widetilde L^4_I H^{1}_x L^2_v} 
+ \| \mathbf P_0 f \|_{\widetilde L^\infty_I H^{\frac12}_x L^2_v} \| g^\eps \|_{L^2_I H^\frac32_x L^2_v} \, ,
\end{aligned}
$$
where we have used that that $g^\eps = \P_0 g^\eps$ and $H^{s,*}_v \subset L^2_v$. Therefore we obtain, using \eqref{eq:interpolation},
$$
\begin{aligned}
&\| \Psi^\eps[f,g^\eps] \|_{\widetilde L^\infty_I H^{\frac12}_x L^2_v} +\| \Psi^{\eps,\flat} [f,g^\eps] \|_{L^2_I H^\frac32_x H^{s,*}_v}\\
&\quad
\lesssim \|  f \|_{\widetilde L^\infty_I H^{\frac12}_x L^2_v} \| g^\eps \|_{L^2_I H^\frac32_x L^2_v}
+ \frac{1}{\sqrt\eps} \| \mathbf P_0^\perp f \|_{L^2_I H^\frac32_x H^{s,*}_v} \sqrt\eps \| g^\eps \|_{\widetilde L^\infty_I H^{\frac12+0}_x L^2_v} \\
&\quad\quad
+ \| \mathbf P_0 f \|_{\widetilde L^\infty_I H^{\frac12}_x L^2_v}^{\frac12} \| \mathbf P_0 f \|_{L^2_I H^\frac32_x H^{s,*}_v}^{\frac12} \| g^\eps \|_{\widetilde L^4_I H^{1}_x L^2_v}\, .
\end{aligned}
$$
Moreover, still from Corollary~\ref{cor:estimate_Psieps} and Proposition~\ref{prop:nonlinear},
$$
\begin{aligned}
&\| \mathbf P_0^\perp \Psi^{\eps}[f,g^\eps] \|_{L^2_I H^\frac32_x H^{s,*}_v}
+ \| \Psi^{\eps,\sharp}[f,g^\eps] \|_{L^2_I H^\frac32_x H^{s,*}_v} \\
&\quad
\lesssim \eps \| \Gamma_{\mathrm{sym}} [\mathbf P_0^\perp f , g^\eps ]  +  \Gamma_{\mathrm{sym}} [\mathbf P_0 f , g^\eps ]  \|_{L^2_I H^\frac32_x (H^{s,*}_v)'} \\
&\quad
\lesssim \eps \| \mathbf P_0^\perp f \|_{L^2_I H^\frac32_x H^{s,*}_v} \| g^\eps \|_{\widetilde L^\infty_I H^\ell_x L^2_v} 
+ \eps \| \mathbf P_0 f \|_{L^2_I H^\frac32_x H^{s,*}_v} \| g^\eps \|_{\widetilde L^\infty_I H^\ell_x L^2_v}\,   .
\end{aligned}
$$
This implies
$$
\begin{aligned}
\frac{1}{\sqrt \eps} \| \mathbf P_0^\perp \Psi^{\eps}[f,g^\eps] \|_{L^2_I H^\frac32_x H^{s,*}_v}
&\lesssim  \frac{1}{\sqrt \eps} \| \mathbf P_0^\perp f \|_{L^2_I H^\frac32_x H^{s,*}_v} \eps \| g^\eps \|_{\widetilde L^\infty_I H^\ell_x L^2_v} \\
&\quad
+ \| \mathbf P_0 f \|_{L^2_I H^\frac32_x H^{s,*}_v} \sqrt \eps\| g^\eps \|_{\widetilde L^\infty_I H^\ell_x L^2_v} \,  ,
\end{aligned}
$$
and also
$$
\begin{aligned}
\| \mathbf P_0 \Psi^{\eps,\sharp}[f,g^\eps] \|_{L^2_I H^\frac32_x H^{s,*}_v} 
&\lesssim  \frac{1}{\sqrt\eps}\| \mathbf P_0^\perp f \|_{L^2_I H^\frac32_x H^{s,*}_v}  \eps^{\frac32}\| g^\eps \|_{\widetilde L^\infty_I H^\ell_x L^2_v} \\
&\quad
+  \| \mathbf P_0 f \|_{L^2_I H^\frac32_x H^{s,*}_v} \eps \| g^\eps \|_{\widetilde L^\infty_I H^\ell_x L^2_v} \,  .
\end{aligned}
$$
Putting together the previous estimates provides
\begin{equation}\label{linear_low}
\begin{aligned}
&\| \Psi^\eps[f,g^\eps] \|_{\widetilde L^\infty_I H^{\frac12}_x L^2_v} 
+ \frac{1}{\sqrt\eps} \| \mathbf P_0^\perp \Psi^{\eps}[f,g^\eps] \|_{L^2_I H^\frac32_x H^{s,*}_v} 
+ \| \mathbf P_0 \Psi^{\eps}[f,g^\eps] \|_{L^2_I H^\frac32_x H^{s,*}_v} \\
&\quad
\lesssim
\|f\|_{\mathcal X^\eps_I}
\times \left( \sqrt \eps  \| g^\eps \|_{\widetilde L^\infty_I H^\ell_x L^2_v}  +\| g^\eps \|_{\widetilde L^4_I H^{1}_x L^2_v}
+ \| g^\eps \|_{L^2_I H^\frac32_x L^2_v} \right) \, .
\end{aligned}
\end{equation}

\subsubsection{Conclusion}

Gathering estimates~\eqref{linear_high} and~\eqref{linear_low} concludes the proof  of Proposition~\ref{estimates on delta eps}--(4) since $\beta<1/2$. 

 \subsection{Estimates on the nonlinear term~$ \Psi^{\eps}[\cdot,\cdot]$}\label{sct:nonlinear}
In this paragraph we prove Proposition~\ref{estimates on delta eps}--(5). 
 Consider $f_1$ and $f_2 \in \mathcal X_I^\eps$. 
\subsubsection{Nonlinear estimates: high regularity}

Corollary~\ref{cor:estimate_Psieps} and Proposition~\ref{prop:nonlinear} imply, using that $H^{s,*}_v \subset L^2_v$,
$$
\begin{aligned}
& \frac{1}{\eps} \| \mathbf P_0^\perp \Psi^{\eps}[f_1,f_2] \|_{L^2_I H^\ell_x H^{s,*}_v}
+ \frac{1}{\eps} \| \Psi^{\eps,\sharp}[f_1,f_2] \|_{L^2_I H^\ell_x H^{s,*}_v}\\
&\qquad
\lesssim \| \Gamma_{\mathrm{sym}}[f_1,f_2] \|_{L^2_I H^\ell_x (H^{s,*}_v)'} 
\\
&\qquad\lesssim
\| f_1 \|_{\widetilde L^\infty_I H^\ell_x L^2_v}\| f_2 \|_{L^2_I H^\ell_x H^{s,*}_v}
+ \| f_2 \|_{L^2_I H^\ell_x H^{s,*}_v}\| f_2 \|_{\widetilde L^\infty_I H^\ell_x L^2_v}\, .
\end{aligned}
$$
Moreover from Proposition~\ref{prop:Psiepsflat} and Proposition~\ref{prop:nonlinear}, we get
$$
\begin{aligned}
\| \Psi^{\eps,\flat} [f_1,f_2] \|_{L^2_I H^{\ell}_x H^{s,*}_v} 
& \lesssim \| \Gamma_{\mathrm{sym}}[f_1, f_2]   \|_{L^2_I H^{\ell-1}_x (H^{s,*}_v)'} \\
&\lesssim 
\| f_1 \|_{\widetilde L^\infty_I H^\ell_x L^2_v} \| f_2 \|_{L^2_I H^\frac32_x H^{s,*}_v}
+ \| f_1 \|_{\widetilde L^\infty_I H^{\frac12}_x L^2_v} \| f_2 \|_{L^2_I H^{\ell}_x H^{s,*}_v} \\
&\quad
+\| f_1 \|_{L^2_I H^\frac32_x H^{s,*}_v} \| f_2 \|_{\widetilde L^\infty_I H^\ell_x L^2_v}
+ \| f_1 \|_{L^2_I H^{\ell}_x H^{s,*}_v} \| f_2 \|_{\widetilde L^\infty_I H^{\frac12}_x L^2_v} \, .
\end{aligned}
$$
We now turn to the $L^\infty_I$ norm and we decompose 
$$
\begin{aligned}
\Psi^{\eps}[f_1,f_2] &= \Psi^{\eps}[\mathbf P_0^\perp f_1 , \mathbf P_0^\perp f_2]
+ \Psi^\eps[ \mathbf P_0^\perp f_1 , \mathbf P_0 f_2 ]
+ \Psi^\eps[ \mathbf P_0 f_1 , \mathbf P_0^\perp f_2 ]
\\
&\quad+ \Psi^{\eps,\sharp} [\mathbf P_0 f_1 , \mathbf P_0 f_2]
+ \Psi^{\eps,\flat} [\mathbf P_0 f_1 , \mathbf P_0 f_2]\, .
\end{aligned}
$$
Thanks to Corollary~\ref{cor:estimate_Psieps} and Proposition~\ref{prop:nonlinear} we obtain
$$
\begin{aligned}
&\| \Psi^{\eps}[ \mathbf P_0^\perp f_1, \mathbf P_0^\perp f_2] \|_{\widetilde L^\infty_I H^\ell_x L^2_v}
\lesssim \| \Gamma_{\mathrm{sym}}[ \mathbf P_0^\perp f_1,  \mathbf P_0^\perp f_2]   \|_{L^2_I H^\ell_x (H^{s,*}_v)'} 
\\
&\qquad 
\lesssim
\| \mathbf P_0^\perp f_1 \|_{\widetilde L^\infty_I H^\ell_x L^2_v}\| \mathbf P_0^\perp f_2 \|_{L^2_I H^\ell_x H^{s,*}_v}
+ \| \mathbf P_0^\perp f_1 \|_{L^2_I H^\ell_x H^{s,*}_v}\| \mathbf P_0^\perp f_2 \|_{\widetilde L^\infty_I H^\ell_x L^2_v}\, ,
\end{aligned}
$$
moreover
$$
\begin{aligned}
\| \Psi^{\eps}[ \mathbf P_0^\perp f_1, \mathbf P_0 f_2] \|_{\widetilde L^\infty_I H^\ell_x L^2_v}
&\lesssim \| \Gamma_{\mathrm{sym}}  [ \mathbf P_0^\perp f_1, \mathbf P_0 f_2 ] \|_{L^2_I H^\ell_x (H^{s,*}_v)'} 
\\
&\lesssim
\| \mathbf P_0^\perp f_1 \|_{L^2_I H^\ell_x H^{s,*}_v}\| \mathbf P_0 f_2 \|_{\widetilde L^\infty_I H^\ell_x L^2_v}\, ,
\end{aligned}
$$
and also
$$
\begin{aligned}
\| \Psi^{\eps}[ \mathbf P_0 f_1, \mathbf P_0^\perp f_2] \|_{\widetilde L^\infty_I H^\ell_x L^2_v}
&\lesssim \| \Gamma_{\mathrm{sym}} [\mathbf P_0 f_1, \mathbf P_0^\perp f_2 ]\|_{L^2_I H^\ell_x (H^{s,*}_v)'} 
\\
&\lesssim
\| \mathbf P_0 f_1 \|_{\widetilde L^\infty_I H^\ell_x L^2_v}\| \mathbf P_0^\perp f_2 \|_{L^2_I H^\ell_x H^{s,*}_v}\, .
\end{aligned}
$$
Moreover from Proposition~\ref{prop:Psiepssharp} we have
$$
\begin{aligned}
\| \Psi^{\eps,\sharp} [\mathbf P_0 f_1, \mathbf P_0 f_2] \|_{\widetilde L^\infty_I H^\ell_x L^2_v} 
&\lesssim  \eps \| \Gamma_{\mathrm{sym}}[\mathbf P_0 f_1, \mathbf P_0 f_2]   \|_{L^\infty_I H^\ell_x L^2_v} \\
&\lesssim \eps \| \mathbf P_0 f_1 \|_{\widetilde L^\infty_I H^\ell_x L^2_v} \| \mathbf P_0 f_2 \|_{\widetilde L^\infty_I H^\ell_x L^2_v}\, .
\end{aligned}
$$
Applying Proposition~\ref{prop:Psiepsflat} and Proposition~\ref{prop:nonlinear_macro}--(1) we get
$$
\begin{aligned}
&\| \Psi^{\eps,\flat} [\mathbf P_0 f_1, \mathbf P_0 f_2] \|_{\widetilde L^\infty_I H^\ell_x L^2_v} 
\lesssim   \| \Gamma_{\mathrm{sym}}[\mathbf P_0 f_1,\mathbf P_0 f_2]   \|_{\widetilde L^\infty_I H^{\ell-1}_x L^2_v} \\
&\qquad
\lesssim  \| \mathbf P_0 f_1 \|_{\widetilde L^\infty_I H^\ell_x L^2_v} \| \mathbf P_0 f_2 \|_{\widetilde L^\infty_I H^{\frac12}_x L^2_v}
+ \| \mathbf P_0 f_1 \|_{\widetilde L^\infty_I H^{\frac12}_x L^2_v} \| \mathbf P_0 f_2 \|_{\widetilde L^\infty_I H^{\ell}_x L^2_v} \,  .
\end{aligned}
$$
Gathering the previous estimates, we   finally deduce
\begin{equation}\label{nonlinear_high}
\begin{aligned}
&\eps^\beta \| \Psi^\eps[f_1,f_2] \|_{\widetilde L^\infty_I H^\ell_x L^2_v} 
+ \frac{\eps^\beta}{\sqrt\eps} \| \mathbf P_0^\perp \Psi^{\eps}[f_1,f_2] \|_{L^2_I H^\ell_x H^{s,*}_v}
+ \eps^{\beta}\| \mathbf P_0 \Psi^{\eps}[f_1,f_2] \|_{L^2_I H^\ell_x H^{s,*}_v} \\
&\qquad
\lesssim \| f_1 \|_{\mathcal X^\eps_I} \| f_2 \|_{\mathcal X^\eps_I}\, .
\end{aligned}
\end{equation}

\subsubsection{Nonlinear estimates: low regularity}

We first compute thanks to Corollary~\ref{cor:estimate_Psieps}
$$
\begin{aligned}
\| \Psi^\eps[f_1,f_2] \|_{\widetilde L^\infty_I H^{\frac12}_x L^2_v} 
&\lesssim \| \Gamma[f_1, \mathbf P_0^\perp f_2] + \Gamma[f_1, \mathbf P_0 f_2]  + \Gamma[f_2, \mathbf P_0^\perp f_1] + \Gamma[f_2, \mathbf P_0 f_1] \|_{L^2_I H^{\frac12}_x (H^{s,*}_v)'} 
\end{aligned}
$$
and
$$
\begin{aligned}
&\frac{1}{\eps} \| \mathbf P_0^\perp \Psi^\eps[f_1,f_2] \|_{L^2_I H^\frac32_x H^{s,*}_v}
+\frac{1}{\eps} \| \Psi^{\eps,\sharp}[f_1,f_2] \|_{L^2_I H^\frac32_x H^{s,*}_v}\\
&\quad
\lesssim \| \Gamma[f_1, \mathbf P_0^\perp f_2] + \Gamma[f_1, \mathbf P_0 f_2]+\Gamma[f_2, \mathbf P_0^\perp f_1] + \Gamma[f_2, \mathbf P_0 f_1] \|_{L^2_I H^\frac32_x (H^{s,*}_v)'}\,  .
\end{aligned}
$$
From Proposition~\ref{prop:nonlinear} we get
\begin{equation} \label{eq:L2TH12}
\begin{aligned}
&\| \Gamma[f_1, \mathbf P_0^\perp f_2] + \Gamma[f_1, \mathbf P_0 f_2] \|_{L^2_I H^{\frac12}_x (H^{s,*}_v)'} \\
&
\lesssim \| f_1 \|_{\widetilde L^\infty_I H^{\frac12+0}_x L^2_v}\| \mathbf P_0^\perp f_2 \|_{L^2_I H^{\frac32-0}_x H^{s,*}_v} 
 \\
&\quad
+ \| f_1 \|_{L^2_I H^\frac32_x L^2_v} \| \mathbf P_0 f_2 \|_{\widetilde L^\infty_I H^{\frac12}_x H^{s,*}_v}
+ \| f_1 \|_{\widetilde L^\infty_I H^{\frac12}_x L^2_v} \| \mathbf P_0 f_2 \|_{L^2_I H^\frac32_x H^{s,*}_v} \\
&\lesssim \eps^{\frac12-\beta} \left( \eps^\beta \| f_1 \|_{\widetilde L^\infty_I H^\ell_x L^2_v} \right) \left(\frac{1}{\sqrt \eps} \| \mathbf P_0^\perp f_2 \|_{L^2_I H^\frac32_x H^{s,*}_v} \right)\\
&\quad
+ \| f_1 \|_{L^2_I H^\frac32_x L^2_v} \| \mathbf P_0 f_2 \|_{\widetilde L^\infty_I H^{\frac12}_x H^{s,*}_v} + \| f_1 \|_{\widetilde L^\infty_I H^{\frac12}_x L^2_v} \| \mathbf P_0 f_2 \|_{L^2_I H^\frac32_x H^{s,*}_v}\,  .
\end{aligned}
\end{equation}
Moreover we also obtain
$$
\begin{aligned}
&\| \Gamma[f_1, \mathbf P_0^\perp f_2] + \Gamma[f_1, \mathbf P_0 f_2]\|_{L^2_I H^\frac32_x (H^{s,*}_v)'} \\
&\quad
\lesssim \| f_1 \|_{\widetilde L^\infty_I H^{\frac32+0}_x L^2_v} \| \mathbf P_0^\perp f_2 \|_{L^2_I H^{\frac32-0}_x H^{s,*}_v}
+ \| f_1 \|_{\widetilde L^\infty_I H^{\ell}_x L^2_v} \| \mathbf P_0^\perp f_2 \|_{L^2_I H^\frac32_x H^{s,*}_v} \\
&\quad\quad
+ \| f_1 \|_{L^2_I H^\frac32_x L^2_v} \| \mathbf P_0 f_2 \|_{\widetilde L^\infty_I H^{\ell}_x L^2_v}
+ \| f_1 \|_{L^2_I H^{\frac32-0}_x L^2_v} \| \mathbf P_0 f_2 \|_{\widetilde L^\infty_I H^{\ell}_x L^2_v}\,  ,
\end{aligned}
$$
from which we deduce
$$
\begin{aligned}
&\sqrt\eps \| \Gamma[f_1, \mathbf P_0^\perp f_2] + \Gamma[f_1, \mathbf P_0 f_2]\|_{L^2_I H^\frac32_x (H^{s,*}_v)'}\\
&\quad 
\lesssim \eps^{1-\beta} \left( \eps^\beta \| f_1 \|_{\widetilde L^\infty_I H^{\ell}_x L^2_v} \right)  \left( \frac{1}{\sqrt \eps}\| \mathbf P_0^\perp f_2 \|_{L^2_I H^\frac32_x H^{s,*}_v} \right)\\
&\quad\quad
+ \eps^{\frac12-\beta} \| f_1 \|_{L^2_I H^\frac32_x L^2_v} \left( \eps^\beta \| \mathbf P_0 f_2 \|_{\widetilde L^\infty_I H^{\ell}_x L^2_v} \right)\, .
\end{aligned}
$$
Furthermore Proposition~\ref{prop:Psiepsflat} yields
$$
\begin{aligned}
\| \Psi^{\eps,\flat}[f_1,f_2] \|_{L^2_I H^\frac32_x H^{s,*}_v}
&\lesssim \| \Gamma[f_1, \mathbf P_0^\perp f_2] + \Gamma[f_1, \mathbf P_0 f_2] \|_{L^2_I H^{\frac12}_x (H^{s,*}_v)'}\\
&\quad +\|\Gamma[f_2, \mathbf P_0^\perp f_1] + \Gamma[f_2, \mathbf P_0 f_1] \|_{L^2_I H^{\frac12}_x (H^{s,*}_v)'}\,  .
\end{aligned}
$$
We can then use~\eqref{eq:L2TH12} to bound $\Gamma[f_1, \mathbf P_0^\perp f_2] + \Gamma[f_1, \mathbf P_0 f_2]$ in $L^2_I H^{\frac12}_x (H^{s,*}_v)'$. 
Gathering the previous estimates, and observing that the terms~$\Gamma[f_2, \mathbf P_0^\perp f_1]$ and~$\Gamma[f_2, \mathbf P_0 f_1]$ can be handled in a similar way, we thus deduce
\begin{equation}\label{nonlinear_low}
\begin{aligned}
& \| \Psi^\eps[f_1,f_2] \|_{\widetilde L^\infty_I H^{\frac12}_x L^2_v} 
+ \frac{1}{\sqrt\eps} \| \mathbf P_0^\perp \Psi^{\eps}[f_1,f_2] \|_{L^2_I H^\frac32_x H^{s,*}_v}
+ \| \mathbf P_0 \Psi^{\eps}[f_1,f_2] \|_{L^2_I H^\frac32_x H^{s,*}_v} \\
&\quad
\lesssim \| f_1 \|_{\mathcal X^\eps_I} \| f_2 \|_{\mathcal X^\eps_I}
\end{aligned}
\end{equation}
since $\beta<1/2$.

\subsubsection{Conclusion}

The bounds obtained in~\eqref{nonlinear_high} and~\eqref{nonlinear_low} yield the continuity estimate given in Proposition~\ref{estimates on delta eps}--(5).

 \appendix
 \section{Hypocoercivity}\label{app:hypocoercivity}

It is well-known, see for instance \cite{MR1463805,MR1946444,MR2322149,MR2231011,MR2254617}, that the linearized Boltzmann and Landau collision operators satisfy the following coercive-type inequality
\begin{equation}
\la L f , f \ra_{L^2_v} \le - \lambda_2 \|  \P_0^\perp f \|_{H^{s,*}_v}^2\, ,
\end{equation}
for some $\lambda_2>0$.

For all $\eps \in (0,1]$ and all $ k \in \Z^3$, we recall that $\widehat \Lambda^\eps ( k)$ is the Fourier transform in space of the full linearized operator~$\displaystyle  \frac{1}{\eps^2} L - \frac{1}{\eps} v \cdot \nabla_x $ , namely 
\begin{equation}\label{eq:def:Lambdaeps_xi}
\widehat\Lambda^\eps ( k) :=  \frac{1}{\eps^2} ( L - i \eps v \cdot  k )\,.
\end{equation}

We now state a hypocoercive result for $\widehat\Lambda^\eps ( k)$ (for a detailed presentation of the subject, we refer to~\cite{BCMT} and the references therein, we also point out the papers~\cite{Strain}  and~\cite{Duan} in the case of the whole space), as presented in \cite{CC,Carrapatoso-Gervais}.
 \begin{prop}\label{prop:hypocoercivity}
 There is an  inner product $\la\!\la \cdot , \cdot \ra\!\ra_{L^2_v}$ on $L^2_v$ (depending on $ k$)  such that the associate norm~$	\Nt \cdot \Nt_{L^2_v}$
 is equivalent to the standard norm $\| \cdot \|_{L^2_v}$ on $L^2_v$ with bounds that are independent of $ k$ and $\eps$, and there exists $\lambda_3>0$ such that for every $f$ satisfiying~\eqref{eq:normalization} and all $ k \in  \Z^3$, there holds
	$$
	\re \la\!\la\widehat \Lambda^\eps ( k) \widehat f ( k) , \widehat f ( k) \ra\!\ra_{L^2_v } \le - \lambda_3 \left( \frac{1}{\eps^2} \|  \P_0^\perp \widehat 	f ( k) \|_{H^{s,*}_v}^2 + \| \P_0 \widehat f ( k) \|_{L^2_v}^2 \right)\,.
	$$
\end{prop}
\begin{proof}
For every $ k \in \Z^3$, we define
	\begin{equation*}
	\begin{aligned}
		\psi[ f_1 ,  f_2 ] ( k)
		&:=  \frac{\delta_1 i  }{\la  k\ra^2}   k \theta[\widehat f_1( k)]  \cdot M[  \P_0^\perp  \widehat f_2( k)]  
		+  \frac{\delta_1 i }{\la  k\ra^2}    k \theta[\widehat f_2( k)] \cdot M[  \P_0^\perp  \widehat f_1 ( k)] \\
		&\quad 
		+  \frac{\delta_2i}{\la  k\ra^2} 
		( k \otimes u[\widehat f_1( k)] )^{\mathrm{sym}} : \left\{ \Theta[ \P_0^\perp  \widehat f_2( k)] + \theta[\widehat g( k)] \operatorname{Id} \right\}\\
		&\quad 
		+  \frac{\delta_2i}{\la  k\ra^2} 
		( k \otimes u[\widehat f_2( k)] )^{\mathrm{sym}} : \left\{\Theta[ \P_0^\perp  \widehat f_1( k)] + \theta[\widehat f( k)] \operatorname{Id} \right\}\\
		&\quad 
		+  \frac{\delta_3i }{\la  k\ra^2}  k \rho[\widehat f_1( k)] \cdot  u[\widehat f_2( k)] 
		+  \frac{\delta_3i }{\la  k\ra^2}  k \rho[\widehat f_2( k)] \cdot  u[\widehat f_1( k)] \,,
	\end{aligned}
	\end{equation*}
with constants $0 < \delta_3 \ll \delta_2 \ll \delta_1 \ll 1$, where $\operatorname{Id}$ is the $3 \times 3$ identity matrix and the moments $M$ and $\Theta$ are defined by
	$$
	M [f] := \int_{\R^3} f v \,(|v|^2-5) \mu^{\frac12}(v) \, \d v\,, \qquad 
	\Theta[f] := \int_{\R^3} f \left(v \otimes v - \operatorname{Id}\right) \mu^{\frac12}(v) \, \d v\,,
	$$
and where for vectors $a, b \in \R^3$ and matrices $A, B \in \R^{3 \times 3}$, we denote
	$$
	(a \otimes b)^{\mathrm{sym}} = \frac{1}{2} (a_j b_k + a_kb_j)_{1 \le j,k \le 3}\,, \qquad 
	A : B = \sum_{j,k=1}^3 A_{jk} B_{jk}\,.
	$$
We then define the inner product $\la\!\la \cdot , \cdot \ra\!\ra_{L^2_v}$ on $L^2_v$ (depending on $ k$) by
	\begin{equation}\label{eq:new_inner_product}
	\la\!\la \widehat f_1( k), \widehat f_2 ( k) \ra\!\ra_{L^2_v} 
	:=  \la \widehat f_1( k), \widehat f_2 ( k) \ra_{L^2_v} + \eps \, \psi[f_1, f_2]( k) \,,
	\end{equation}
and the associated norm
	\begin{equation}\label{eq:new_norm}
	\Nt \widehat f( k) \Nt_{L^2_v}^2 :=  \la\!\la \widehat f( k), \widehat f ( k) \ra\!\ra_{L^2_v} \,.
	\end{equation}
We then argue  as in \cite{Strain}, the only difference being the factor $\eps$ in the second term of~\eqref{eq:new_inner_product}.
\end{proof}

Using this hypocoercivity result, we are able to prove Proposition \ref{prop:estimate_Ueps}.

\begin{proof}[Proof of Proposition~{\rm\ref{prop:estimate_Ueps}}] $ $

$\bullet$ $ $ (1) Let $f(t) := U^\eps(t) f_{\rm in}$ for all $t \ge 0$, which satisfies the equation
	\begin{equation}\label{eq:Ueps_f0}
	\partial_t f = \frac{1}{\eps^2} (L - \eps v \cdot \nabla_x) f\,, \quad 
	f_{| t=0} = f_{\rm in}\,.
	\end{equation}
We already observe that $f(t)$ verifies \eqref{eq:normalization} thanks to the conservation properties of $\Gamma$ (and hence of $L$).
Taking the Fourier transform in space of the above equation, we obtain that~$\widehat f$ satisfies
	\begin{equation}\label{eq:hatUeps_hatf0}
	\partial_t \widehat f( k) =\widehat  \Lambda^\eps ( k) \widehat f( k)\,, \quad 
	\widehat f( k)_{|t=0} = \widehat f_{\rm in}( k)\,,
	\end{equation}
for all $ k \in \Z^3$.
Applying Proposition~\ref{prop:hypocoercivity} yields, for all $t \ge 0$,
	\begin{align*}	
	\frac{1}{2} \frac{\d}{\d t} \Nt \widehat f( k) \Nt_{L^2_v }^2
	&= \re \la\!\la \widehat\Lambda^\eps ( k) \widehat f ( k) , \widehat f ( k) \ra\!\ra_{L^2_v } \\
	&\le - \lambda_3 \left( \frac{1}{\eps^2} \| \P_0^\perp \widehat f ( k) \|_{H^{s,*}_v}^2 
	+ \| \P_0 \widehat f ( k) \|_{L^2_v}^2 \right)\,,
	\end{align*}
which implies	
	\begin{align*}
	\| \widehat f(t, k) \|_{L^2_v}^2
	+\frac{1}{\eps^2}\int_0^t \|  \P_0^\perp \widehat f(t', k) \|_{H^{s,*}_v}^2 \, \d t'
	+\int_0^t   \| \P_0 \widehat f(t', k) \|_{L^2_v}^2 \, \d t' 
	\lesssim \| \widehat f_{\rm in} ( k)\|_{L^2_v}^2\,,
	\end{align*}
where we have used that $\Nt \cdot \Nt_{L^2_v}$ is equivalent to $\| \cdot \|_{L^2_v}$ independently of $ k$ and $\eps$.
Taking the supremum in time and then multipliyng by $\la  k \ra^{2m}$ yields
	\begin{equation*}
	\la  k \ra^{2m} \| \widehat f( k) \|_{L^\infty_t L^2_v}^2
	+\frac{\la  k \ra^{2m}}{\eps}\|  \P_0^\perp \widehat f( k) \|_{L^2_t H^{s,*}_v}^2 
	+ \la  k \ra^{2m}  \| \P_0 \widehat f( k) \|_{L^2_t L^2_v}^2	
	\lesssim \la  k \ra^{2m}\| \widehat f_{\rm in} ( k) \|_{L^2_v}^2\,.
	\end{equation*}
We conclude by summing in $ k$.

\medskip

$\bullet$ $ $ (2) Denote 
	$$
	h(t) := \int_0^t U^\eps(t-t') S(t') \, \d t'
	$$
which is the solution to
	\begin{equation}\label{eq:gS}
	\partial_t  h = \frac{1}{\eps^2} ( L - \eps v \cdot \nabla_x )h +  S\,, \quad  h_{|t=0} = 0\,.
	\end{equation}
Taking the Fourier transform in space gives
	\begin{equation}\label{eq:hatgS}
	\partial_t \widehat h( k) =\widehat \Lambda^\eps ( k) \widehat h( k) + \widehat S( k), \quad \widehat h( k)_{|t=0} = 0\,,
	\end{equation}
for all $ k \in \Z^3$.
From the definition of \eqref{eq:new_inner_product} and the hypothesis $\P_0 S = 0$, we observe that
	\begin{align*}
	\la\!\la \widehat S ( k) ,  \widehat h ( k) \ra\!\ra_{L^2_v}
	&= \la \widehat S ( k) ,   \widehat h ( k) \ra_{L^2_v} 
	+ \eps  \Psi [S , h ]( k) \\
	&= \la \widehat S ( k) ,   \P_0^\perp \widehat h ( k) \ra_{L^2_v} 
	+ \eps  \frac{\delta_1 i }{1+| k|^2}    k \theta[\widehat h( k)] \cdot M[  \P_0^\perp \widehat S ( k)] \\
	&\quad
	+  \eps \frac{\delta_2i}{1+| k|^2} ( k \otimes u[\widehat h( k)] )^{\mathrm{sym}} : \Theta[ \P_0^\perp \widehat S( k)] \, .
	\end{align*}
Observing that for any polynomial $p=p(v)$ we have
	$$
	\left|\int_{\R^3} \widehat S( k) p(v) \mu^{\frac12}(v) \, \d v  \right|
	\lesssim \| \widehat S( k) \|_{(H^{s,*}_v)'}\,,
	$$
we get
	$$
	|\Psi [S , h ]( k)|
	\lesssim \|   \P_0^\perp \widehat S( k) \|_{(H^{s,*}_v)'} \frac{| k|}{\la  k \ra}\| \P_0 \widehat h( k) \|_{L^2_v}\,.
	$$
By duality, we also have
	$$
	\la \widehat S ( k) ,   \P_0^\perp \widehat h ( k) \ra_{L^2_v} 
	\lesssim \| \widehat S( k) \|_{(H^{s,*}_v)'}  \|  \P_0^\perp \widehat h ( k) \|_{H^{s,*}_v} \,,
	$$
therefore gathering previous estimates yields
	\begin{equation}\label{eq:bound_S}
	\begin{aligned}
	\la\!\la \widehat S ( k) ,  \widehat h ( k) \ra\!\ra_{L^2_v}
	&\lesssim \eps \| \widehat S( k) \|_{(H^{s,*}_v)'} 
	\left( \frac{1}{\eps} \|  \P_0^\perp \widehat h ( k)\|_{H^{s,*}_v} + \| \P_0 \widehat h ( k)\|_{L^2_v} \right)\,.
	\end{aligned}
	\end{equation}
Using Proposition~\ref{prop:hypocoercivity} and arguing as in the proof of Proposition~\ref{prop:estimate_Ueps}--(1) we have, for all~$t \ge 0$ and all~$ k \in \Z^3$,
	\begin{equation}\label{eq:dt_hatgS}
	\begin{aligned}
	\frac{1}{2} \frac{\d}{\d t} \Nt \widehat h( k) \Nt_{L^2_v}^2
	&\le - \lambda_3   \left( \frac{1}{\eps^2} \|  \P_0^\perp \widehat h ( k) \|_{H^{s,*}_v}^2 
	+  \| \P_0 \widehat h ( k) \|_{L^2_v}^2 \right) \\
	&\quad
	+ \eps C \|  \widehat S( k) \|_{(H^{s,*}_v)'} \left( \frac{1}{\eps} \|  \P_0^\perp \widehat h ( k) \|_{H^{s,*}_v} 
	+  \| \P_0 \widehat h ( k) \|_{L^2_v} \right)\\
	&\le - \frac{\lambda_3}{2}  \left( \frac{1}{\eps^2} \|  \P_0^\perp \widehat h ( k) \|_{H^{s,*}_v}^2 
	+  \| \P_0 \widehat h ( k) \|_{L^2_v}^2 \right)  
	+ C \eps^2 \|  \widehat S( k) \|_{(H^{s,*}_v)'}^2 \,,
	\end{aligned}
	\end{equation}
where we have used Young's inequality in last line. This implies
	\begin{multline*}
	\| \widehat h(t, k) \|_{L^2_v}^2
	+\frac{1}{\eps^2}\int_0^t \|  \P_0^\perp \widehat h(t', k) \|_{H^{s,*}_v}^2 \, \d t'
	+\int_0^t  \| \P_0 \widehat h(t', k) \|_{L^2_v}^2 \, \d t' \\
	\lesssim \eps^2 \int_0^t \| \widehat S(t', k) \|_{(H^{s,*}_v)'}^2 \, \d t' \,.
	\end{multline*}
Taking the supremum in time and then multiplying by $\la  k \ra^{2m}$ yields
	\begin{multline*}
	\la  k \ra^{2m} \|  \widehat h( k) \|_{L^\infty_t L^2_v}^2
	+\frac{\la  k \ra^{2m}}{\eps^2}\|   \P_0^\perp \widehat h( k) \|_{L^2_t H^{s,*}_v}^2
	+\la  k \ra^{2m}  \| \P_0 \widehat h( k)\|_{L^2_t L^2_v}^2\\
	\lesssim \eps^2 \la  k \ra^{2m} \|   \widehat S( k) \|_{L^2_t (H^{s,*}_v)'}^2\,,
	\end{multline*}
and we conclude by summing in $ k$.  Proposition~\ref{prop:estimate_Ueps} is proved.
\end{proof}

\bigskip

\end{document}